%% file: winding_final_arxiv_combined_251118.tex
\documentclass[aop]{imsart}

\usepackage{enumerate}

\usepackage{amsmath,amsfonts,amssymb,graphics,amsthm, comment}
\usepackage{hyperref}
\hypersetup{
    colorlinks=true,
    linkcolor=blue,
    citecolor=red,
    urlcolor=blue,
    pdfborder={0 0 0}
}

\usepackage{graphicx}
\usepackage{xcolor}
\usepackage{bbm}
\usepackage{wrapfig}

\usepackage[font=sf, labelfont={sf,bf}, margin=1cm]{caption}

\usepackage{cleveref}
  \crefname{theorem}{Theorem}{Theorems}
  \crefname{thm}{Theorem}{Theorems}
  \crefname{lemma}{Lemma}{Lemmas}
  \crefname{lem}{Lemma}{Lemmas}
  \crefname{remark}{Remark}{Remarks}
  \crefname{prop}{Proposition}{Propositions}
\crefname{notation}{Notation}{Notations}
\crefname{claim}{Claim}{Claims}
  \crefname{defn}{Definition}{Definitions}
  \crefname{corollary}{Corollary}{Corollaries}
  \crefname{section}{Section}{Sections}
  \crefname{figure}{Figure}{Figures}
    \crefname{assumption}{Assumption}{Assumptions}

\newtheorem{thm}{Theorem}[section]

\newtheorem{lemma}[thm]{Lemma}
\newtheorem{corollary}[thm]{Corollary}
\newtheorem{prop}[thm]{Proposition}

\newtheorem{assumption}[thm]{Assumption}
\numberwithin{equation}{section}

\theoremstyle{definition}
\newtheorem{remark}[thm]{Remark}

\def \min {\text{min}}

\def\cY{\mathcal{Y}}

\def\cV{\mathcal{V}}

\def\cT{\mathcal{T}}
\def\cS{\mathcal{S}}
\def\cR{\mathcal{R}}
\def\cQ{\mathcal{Q}}
\def\cP{\mathcal{P}}

\def\cI{\mathcal{I}}
\def\cH{\mathcal{H}}
\def\cG{\mathcal{G}}

\def\cE{\mathcal{E}}
\def\cD{\mathcal{D}}
\def\cC{\mathcal{C}}
\def\cB{\mathcal{B}}
\def\cA{\mathcal{A}}
\def \ve {\varepsilon}

\def \Gd {G^{\#\delta}}

\def\P{\mathbb{P}}
\def\E{\mathbb{E}}
\def\C{\mathbb{C}}
\def\R{\mathbb{R}}
\def\Z{\mathbb{Z}}
\def\N{\mathbb{N}}
\def\D{\mathbb{D}}

\def\H{\mathbb{H}}

\def  \p- {p\textunderscore}

\def\eps{\varepsilon}

\def\ph{\varphi}

\DeclareMathOperator{\gff}{GFF}
\def \hg {h_{ \gff }}
\def \d {{\# \delta}}

\DeclareMathOperator{\var}{Var}
\DeclareMathOperator{\Arg}{Arg}
\DeclareMathOperator{\dist}{dist}

\DeclareMathOperator{\Leb}{Leb}

\DeclareMathOperator{\diam}{Diam}
\DeclareMathOperator{\harm}{Harm}

\newcommand{\note}[1]{{\color{red}{[note: #1]}}}
\newcommand{\notet}[1]{{\color{red}{[#1]}}}
\newcommand{\noteb}[1]{{\color{blue}{[#1]}}}

\newcommand{\abs}[1]{ \lvert #1 \rvert}
\newcommand{\norm}[1]{ \lVert #1 \rVert}

\def\mn{\medskip \noindent}

\begin{document}

\begin{frontmatter}

\title{Dimers and imaginary geometry}
\runtitle{Dimers and imaginary geometry}

\begin{aug}
 \author{\fnms{Nathana\"el}  \snm{Berestycki}\thanksref{t1}\ead[label=e1]{N.Berestycki@statslab.cam.ac.uk}}
 ,
  \author{\fnms{Beno\^{i}t} \snm{Laslier}\thanksref{t2} \ead[label=e2]{laslier@math.univ-paris-diderot.fr}}
  \and
  \author{\fnms{Gourab} \snm{Ray}\thanksref{t3} \ead[label=e3]{gourab1987@gmail.com}}

  \thankstext{t1}{Supported in part by EPSRC grants EP/L018896/1 and EP/I03372X/1. On leave from the University of Cambridge.}

    \thankstext{t2}{Supported in part by EPSRC grant EP/I03372X/1}

    \thankstext{t3}{Supported in part by EPSRC grant EP/I03372X/1, NSERC 50311-57400 and University of Victoria start-up 10000-27458}

  \runauthor{Berestycki, Laslier, Ray}

  \affiliation{Universit\"at Wien, Universit\'e Paris Diderot and University of Victoria}

\end{aug}

\begin{abstract}
We show that the winding of the branches in a uniform spanning tree on a planar graph converge in the limit of fine mesh size to a Gaussian free field. The result holds assuming only convergence of simple random walk to Brownian motion and a Russo--Seymour--Welsh type crossing estimate, thereby establishing a strong form of universality. 
As an application, we prove universality of the fluctuations of the height function associated to the dimer model, in several situations.

 The proof relies on a connection to imaginary geometry, where the scaling limit of a uniform spanning tree is viewed as a set of flow lines associated to a Gaussian free field. In particular, we obtain an explicit construction of the a.s. unique Gaussian free field coupled to a continuum uniform spanning tree in this way, which is of independent interest.

\end{abstract}

\end{frontmatter}

\renewcommand{\i}{\text{\textnormal{int}}}



\section{Introduction}\label{sec:introduction}

\subsection{Main results}


Let $G$ be a finite bipartite planar graph. A dimer covering of $G$ is a set
of edges such that each vertex is incident to exactly one edge; in other words
it is a perfect edge-matching of its vertices. The \textbf{dimer model} on $G$
is simply a uniformly chosen dimer covering of $G$. It is a classical model of
statistical physics, going back to work of Kasteleyn \cite{Kasteleyn} and Temperley--Fisher \cite{TemperleyFisher}
who computed its partition function. It is the subject of an extensive physical
and mathematical literature; we refer the reader to \cite{KenyonSurvey} for a
relatively recent discussion of some of the most important progress.
A key feature of this model is its ``exact solvability" which comes from its determinantal structure (\cite{Kasteleyn}) and brings in tools from subjects such as discrete complex analysis, algebraic combinatorics and algebraic geometry.
This is one reason the study of this model has been so successful.

An important tool for the dimer model is a notion of height function introduced by Thurston \cite{Thurston} which turns a dimer configuration into a discrete random surface in $\R^3$ (i.e., a random function indexed by the faces of $G$ with values in $\R$). Therefore a key question concerns the large-scale behaviour of this height function. It is widely believed that in the planar case and under very general assumptions, the fluctuations of the height function are described by (a variant of) the \textbf{Gaussian free field}.

\begin{figure}
\begin{center}
\includegraphics[height=5.5cm]{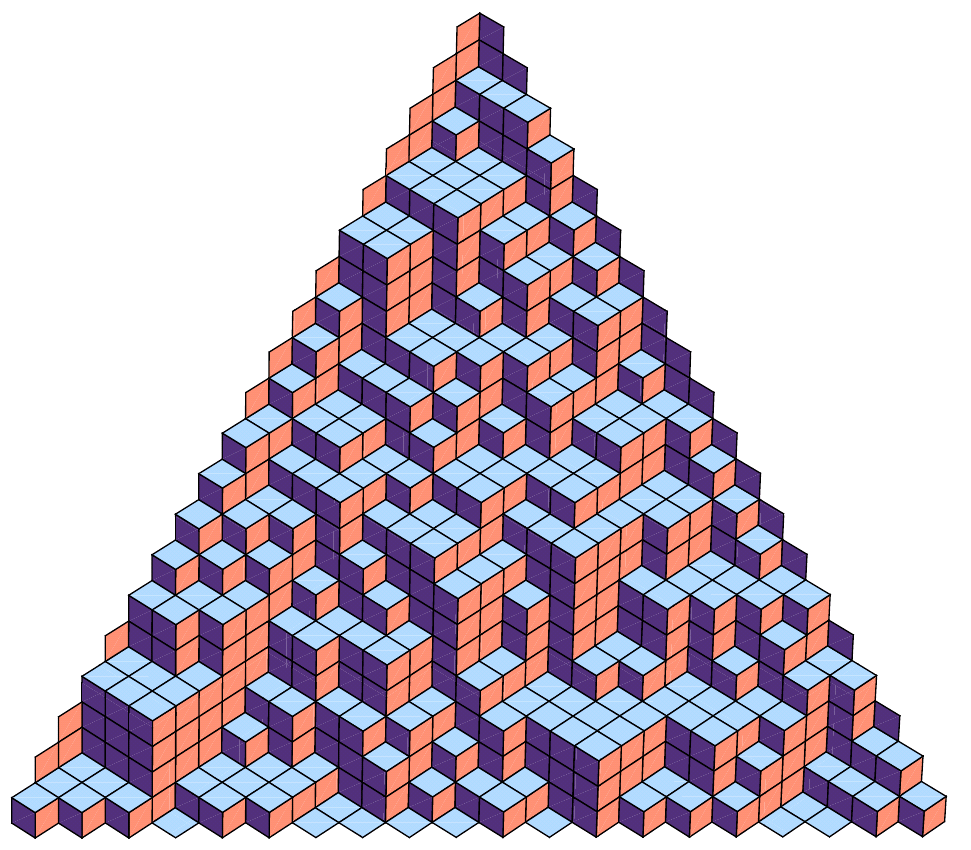}
\caption{Height function of a dimer model (or lozenge tiling) with planar boundary conditions on a triangle. Picture by R. Kenyon. \label{F:simulation}}
\end{center}
\end{figure}

In this paper we present a robust approach to proving
such results. We now state an example of application of this technique.
Consider lozenge tiling of the plane by lozenges with angles $\pi/3$ and $2\pi/3$ and sidelength $\delta$, which can be seen equivalently as dimer configurations on the hexagonal lattice of mesh size $\delta$, or stack of cubes in $\R^3$ of size $\delta$. As usual we describe a tiling by its height function $h^\d$ which we can take to be the $z$ coordinate in the stack of cubes at each point of the tiling. Given a bounded domain that can be tiled with lozenges, we define the \textbf{boundary height} to be the curve in $\R^3$ obtained by considering the height function along the outermost lozenges (which does not depend on the tiling configuration). Set $\chi = 1/\sqrt{2}$ (this is the parameter in imaginary geometry associated with $\kappa = 2$, see \cite{IG1,IG4}).

\begin{thm}\label{T:plane}
Let $P$ be a plane in $\R^3$ whose normal vector has positive coordinates, and let $D \subset \R^2$ be a simply connected bounded domain with locally connected boundary. Then there exists a sequence of domains $U^\d \subset \R^2$, which can be tiled by lozenges of size $\delta$, with the following properties.
The boundary height of $\partial U^\d$ stays at distance $o(1)$ of $P$, $\partial U^\d$ converges to $\partial D$ in Hausdorff sense, and
$$
\frac{h^\d - \E(h^\d) }{\delta} \xrightarrow[\delta \to 0]{}  \frac1{2\pi\chi} \hg^0 \circ \ell,
$$
in distribution where $\ell $ is an explicit linear map determined by $P$ and $\hg^0$ is
a Gaussian free field with Dirichlet boundary conditions in $\ell(D)$.
\end{thm}
Note that convergence holds in distribution on the Sobolev space $H^{-1 - \eta}(D)$ for all $\eta >0$, once $h^\d$ has been extended to a continuous function on $D$ (essentially by interpolation), see \cref{S:main_winding_full} for more details.  We emphasise here that in Theorem \ref{T:plane} above we only prove the existence of a sequence of domains $U^\d$ such that the result holds. See Section 4.2 of \cite{BLRannex} for details of the construction of $U^\d$.

{Theorem \ref{T:plane} is the consequence of a more general theorem (Theorem \ref{T:winding_intro}) which will be the focus of this article. The connection between these two theorems is explained in \cite{BLRannex} and exploits a relation between the dimer model and the uniform spanning tree model on a modified graph called the T-graph introduced in \cite{dimer_tree}. More precisely this connection is a generalisation of Temperley's celebrated bijection, which equates the height function of a dimer configuration to the winding of branches in an associated uniform spanning tree.
}

\begin{figure}
\centering
\includegraphics[width=.65 \textwidth]{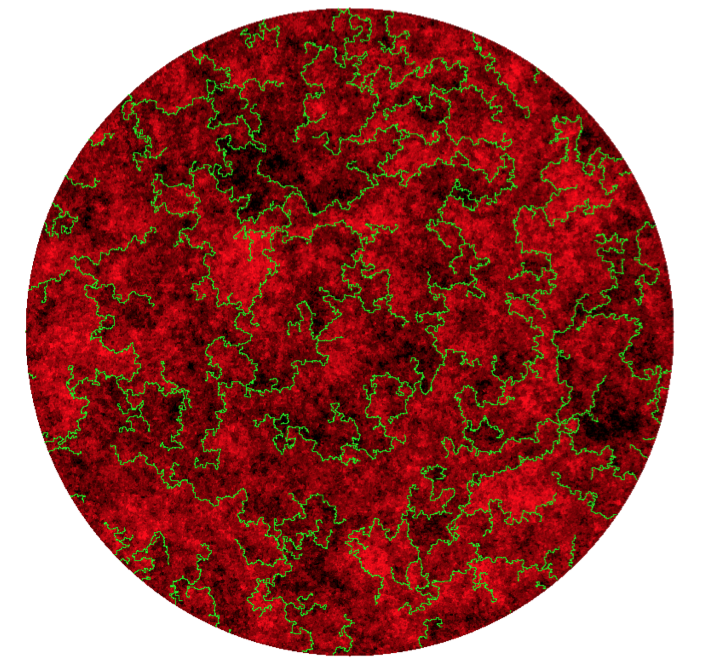}
\caption{
A Uniform spanning tree and its winding field.}
\end{figure}

\medskip We now state the general theorem which concerns the winding of branches in a uniform spanning tree.
Let
$G^\d$ be a sequence of planar (possibly directed) graphs properly embedded in the plane. We assume
that $G^\d$ satisfies some natural conditions (stated precisely in \cref{sec:assumption}). In particular, the two main assumptions are: (i)
 simple random walk on $G^\d$ converges to a Brownian motion as $\delta \to 0$, and (ii) a Russo--Seymour--Welsh type
crossing condition, namely, simple random walk can cross any rectangle of fixed aspect ratio and size at least $\delta$, with a probability uniformly positive over the position, orientation and scale of the rectangle.

Let $D \subset \C$ be a
bounded domain with locally connected boundary. 
Let $D^\d$ be the
graph induced by the vertices of $G^\d$ in $D$ with boundary $\partial D^\d$ (the precise description is in \cref{sec:discrete_wind_def}). 
{Recall also that a wired uniform spanning tree is simply the uniform spanning tree on the graph obtained from $D^\d$ by identifying all the boundary vertices of $D^\d$.
For more details on this topic, see Section \ref{sec:Wilson_supp} as well as \cite{LP:book}} for (much) more background.


\begin{thm}\label{T:winding_intro}
Let $\cT^\d$ be a wired uniform spanning tree on $D^\d$, and for any $v \in D^\d$ let $h^\d(v)$ denote
the  winding of the branch of $\cT^\d$ connecting $v$ and $\partial D^\d$. Then
$$
h^\d - \E( h^\d) \xrightarrow[\delta \to 0]{} \frac 1 \chi \hg^0
$$
in the sense of distributions, where $\hg^0$ is a Gaussian free field with Dirichlet boundary conditions in $D$.
\end{thm}

 By \emph{winding} in \cref{T:winding_intro}, we mean the intrinsic winding, i.e., the sum of the turning angles along the path. See \cref{E:windingsmooth} for precise definition.
 Note that the scaling is somewhat different from \cref{T:plane} (there is no renormalisation here) because in that theorem we measure the height defined by lozenges of diameter $O(\delta)$ whereas here we measure the winding (unnormalised) along paths in the tree.

 A more precise form of \cref{T:winding_intro} is stated later on in \cref{thm:main_details}.
 Furthermore, in \cref{thm:jt_conv} we prove a stronger version of this theorem: we obtain the joint convergence of the winding function \emph{and} spanning tree to a pair  (GFF, continuum spanning tree) which are coupled together according to the imaginary geometry coupling. The connection to the theory of \textbf{imaginary geometry}, initiated in \cite{Dub1} and further developed in a sequence of papers of which \cite{IG1} and \cite{IG4} will be the most relevant here, will in fact play a crucial role in this work.
Very informally, imaginary geometry provides a coupling between a Gaussian free field and an SLE curve so that the ``pointwise values" of the field along the curve are given by the ``intrinsic winding" of the SLE curve. Hence this coupling can be viewed as a continuum analogue of Temperley's bijection, an observation already alluded to in \cite{Dub1}. In particular, our approach provides an explicit construction of the a.s. unique Gaussian free field associated to a continuum uniform spanning tree which may be of independent interest: see \cref{T:winding_continuum} for a statement and the discussion immediately below.

\cref{T:winding_intro} may be applied to various other dimer models to show Gaussian free field fluctuations. We give a brief overview of such examples:

\begin{itemize}
\item generalised Temperleyan domains, as described in \cite{KPWtemperley}, on graphs which satisfy the assumptions of \cref{sec:assumption},

\item dimers on double isoradial graphs with uniformly elliptic angles. This recovers and in fact significantly strengthens a result of Li \cite{Li} as her work requires the discrete boundary of the domain to contain a macroscopic straight line. (Note that the assumptions in Theorem \ref{T:winding_intro} are satisfied in this this case by results of Chelkak and Smirnov \cite{ChelkakSmirnov}: for instance, the crossing assumption is an easy consequence of Theorem 3.10 in \cite{ChelkakSmirnov}.)



\item dimers in random environment: e.g., on $\mathbb{Z}^2$ with random i.i.d.\ weights on the even edges of $\mathbb{Z}^2$, (in which case the law of the dimers is simply proportional to the product of the weights). We restrict the randomness of the weights to the even sublattice in order to apply the Temperley bijection, and assume for instance the weights to be balanced and uniformly elliptic.

\item dimers with a defect line: suppose the weight of all the edges in a horizontal line of $\mathbb{Z}^2$ is changed from 1 to $z>0$.

\end{itemize}

\subsection{Discussion of the results}

\paragraph{Mean height in dimer models and spanning trees}
\cref{T:plane} describes the limiting distribution of $h^\d - \E( h^\d)$ and the reader might be interested to know what can be said about the mean itself, $\E( h^\d)$. First, we point out that on the law of large number scale, the mean height of the lozenge tiling is known by a result of Cohn, Kenyon and Propp \cite{CohnKenyonPropp} to converge to a deterministic function which here is simply an affine function (due to our assumptions about the boundary values of the height function).

Our approach yields further information about $\E(h^\d)$. In the spanning tree setting,
if $h^\d$ is the winding of branches in a uniform spanning tree (as in the setup of \cref{T:winding_intro}) from a fixed marked point $x$ on the boundary then we obtain
\[
\E(h^\d) = m^\d + u_{D,x} + \frac{\pi}{2} + o(1)
\]
where $u_{D,x}$ is the harmonic extension of the anticlockwise winding from $x$ (see \cref{E:uDx} for a precise definition) and $m^\d$ depends only on the graph and the vertex $v$ at which we are computing the winding (but interestingly not the domain in which the spanning tree/dimer configuration is being sampled). Note that a consequence of the above mentioned result of Cohn--Kenyon--Propp \cite{CohnKenyonPropp} is that $m^\d = o(1/\delta)$ uniformly over the graph; in fact much better bounds can be derived.

For many ``reasonable" graphs we suspect that $m^\d$ actually converges to 0, as it is essentially the expected winding of a path converging to a full-plane SLE$_2$.
Nevertheless some assumptions are clearly needed, as the fact that random walk converges to Brownian motion alone is \emph{not} enough to give control on the mean winding in a UST. For an example, take the usual square grid and add a spiral path at every vertex.
This example shows that it is only the \textbf{fluctuations} which may be hoped to be universal, while the mean itself will usually depend on the microscopic details of the graph.

\paragraph{Relation to earlier results on fluctuations of dimer models}

The study of fluctuations in dimer models has a long and distinguished history, which it is not the purpose of this paper to recall, see \cite{KenyonSurvey} for references. However, we mention a few highlights.
 In \cite{Kenyon_ci,KenyonGFF}, Kenyon showed that the height function on the square lattice for Temperleyan domains (for which the boundary conditions are planar of slope 0) converge to a multiple of the Gaussian free field with Dirichlet boundary conditions. The study of dimers on graphs more general than the square or hexagonal lattices was initiated in \cite{KenyonOkounkovSheffield} where they consider tilings on arbitrary periodic bipartite planar graphs. The non-periodic case was first mentioned in \cite{Kenyon_iso}, also in the whole plane setting. Convergence to the full plane Gaussian free field on isoradial periodic bipartite graphs (including ergodic lozenge tilings of arbitrary slope), as well as on Temperleyan superpositions of isoradial (not necessarily periodic) graphs, is a consequence of a remarkable work by De Tili\`{e}re \cite{Tiliere07}.

The interest in the role of boundary conditions was sparked by the observation of the arctic circle phenomenon: for some domains, in the limit the dimer configuration outside of some region (the liquid or temperate region) is deterministic (also called frozen). This was first identified in the case of the aztec diamond by Jokusch, Propp and Shor \cite{Arctic} (see also the more recent paper \cite{Romik} by Romik for a different approach and fascinating connections to alternating sign matrices). The case of general boundary conditions for the hexagonal lattice was solved later by Cohn, Kenyon and Propp \cite{CohnKenyonPropp} who obtained a variational problem determining the law of large numbers behaviour for the height function. This variational principle was studied by Kenyon and Okounkov in \cite{KenyonOkounkov} who discovered that in polygonal domains the boundary between the frozen and liquid regions are always explicit algebraic curves.
 In this direction we also point out the recent paper by Petrov \cite{Petrov} and by Bufetov--Gorin \cite{BufetovGorin} who obtained convergence of the height function fluctuations to the GFF in liquid regions for some polygonal domains.

A paper by Kenyon \cite{KenyonHex} discusses the question of fluctuations, with the goal of proving convergence of the centered height function to a (deformation of) the Gaussian free field in the liquid region. Unfortunately, the crucial argument in his proof, Lemma 3.6, is incomplete and at this point it is unclear how to fix it\footnote{We thank Fabio Toninelli and Rick Kenyon for helpful discussions regarding this lemma.}. The issue is the following. The central limit theorem proved in \cite{LaslierCLT} provides an information about convergence of discrete harmonic functions to continuous harmonic functions. However what is needed in \cite{KenyonHex} is an estimate on the discrete derivative of such functions (i.e., the entries of the inverse Kasteleyn matrix) as well as a control on the speed of convergence so that the errors can be summed when integrating along paths. (There is a more general question here, which is to better understand the links between discrete and continuous harmonic functions on quasi-periodic graphs.) Our work can be seen as a way to get around these issues but more importantly provides a unified and robust approach to the convergence of fluctuations.

Finally let us mention that all the above works on fluctuations rely on writing an exact determinantal formula for the correlations between dimers. The main body of work is then to find the asymptotic of the entries of these determinants using either exact combinatorics or discrete complex analytic techniques. Our approach is completely orthogonal, relying on properties of the limiting objects in the continuum rather than exact computations at the microscopic level. This is one reason why the results we obtain are valid under less restrictive conditions on the regularity of the boundary (while such assumptions are typically needed for the tools of discrete complex analysis). In particular, we do not assume the domain to be Jordan or smooth, only to have a locally connected boundary. This is the condition required so that the conformal map from the unit disc to the domain extends to the boundary (Theorem 2.1 in \cite{Pommerenke}). It is plausible that even this mild condition can be relaxed by appealing to a suitable notion of conformal boundary (e.g., prime ends, see Section 2.4 in \cite{Pommerenke}) but we did not pursue this here in an attempt to keep the paper at a reasonable length.

\subsection{A conjecture}

\cref{T:winding_continuum} provides a continuum analogue of Theorem \ref{T:winding_intro}, in the sense that the continuum field is regularised by truncating the SLE branches rather than discretisation.
As already mentioned this is of independent interest since it gives an explicit construction of the GFF coupled to a uniform spanning tree according to imaginary geometry. We strongly believe that the same result holds for other values of $\kappa$. Our proof of \cref{T:winding_continuum} is written in a way that is mostly independent of the value of $\kappa$ except for a few lemmas, gathered in \cref{S:SLE}. These lemmas concern fairly basic properties of flow lines which seem very plausible for arbitrary values of $\kappa$. However we did not try to establish them, preferring to focus on the case $\kappa =2$ only since we also need the analogous discrete statements later on in the paper.


The above discussion suggests a number of results concerning interacting dimers recently introduced by Giuliani, Mastropietro and Toninelli \cite{InteractingDimers}. We conjecture that if one applies Temperley's bijection to a configuration of interacting dimers as in \cite{InteractingDimers}, the Peano curve of the resulting tree converges to certain space-filling SLE$_{\kappa'}$ defined by Miller and Sheffield \cite{IG4} in these cases and that by adjusting the interaction parameter one can at least obtain any $\kappa' \in (8-\eps, 8+\eps)$. However it is quite speculative at the moment as we lack tools (like Wilson's algorithm) to study interacting dimers or corresponding Temperleyan spanning trees. See \cite{giuliani2017haldane} (which appeared after a draft of our paper was first put on arxiv) for additional support for our conjecture, and see \cite{SixVertexSLE} for a related question.

\subsection{Overview of the proof}
\label{S:overview}

For the convenience of the reader, we summarise briefly the main steps of the proof of \cref{T:winding_intro}.

\textbf{Step 1}. We first formulate in \cref{T:winding_continuum} a continuous analogue of this theorem, where we study the winding of truncated branches in a continuum wired Uniform Spanning Tree. Branches of this tree are SLE$_2$ curves, and therefore a \textbf{key idea} is to introduce a suitable notion of (intrinsic) winding. To do so we rely on a simple deterministic observation, see \cref{lem:intrinsic->top}, which shows that the intrinsic winding of a smooth simple curve is equal to the sum of its topological winding with respect to either endpoints.  After that, we prove by hand a version of the change of coordinate formula in imaginary geometry:
$$
\tilde h \circ \ph - \chi \arg \ph' =  h
$$
 where $\ph: D \to \tilde D$ is a conformal mapping, $\chi = \tfrac{2}{\sqrt{\kappa}} - \tfrac{\sqrt{\kappa}}{2}$ is the constant of imaginary geometry (note that $\chi = 1/\sqrt{2}$ for $\kappa =2$), and $h, \tilde h$ are GFF with appropriate boundary conditions in the domains $D,\tilde D$. This equation is taken as the starting point of the theory of imaginary geometry (see e.g. \cite{IG1,zipper}) but here it must be derived from the model and our definition of winding. Together with the domain Markov property of the GFF and of the continuum UST (inherited from the domain Markov property of SLE), this implies that the winding of a continuum UST is a Gaussian free field with appropriate boundary conditions.

\textbf{Step 2.} After \cref{T:winding_continuum} is proved, we return to the discrete UST, and we write
 \begin{equation}\label{E:overview}
 h^\d = h_t^\d + \epsilon^\d
 \end{equation}
where $h^\d$ is the winding of the branches of the discrete tree, $h_t^\d$ is the winding of the branches truncated at capacity $t$, and $\epsilon^\d$ is the difference. When $t$ is fixed and $\delta \to 0$ there is no problem in showing that $h^\d_t$ converges to the regularised winding of the continuum UST (this follows from results of Yadin and Yehudayoff \cite{YY} and results about winding in Step 1). By \cref{T:winding_continuum} mentioned above, we also know that as $t\to \infty$, $h_t$ converges to a GFF.

\textbf{Step 3.} It remains to deal with the error term $\epsilon^\d$.
The main idea for this is to construct a multiscale coupling (Theorem \ref{lem:exp_tail}) with independent full plane USTs, which relies on a modification of a lemma of Schramm \cite{SLE}. This allows us to show that the terms $\epsilon^\d$ from point to point have a fixed mean and are independent of each other, even if the points come close to each other. This is enough to show that when we integrate against a test function, the contribution of these terms will vanish.

\textbf{Step 4.} In order to do so, we need to evaluate the moments of $h^\d$ integrated against a test function; however this requires precise a priori bounds on the moments of the discrete winding to deal with bad events when the coupling fails. We therefore first derive a priori tail estimates on the winding of loop-erased random walks (\cref{prop:tail_winding2}). This is where we make use of our RSW crossing assumptions.

\subsection{Organisation of the paper}

 The paper is organised as follows. In \cref{sec:background}, some background and definitions are provided. In \cref{sec:continuum} we formulate and prove the continuum analogue of \cref{T:winding_intro}, \cref{T:winding_continuum}. In \cref{S:UI} we derive the required a priori estimates on winding and describe the multiscale coupling. We put all those ingredients together in \cref{S:proof_discrete}, which completes the proof of \cref{T:winding_intro}.

 {The paper includes fairly technical proofs related to several different areas. This makes a full account quite long. For the sake of brevity and readability, we will defer the proofs of some technical statements to a separate file containing these supplementary materials, which has been appended after the end of the paper in this version, with lettered sections.
}


Throughout the paper, $c, C, c', C'$ etc. will denote constants whose numerical value may change from line to line.  $\Arg$ will denote the principal branch of argument with branch cut $(-\infty,0]$. Also all our domains are bounded unless explicitly stated.

Throughout this paper, universal constants mean constants which do not depend upon anything else in consideration. This should not be confused with our results of ``universality" which is the main topic of this article.

\paragraph{Acknowledgements} We are grateful to a number of people for useful and stimulating discussions, including: Dima Chelkak, Julien Dub\'edat, Laure Dumaz, Alexander Glazman, Rick Kenyon, Jason Miller, Marianna Russikh, Xin Sun, Vincent Tassion and Fabio Toninelli. Special thanks to Vincent Beffara for useful discussions on the winding of SLE towards the beginning of this project.
This work started while visiting the \emph{Random Geometry} programme at the Isaac Newton Institute. We wish to express our gratitude for the hospitality and the stimulating atmosphere.

\section{Background}
\label{sec:background}

For background on SLE and Gaussian free field we refer the readers to Section \ref{sec:background_supp}. Our normalisation of the Gaussian free field is such that the two point function blows up like $-\log |z-w|$ as $w \to z$.

\paragraph{Notation} For $z \in D$, we denote by $R(z,D)$ the conformal radius of $z$ in the domain $D$. That is, if $g$ is any conformal map sending $D$ to the unit disc $\D$ and $z
$ to $0$, then $R(z, D) = \abs{g'(z)}^{-1}$.

\subsection{Winding of curves}
%
%
In this section, we recall simple facts about the winding of smooth curves, which we think are important motivations for the definitions we will use later. Let $\gamma : [0, 1] \to \C $ be a (continuous) curve. For $0 \leq s < t \leq 1$, we will write $\gamma[s, t]$ for the curve $\gamma|_{[s, t]}$.

\paragraph{Topological winding} The \emph{topological} winding of a curve
around a point $p \notin \gamma[0,1]$ is defined as follows. We can write
\begin{equation}
 \gamma(t)  - p = r(t)e^{i \theta(t)}
\end{equation}
where the function $\theta(t): [0,\infty) \mapsto [0,\infty)$ is taken to
be continuous. We define the
winding of $\gamma$ around $p$, denoted $W(\gamma, p)$, to be $\theta(1) - \theta(0)$.
We extend this definition to $p = \gamma(0)$ or $p=\gamma(1)$ by the following formulas when they make sense (that is, the limits exist):
\begin{gather*}
W( \gamma, \gamma(1) )  = \lim_{t \to 1} W(\gamma[0, t], \gamma(1)) ; \quad
W( \gamma, \gamma(0) )  = \lim_{s \to 0} W(\gamma[s, 1], \gamma(0)) .
\end{gather*}

\paragraph{Intrinsic winding}
The \emph{intrinsic} winding of a (smooth) curve is defined as follows. Suppose that $\gamma$ is continuously differentiable and $\forall t, \,\gamma'(t) \neq 0$, and write
\begin{equation}
\gamma'(t) = r_\i(t)e^{i \theta_\i(t)}
\end{equation}
where again
$\theta_\i (t): [0,\infty) \mapsto [0,\infty)$ is taken to be continuous. We define the intrinsic winding of $\gamma$ to be
\begin{equation}\label{E:windingsmooth}
W_{\i}(\gamma) :
= \theta_\i(1)- \theta_\i(0).
\end{equation}

The definition can be extended to piecewise smooth paths by summing the intrinsic winding of each smooth piece together with the jumps in between these pieces. In general, these two definitions are very different, think for example of an ``8'' curve whose intrinsic winding is $0$ while its topological winding is either $-1$, $0$, or $1$ depending on the point. For simple curves however they are related by the following topological lemma which in a sense says that the only amount of nontrivial winding that a simple curve can accumulate is near its endpoints -- anything else has to be unwinded (cancelled out).
Its proof can be found in the supplementary (Lemma B.1).

\begin{lemma}\label{lem:intrinsic->top}
Let $\gamma : [0, 1] \mapsto \C $ be a smooth simple curve with $\gamma'(s) \neq 0$ for all $s$. We have
\begin{equation}
W_{\i}(\gamma) = W(\gamma, \gamma(1) ) + W( \gamma, \gamma(0) ).\label{eq:int->top_gen}
\end{equation}
\end{lemma}

A further important fact is that the topological winding of any path around a boundary point can only arise due to winding of the domain itself. To state this precisely we recall the following notion of argument $\arg_{D;x}$ in $D$ with respect to a boundary point $x$.
This is defined so that
$$
\arg_{D; x} (b) - \arg_{D;x} (a) = \Im ( \int_p \frac{dz}{z - x} )
$$
over any smooth path $p \subset D$ going from $a $ to $b$. In other words $x$ is taken to be the origin and the argument is determined in a continuous way in the simply connected domain $D$. A priori this is defined only up to a global additive constant, whose choice for now can be made in an arbitrary way. Note that if the boundary is locally smooth at $x$, and if $\gamma$ is a smooth path in $\bar D$ such that $\gamma(0) = x$ and $\gamma(0,1] \subset D$ then we can define with an abuse of notation $\arg_{D;x} (\gamma'(0))$ as $\lim_{\eps \to 0} \arg_{D;x} (\gamma(\eps))$, up to the same global additive constant. With these definitions we have the following obvious lemma:

\begin{lemma}\label{lem:wind_start}
Let $D$ be a simply connected domain and let $x$ be a fixed boundary point. Let $\gamma$ be a smooth curve with $\gamma(0) = x$ and $\gamma(0, 1] \in D$. We have
\[
W(\gamma, x) = \arg_{D;x}(\gamma(1) ) - \arg_{D;x}( \gamma'(0) ),
\]
In particular if $\gamma$ is in addition simple:
\begin{equation*}
W_{\i}(\gamma) = W(\gamma, \gamma(1) ) + \arg_{D;x}(\gamma(1) ) - \arg_{D;x}( \gamma'(0) )
\end{equation*}
\end{lemma}

\begin{remark}\label{rq:intrinsic->top}
We will be interested in branches of the uniform spanning trees which are rough  self avoiding curves between the boundary of a fixed domain and an inside point. Furthermore in the discrete, the natural relation is between the intrinsic winding of branches (it is easily extended to piecewise smooth curves) and the height function so we want to make sense of the intrinsic winding of an SLE curve.
\cref{lem:intrinsic->top} will be crucial because it motivates the definition of intrinsic winding for a simple curve using only regularity at the endpoint. Actually as long as we work in a fixed domain the second formula in \cref{lem:wind_start} will allow us to think that $W_{\i}( \gamma ) = W(\gamma, \gamma(1) )$ losing only unimportant deterministic correction terms.
\end{remark}


We now state a lemma showing how the intrinsic winding behaves under conformal maps. This is one
of the key
deterministic statements used in this paper:  it states that the change in winding under an
application of conformal map $\psi$ is roughly $\arg \psi'$.
See \cref{R:ig} below for a clean corresponding statement, which however is only valid for smooth curves.

\begin{lemma}\label{lem:det_winding}
Let $D,D'$ be bounded domains with locally connected boundary and let $\psi$ be conformal map sending $D$ to $D'$. Let $\gamma : [0,1] \mapsto \bar D$ be a curve in $\bar D$. Assume further that $\arg(\psi')$ extends continuously to $\gamma(0)$ and $\gamma(1)$. Let $z$ be a point in $D\setminus \gamma[0, 1]$ and let $R = R(z, D)$ be its conformal radius and assume that $\abs{z-\gamma(1)}\leq R/8$. Then, letting $x = \gamma(0) $ and $x' = \psi (x)$,
\begin{multline}
W\left( \psi(\gamma), \psi(z) \right)  - W(\gamma, z )\\
= \arg_{\psi'(D)} (\psi'(z)) + \arg_{D;x}(z) - \arg_{D';x'}(\psi(z)  ) +O(\abs{z-\gamma(1)}/R),\label{eq:det_winding}
\end{multline}
where the implicit constant in the $O(\abs{z-\gamma(1)}/R)$ is universal and we \emph{choose} the global constants defining the arguments so that the chain rule holds at $x =\gamma(0)$, i.e.,
\begin{equation}
\arg_{D';x'} ( (\psi \circ \gamma)'(0)) = \arg_{D,x} (\gamma'(0)) + \arg_{\psi'(D)} (\psi'(x)).
\label{eq:arg_condition}
\end{equation}
 Furthermore if $\arg(\psi')$ does not extend to $x$, the formula still holds up to a global constant in $\R$ depending on the choice of the constants for the arguments and not on $\gamma$.
\end{lemma}
The proof of this lemma can be found in the supplementary (Corollary B.7). See also Lemma B.8 in the supplementary for a simple geometric condition guaranteeing that $\arg \psi'$ extends continuously near some fixed boundary point $x$: essentially all that is required, beyond local connectedness, is a bit of smoothness for $\partial D$ locally around $x$. It is this condition which explains why without smoothness, the height function is only defined up to a global additive constant (see \eqref{eq:general_D_rough}).

\begin{remark} \label{R:ig} By letting $z \to \gamma(1)$, for a smooth curve $\gamma$ in $D$, we deduce the following somewhat cleaner statement:
\begin{equation*}
W_{\i}( \psi( \gamma ) )=  W_{\i}(\gamma) + \arg_{\psi'(D)}(\psi'(\gamma(1) ) ) - \arg_{\psi'(D)}( \psi'(\gamma(0) ) )
\end{equation*}
where $\arg_{\psi'(D)}$ here is any determination of the argument on the image of $\psi'$. This is significant for the following reasons. The SLE/GFF coupling results developed by Dub\'{e}dat, Miller and
  Sheffield \cite{Dub1,IG4} (referred to as imaginary
  geometry) was defined using a change in coordinate formula
  under conformal map using $\arg \psi'$. \cref{lem:det_winding}
  shows that
  this definition is consistent with the idea that along a branch, the field takes values equal to the intrinsic winding of the branch. In that setting, a key insight is that while the intrinsic winding itself doesn't make sense, its harmonic extension does and this is the only information needed for the GFF.
\end{remark}

\subsection{Continuum uniform spanning tree and coupling with
  GFF}\label{sec:UST}

The breakthrough papers of Schramm \cite{SLE} followed by the paper of Lawler, Schramm and Werner \cite{LSW} established, among other things, the existence and a precise description of the scaling limit of a uniform spanning tree of a domain on a square lattice. We call this limit the \emph{continuum uniform spanning tree}. The following lemma is a consequence of their work which relies on the major result in \cite{LSW} that loop erased random walk when rescaled converges to a SLE$_2$ curve and Wilson's algorithm (see Section A.3 for background on Wilson's algorithm). For now we state the following proposition which is a simple consequence of their work.

\begin{prop}[Wilson's algorithm in the continuum]\label{prop:Wilson_cont}
Let $D$ be a simply connected domain and $z_1,\ldots, z_k \in D$.
We can sample the (a.s. unique) branches of the continuum wired UST in a domain
$D$ from
$z_1,\ldots, z_k$ as follows. Given the branches $\gamma_i$ from $z_i$ for $1\le
i < j$, we inductively sample the branch from $z_j$ as follows. We pick a
point $p$ from the boundary of $D' := D \setminus \cup_{1\le i < j} \gamma_i$
according to harmonic measure from $z_j$ and draw a radial SLE$_2$ curve in $D'$ from $p$ to
$z_{j}$. The joint law of the branches does not depend on the order in
which we sample the branches.
\end{prop}

Readers interested in a more precise exposition are referred to Section A.4 of the supplementary.

\paragraph{Coupling with a GFF}\label{sec:GFF_imaginary}


Let $D$ be a simply connected domain whose boundary $\beta$ is a smooth closed curve and let $x$ be a marked point in the boundary of
the domain. Let us parametrise the boundary $\beta$ of $D$ in an
anticlockwise direction (meaning
that $D$ lies to left of the curve) and such that $\beta(0) = x$. We define \textbf{intrinsic
  winding boundary condition} on $(D,x)$ to be a function $u$ defined on the boundary
by $u_{(D,x)}(\beta(t)) : = W_{\i}(\beta[0,t])$. We call $u_{(D,x)}$ the intrinsic
winding boundary function and extend it harmonically to $D$.

We extend this definition to any simply connected domain $D$ smooth in a neighbourhood
of a marked point $x$ (but not necessarily smooth elsewhere on the
boundary and possibly unbounded) as follows. Let $\varphi: \D \to D$ be a conformal map which
maps $x$ to $1$. Let $u_{(\D,1)}$ be the intrinsic winding boundary
function on $(\D,1)$. Define $u_{(D,x)}$ on $D$ by
\begin{equation}\label{E:uDx}
u_{(\D,1)}= u_{(D,x)} \circ \varphi -\arg_{\varphi'(\D)} \varphi'.
\end{equation}
where we define $\arg_{\varphi'(\D)}$ as the argument defined continuously
in $\varphi'(\D)$ (note $\varphi'(\D)$ do not contain $0$ since $\varphi$ is
conformal) with the global constant chosen such that $u_{(D,x)}$ jumps
from $2\pi$ to $0$ at $x$. One can check that this choice is such that $\arg_{\varphi'(\D)} \varphi'$ verifies the chain rule at $x$ as in \eqref{eq:arg_condition}.
It is elementary but tedious to check that this definition is unambiguous in
the sense that it does not in fact depend on the choice of the conformal map
$\ph$: indeed, if one applies a M\"obius transform of the disc, winding
boundary conditions are changed into winding boundary conditions.

\begin{remark}\label{R:rough_everywhere}
We can still define $u_{(D,x)}$ up to a global constant for domains with general boundary.
\end{remark}





\begin{thm}[Imaginary geometry coupling]
  \label{thm:coupling_intro}
Let $D$ be a simply connected domain with a marked point $x$ on the boundary and let $\chi = \frac1{\sqrt{2}}$. Let $ h = \chi u_{(D,x)} +  h^0_D$ where $ h^0_D$ is a GFF with Dirichlet boundary conditions in $D$ and $u_{(D,x)}$ is defined as in \eqref{E:uDx} and \cref{R:rough_everywhere}.
There
exists a coupling between the continuum wired UST on $D$ and $h$ such that the
following is
true. Let $\{\gamma_i\}_{1 \le i \le k}$ be the branches of the continuum wired
UST
from points $\{z_i\}_{1 \le i \le k}$ in $D$ and let $D' = D \setminus \cup_{1 \le i \le k}
\gamma_i$ . Then the conditional law
of $h$ given $\{\gamma_i\}_{1 \le i \le k}$ is the same as $\chi u_{(D', x)} + h^0_{D'}$ where $h^0_{D'}$ is a GFF with Dirichlet boundary condition in $D'$. Furthermore, $h$ is completely determined by the UST and vice-versa.
\end{thm}
The proof of \cref{thm:coupling_intro} is implicit in \cite{LSW,IG1,IG4,Dub1}. We provide a detailed proof in Section C (Theorem C.2) of the supplementary.

\subsection{SLE$_2$ estimates}\label{S:SLE}

In this section we gather some estimates purely about SLE which are needed for \cref{sec:continuum}. We note that these estimates are the only place in \cref{sec:continuum} where we need to restrict ourself to $\kappa = 2$. These lemmas are no doubt true for SLE$_\kappa$ curves with $\kappa \in (0,8)$ and seem fairly well known in the folklore; however we could not find precise references. Since in any case we will need the corresponding discrete statement for loop-erased random walk, we prefer to provide discrete proofs and deduce the continuum statements below from the known convergence of loop-erased random walk to SLE$_2$ (\cite{LSW}). Since these are the only estimates specific to the case $\kappa = 2$, \cref{T:winding_continuum} extends immediately to other values of $\kappa$ if the statements below are generalised to the corresponding flow lines. (Note however that when $\kappa \neq 2$ flow lines are not simply SLE$_\kappa$ curves but rather specific types of SLE$_{\kappa}(\rho )$ with marked points).


The first estimate controls the probability that an SLE targeted to a point $w$ comes close to another point $z$ in a uniform way and follows from \cref{prop:discrete_coming_close}, \cref{lem:SRW_hit} and \cite{LSW}.
\begin{lemma}
 \label{cor:SLE_distance}
Let $D$ be a domain in $\C$. There exists a universal constant
$c_0>0$ such that the following holds. Let $z,w \in D$ and let $\gamma_w$ be a radial
SLE$_2$ started from a point on the boundary picked according to harmonic
measure from $w$ and targeted at $w$. Let $r = |z-w| \wedge
\dist(z, \partial D) \wedge \dist(w, \partial D)$. Then for all $0<\ve< 1/4$,
\[
 \P(|\gamma_w-z|<\ve r) < \ve ^{c_0}.
\]
There also exists absolute constants $c, c'$ such that if $ r' := \dist(w, \partial D) < \diam(D)/10$, then for all $R> r'$
\[
 \P(\gamma_w \subset B(w,R))\ge 1-c\left(\frac{r'}{R}\right)^{c'}
\]

\end{lemma}

When we work in general domains with possibly rough boundaries, we also need \emph{a priori} bounds on moments of the winding, which follow directly from \cref{prop:tail_winding2} and \cite{LSW}.

\begin{lemma}
 \label{lem:tail_winding2:SLE}
Let $D$ be a simply connected domain, and let $z \in D$.
Let $\gamma_z$ be radial SLE$_2$ towards $z$, started from a point chosen according to harmonic measure on $\partial D$ viewed from $z$. There exist constants $C, c > 0$
such that the following holds.
 For all $t \geq 0$ and $n \ge 1$,
\[
 \P\Big (\sup_{t \leq t_1, t_2 \leq t+1}\abs{W(\gamma_z[0,t_1],z) - W(\gamma_z[0,t_2],z)} >n \Big) < Ce^{-cn} .
\]
\end{lemma}

In a reference domain such as the unit disc, the winding of a single SLE branch has been studied extensively starting with the original paper of Schramm \cite{SLE} itself. In particular, Schramm obtained the following result, which will be used to say that arbitrary moments of the winding at a fixed point blow up at most logarithmically.
\begin{thm}[\cite{SLE}, Theorem 7.2]\label{thm:Schramm}
 Suppose $D$ is the unit disc and let $\gamma_0$ be a radial SLE$_2$ to $0$ started from a point chosen according to the harmonic measure (which is just the uniform measure in this case). We have the following equalities in law
\[
 W(\gamma_0[0,t],0) = B(2t) + y_t, \quad \quad \gamma_0(0) = e^{i \Theta}
\]
where $B(\cdot)$ is a standard Brownian motion started from $0$, $y_t$ is a
random variable having uniform exponential tail and $\Theta \sim Unif[0,2\pi)$. In fact $e^{i(B(2t) + \Theta)}$
is the driving function of $\gamma_0$. Also there exists constants $C, c$ such that for all $s>0$,
\begin{equation}
 \P(|\gamma_0(t) | > e^{-t+s}) \le Ce^{-cs} . \label{eq:Schramm_dist}
\end{equation}
\end{thm}

As a side note, we remark that it is precisely this observation which led Schramm to conjecture that loop-erased random walk converges to SLE$_2$, by combining this result together with Kenyon's work on the dimer model and his computation of the asymptotic pointwise variance of the height function.

\section{Continuum windings and GFF}\label{sec:continuum}

The goal of this section is to show that the winding of the branches in a continuum
UST gives a Gaussian free field.
By analogy with the discrete, we wish to show that the \emph{intrinsic winding} (in the sense of earlier definitions) of the branches of
the continuum UST up to the end points is the Gaussian free field. However
there are two obstacles if we want to deal with this. Firstly, the branches are
rough and hence intrinsic winding does not make sense. Secondly, the winding up
to the end point blows up because the branches wind infinitely often in every
neighbourhood of their endpoints (indeed this should be the case since the GFF
is not defined pointwise).

To tackle the first problem, we note that the topological winding is well defined even for rough curves. We will therefore study the topological winding and add the correction term from \cref{lem:wind_start} by hand (see \cref{rq:intrinsic->top,R:ig} for additional details).

We address the second problem by regularising the winding to obtain a well defined function. The regularisation we use is simply to truncate the UST branches at some point. We will therefore have to show that this regularised winding field converges to a GFF as the regularisation is removed.

\subsection{Winding in the continuum and statement of the result}\label{sec:wind_statement}


Let $D$ be a bounded simply connected domain with a locally connected boundary and a marked point $x$ on its boundary. Let $\cT$ be a continuum wired uniform spanning tree
in $D$. Recall that viewed as a random variable in Schramm's space, a.s. for Lebesgue-almost every $z\in D$ there is a unique branch connecting $z$ to $\partial D$ and for a fixed $z$ this has the law of a radial SLE$_2$.
For $z \in \D$, let $\gamma_z$ be the UST branch starting from $z$ to the boundary (let $z^*$ be the point where it hits $\partial D$), continued by going clockwise along $\partial D$ from $z^*$ to $x$.
Note that since $\partial U$ is locally connected, we can think of $\partial U$ as a curve with some parametrisation \cite{Pommerenke} and hence this description indeed makes sense.
Recall that for any point
$z$, $z^*$ has the distribution of a sample from the harmonic
measure on the boundary seen from $z$, which we denote by $\harm_D(z,
\cdot)$. Also given $z^*$, the
part of the
curve from $z^*$ to $z$ is a radial
SLE$_2$ curve in $D$ from $z^*$ to $z$ in law. We parametrise the part of $\gamma_z$ which lies in $\partial D$ by $[-1, 0]$ so that $\gamma_z(-1) = x$ and $\gamma_z(0) = z^*$. We parametrise the rest of the curve by capacity, i.e for all $t \geq 0$, $-\log ( R(z, D \setminus \gamma_z[-1, t] ) = t - \log( R(z, D) )$, where note that the term $- \log( R(z, D) )$ is necessary for continuity.

If the boundary of $D$ is smooth in a neighbourhood of $x$, then $\gamma$ is smooth near $-1$ and we can define
\begin{equation}
 h_t^D(z) = W(\gamma_{z}[-1,t],z) +\arg_{D;x}(z) - \arg_{D;x}(\gamma'(-1)) \label{eq:general_D}
\end{equation}
where $\arg_{D;x}$ is defined as in \cref{lem:wind_start}.
 The intuition
behind adding these extra terms is to work with (an emulation of) the intrinsic winding rather
than the topological one, see \cref{lem:wind_start}. Note that $h^D_t$ is defined almost surely as an almost everywhere function and hence in particular can be viewed (a.s.) as a random distribution.

For a domain $D$ with general (not necessarily smooth) boundary, the additive constant $\arg_{D;x} (\gamma'(-1))$ might becomes ambiguous. We can nevertheless define $h_t^D$ (and write simply $h_t$ when there is no chance of confusion) as follows:
\begin{equation}
h_t^D(z) = W(\gamma_{z}[-1,t],z) +
\arg_{D;x}(z)  \text{ up to a global constant in }\R \label{eq:general_D_rough}
\end{equation}

For a.e. $z$, we get a branch $\gamma_z$ which is an SLE$_2$ and to which \cref{thm:Schramm} naturally applies. In particular, for a.e. $z$ we get a driving Brownian motion $B_z(2t)$, which forms a Gaussian stochastic process when indexed by $\D$. Informally, the next result, which is the main result of this section, says that this Gaussian process converges to the Gaussian free field as $t \to \infty$. (In fact, the result below even deals with the error term $y_t$). Recall that $u_{(D,x)}$ is the function which gives the intrinsic winding of the boundary curve $\partial D$, harmonically extended to $D$ (see \eqref{E:uDx}).

\begin{thm}\label{T:winding_continuum}
Let $D$ be a bounded simply connected domain with locally connected boundary and a marked point $x \in \partial D$.
As $t \to \infty$, we have the following convergence in probability:
$$
h_t \xrightarrow[t \to \infty]{} \hg.
$$
The convergence is in the Sobolev space $H^{-1 - \eta}$ for all $\eta>0$, and holds almost surely along the set of integers ,i.e, if we only take a limit with $t \in \Z$. Moreover, $\E (\| h_t - \hg \|_{H^{-1- \eta}}^k ) \to 0$ for any $k \ge 1$.
The limit $\hg$ is a Gaussian free field with variance $2$ and winding boundary conditions: i.e., we have
$$
\hg = (1/\chi)\hg^0 + \pi/2 + u_{(D,x)}
$$
where $\hg^0$ is a GFF with Dirichlet boundary conditions on $D$ and $u_{(D,x)}$ is defined as in \cref{E:uDx,R:rough_everywhere}. When the boundary is rough everywhere, the above convergence should be viewed up to a global constant in $\R$.
\end{thm}


\begin{remark}
The coupling defined above between $\cT$ and $\hg$ is in fact the imaginary geometry coupling of \cref{thm:coupling_intro}. In particular this result recovers the fact that $\hg$ is measurable with respect to $\cT$, furthermore providing a fairly explicit construction. It was already proved in \cite{Dub1} that actually both $\cT$ and $\hg$ are measurable with respect to each other and a little known fact is that Section 8.1 in that paper already sketches an explicit construction of the field as a function of the tree which is however different from our own. Note also that the construction in \cite{Dub1} was also conceived as an analogue to Temperley's bijection.
\end{remark}

The rest of this section is dedicated to the proof of \cref{T:winding_continuum}. The general strategy is to first study the $k$-point functions $ \E[\prod h_t(x_i)]$ and to only integrate them at the last step to obtain moments of the integral of $h_t$ against test functions. The advantage of working with the $k$-point function is that it only depends on $k$ branches of the tree, which we know how to sample using \cref{prop:Wilson_cont}.
The existence of $\lim_{t \to \infty}\E[\prod h_t(x_i)]$ will follow from relatively simple distortion arguments and is proved in \cref{lem:quantitative_two_pt}. This essentially shows that $\lim h_t$ exists in the sense of moments (in particular this does not rely on Imaginary Geometry yet).

To identify the limit
we show that the conditional expectation of $\lim h_t$ given some tree branches agrees with the imaginary geometry definition (\cref{sec:conf_kpoint,sec:IG}). The uniqueness in imaginary geometry concludes.
 Finally, \cref{S:generaldomain} covers the extension from the disc to general smooth domains and \cref{S:cont_conv_H} upgrades the convergence from finite dimensional marginals to $H^{-1-\eta}$ using the moment bounds derived earlier.


\subsection{Convergence in the unit disc: one point function}\label{sec:unit disc convergence}


We first prove \cref{T:winding_continuum} in the case $D = \D$ of the unit disc, with the marked point $1$.
The extension of the results to general domains is discussed in \cref{S:generaldomain}. \textbf{Until that section, we henceforth assume $D= \D$.}

\medskip
Recall from \eqref{eq:general_D} that the definition of $h_t$ for this case is given by
\begin{equation}\label{D:hunitdisc}
h_t(z) = W(\gamma_z[-1,t], z) + \arg_{\D;1}(z) - \pi/2,
\end{equation}
where, as in \cref{lem:wind_start}, $\arg_{\D;1}$ is chosen so that $\arg_{\D;1}(0) = \pi$.

\begin{lemma}\label{lem:pointwise}
Let $a_1, \ldots, a_k \in \D$ be distinct and let $K = \gamma_{a_1}[0, t_1] \cup \ldots \cup \gamma_{a_k} [ 0, t_k]$ where $0 \le t_i \le \infty$. Fix $z \in \D$ distinct from any of the $a_i$, and $T>0$. Let $D' =   \D \setminus K$ and assume that $1$ is a smooth point of $\partial D'$. Let $g: D' \to \D$
be a conformal map such that $g(1)=1$ (note such a map is not unique). If $\gamma_z(T) \in B(z,\eps) $ with $\eps \le R(z, D') /8$. Then
\begin{multline}
W ( g(\gamma_z[-1,T]),g(z) ) + \arg_{\D;1} (g (z))  -
\arg_{g' (  D')}(g'(z)) +\epsilon(T)  \\
= W(\gamma_z[-1,T],z) + \arg_{\D;1} (z)
\label{eq:conditional_wind1}
\end{multline}
   where $|\epsilon(T)| = O(\ve/ R(z, D'))$ with the implied constant being universal, and as before $\arg_{g'(D')}$ is chosen so that $\arg_{g' (  D')} (g'(1)) =0$.

Furthermore, assume that $\dist(z, K \cup \partial \D) > \dist(z, \gamma_z [ 0, T])$. Then
\begin{equation}\label{E:changecap}
 \frac{e^{-T}}{4} \le \frac{R(g(z), \D \setminus g(\gamma_z[0,T]))}{R( g(z), \D)}   \le e^{-T} \dist(z, K \cup \partial \D) ^{-1}
\end{equation}
\end{lemma}
\begin{proof}
\eqref{eq:conditional_wind1} is just an application of \cref{lem:det_winding}.
We only have to check the choice of the constant in the arguments. First note that $D'\subset \D$ so we can choose $\arg_{D';1}$ to be a restriction of $\arg_{\D;1}$. Thus it remains to check that the chain rule \eqref{eq:arg_condition} implies that $\arg_{g' (  D')}(g'(1))=0$, which however is easy to check thanks to the fact that $g'(1) >0$.

Moreover, \eqref{E:changecap} follows easily from conformal invariance and domain monotonicity of the conformal radius as well as Koebe's 1/4 theorem (see Theorem 3.17 in \cite{Lawler}).
 \end{proof}

\cref{thm:Schramm} deals with SLE curves towards $0$. We now provide an extension
of this result for SLE curves towards an arbitrary point in the unit disc.

\begin{lemma}\label{lem:one_point_ht}
Let $z \in \D$ and let $\psi: \D \mapsto \D$ be the M\"obius transformation
mapping $z$ to $0$ and
$1$ to $1$. If $\gamma_z(t)\in B(z,\ve)$ where $\eps \le R(z, \D)/8$, then
we have:
\begin{equation}
 W(\psi(\gamma_z[-1,t]),0) =
W(\gamma_z[-1,t],z)+ \pi - \arg_{\D;1}(z)+\epsilon(t)\label{eq:1}
\end{equation}
where the error term $|\epsilon(t)| \le C\ve/R(z,\D)$ for some universal
constant $C>0$
and $\arg_{\D;1}$ is chosen so that $\arg_{\D;1} (0) = \pi$. Also for all $s,t$
\begin{equation}
 \P(\abs{\gamma_z(t) - z} > e^{-t+s}R(z,\D)) \leq ce^{-c's}\label{eq:tip}
\end{equation}
where $c,c'$ are independent of $z$.
\end{lemma}

The proof can be found in Lemma B.9.

 We now want to regularise $h_t$ a bit further by restricting it to an event
where the tip is not too far away from the endpoint. This is something we often need to do in the following, so we will define for $t  \ge 0$ and
$z
\in \D$,
\begin{equation}
\cA(t,z) := \{|\gamma_z(t) -z|<e^{-t/2} R(z, \D)\} \quad ; \quad \hat
h_t(z) :=  h_t(z)\mathbbm{1}_{\cA(t,z)}.
\end{equation}
The event $\cA(t,z)$ and corresponding field $\hat h_t$ will be used throughout our proof of \cref{T:winding_continuum}. By
\cref{lem:one_point_ht}, $\cA(t,z)$ is a very likely event:
\begin{equation}
 \P(\cA(t,z)) >1-ce^{-c't} \label{eq:Atz}
\end{equation}
for some universal constants $c,c'>0$.
\begin{lemma}
 \label{lem:one_point}
We have for every $z \in \D$,
\begin{equation}
\label{E:one_point}
 \lim _{t \to \infty}\E(\hat h_t(z)) = 2\arg_{\D;1}(z) - \frac{\pi}{2}.
\end{equation}
Also, we have the following bounds on the moments:
\begin{equation}\label{E:onepointmoment}
 \E(|\hat h_t(z)|^{k})\le c(k)(1+ t^{k/2}) \quad \quad \E(| h_t(z)|^{k})\le
c(k)(1+ t^{k/2})
\end{equation}

\end{lemma}
\begin{proof}
We first check \cref{E:one_point} and \cref{E:onepointmoment} for $z = 0$. Then $\gamma_z(0)$ is uniformly distributed on $\partial \D$, contributing an expected topological winding of $\pi$ (using the fact that the Loewner equation is invariant under $z \mapsto \bar z$).
Adding the term
$\arg_{\D;1}(0) - \pi/2$ in the definition of $h_t$ in \cref{D:hunitdisc} shows that $\E (h_t ) = 3\pi /2$.  Furthermore, by \cref{thm:Schramm}, we have $\E( h_t(0)^2) \le 2t + o(t)$ and since  $\P(\cA(t,z)) \ge 1 - e^{-ct}$ we deduce from Cauchy--Schwarz that $ \lim_{t \to\infty} \E( \hat h_t(0)) = 3\pi/2$. The moment bound for $h_t(0)$ (and then for $\hat h_t(0)$ using Cauchy--Schwarz) follows again from
\cref{thm:Schramm} and the inequality $|a+b|^k \le 2^{k-1}(|a|^k + |b|^k)$.

For any other $z \in \D$, we start by proving the moment bound \cref{E:onepointmoment}. We join $\gamma_z(t)$ and $z$ by a hyperbolic geodesic
in $\D$, call the resulting union $\gamma'$, and apply $\psi$ to it. Then the image
becomes a concatenation of an SLE$_2$ curve targeted towards $0$ and another
hyperbolic geodesic. Using \cref{eq:1} (which is deterministic) with $\eps= 0$,
  $  W (\psi(\gamma'), 0) - W( \gamma', z)   =\pi - \arg_{\D;1}(z) $. Since the winding of the hyperbolic geodesics are bounded
by at most $\pi$, and the winding of $\psi(\gamma)$ possesses the required moment
bounds, this proves \cref{E:onepointmoment} in $\D$.

Now using \eqref{eq:1}:
\begin{align}
 \E(W(\psi(\gamma_z[-1,t]),0) \mathbbm{1}_{\cA(t,z)}) & =
\E(W(\gamma_z[-1,t],z) \mathbbm{1}_{\cA(t,z)}) \nonumber\\
& +( \pi - \arg_{\D;1}(z) )\P(\cA(t
,z))+\E( \epsilon(t) \mathbbm{1}_{\cA(t,z)})\label{eq:rewrite}
\end{align}
Since $\psi(\gamma_z[-1,t))$ is an SLE$_2$ curve towards $0$ the left hand side of
\eqref{eq:rewrite} converges to $\pi$. Also from \cref{lem:one_point_ht}, the error
term $|\epsilon(t)| <ce^{-c't} \to 0$ on $\cA(t,z)$ and hence converges to $0$
as $t \to \infty$. Recall also the terms added in the definition of $h_t$ in \eqref{eq:general_D}. Combining all these together with \cref{E:onepointmoment}, we have our result.
\end{proof}

\subsection{Conformal covariance of $k$-point function}\label{sec:conf_kpoint}

In the next lemma, we prove the existence of the limit of the $k$-point
function of the regularised winding field of the continuum UST. However we do
not identify the limit at this point as this requires a separate argument. For
this separate argument we will also need a convergence result of the $k$-point
function given several branches of $\cT$, the continuum UST.

\begin{prop}
 \label{lem:two_pt_conv}
Let $\{z_1,\ldots, z_k, w_1, \ldots, w_{k'}\}$ be a set of points
in $\D$ all of which are distinct.  Then the
following is
true.
\begin{itemize}
\item Both $\lim_{t \to \infty} \E (\prod_{i=1}^k  \hat h_{t}(z_i))$ and
$\lim_{t \to \infty} \E (\prod_{i=1}^k   h_{t}(z_i))$ exist and are equal.
\item Let $A=\{\gamma_{w_1}, \ldots,
\gamma_{w_{k'}}\}$ be a set of branches of $\cT$. Let $\E^A$ denote the
conditional expectation given $A$.  Let $g_A: \D
\setminus A \mapsto \D$ be some conformal map which fixes $1$. Let $\tilde h_t$ be an independent copy of $h_t$ in
$\D$. Then
\begin{multline*}
 \lim_{t \to \infty} \E^A(\prod_{i=1}^k   h_{t}(z_i))
 = \lim_{t \to \infty}
\E^A(\prod_{i=1}^k  \hat h_{t}(z_i)) \\
= \lim_{t \to \infty}
\E\big[\prod_{i=1}^k (\tilde h_t(g_A(z_i))- \arg_{g'_A(\D\setminus A)} (g_A'(z_i)) ) \big]\quad a.s.
\end{multline*}
\end{itemize}
\end{prop}
We call the function defined by the first point of the proposition the $k$-point function and we write it $H$:
\begin{equation}\label{eq:def_H}
H(z_1, \ldots, z_k ) := \lim_{t \to \infty} \E (\prod_{i=1}^k   h_{t}(z_i))
\end{equation}

The technical part of the proof of \cref{lem:two_pt_conv} is
accomplished in the following lemma.
\begin{lemma}
  \label{lem:quantitative_two_pt}
Let $\{z_1,z_2,\ldots, z_k\}$ be a set of points
in $\D$ all of which are distinct. Let $r =
\min_i \dist(z_i ,\partial \D) \wedge \min _{i \neq j} |z_i-z_j|$. Let
$t_1\ge t_2>\ldots \ge t_k>t\ge -10 \log r+1$ such that $t_1<10t_k$. Then there are constants $c, c'$ depending only on $k$ such that
\begin{gather*}
|\E(\prod_{i=1}^k  \hat h_{t_i}(z_i))) - \E(\prod_{i=1}^k  \hat
h_{t}(z_i)) | <ct^ke^{-c't} 
\end{gather*}
The same inequality holds with $h$ instead of $\hat h$.
\end{lemma}
Let us comment on why we need to go through the trouble of considering multiple times in \cref{lem:quantitative_two_pt}. A first issue is that when we apply a conformal map, the conformal radius changes differently depending on the point. To get any control both before and after applying the map we therefore need to allow for different $t_i$'s (this is for example the case in equation \eqref{eq:arg_cancel}).
\begin{proof}
We first claim that it is enough to prove that for $t_i$'s as above,
\begin{equation}
  |\E(\prod_{i} \hat h_{t_i}(z_i)) - \E(\prod_{i} \hat h_{t}(z_i))|
  \le ct_1^{k/2}e^{-c't}. \label{eq:bootstrap}
\end{equation}
This clearly completes the
proof since we can break up the interval $[t,t_k]$
into $\cup_{i=1}^J[t{2^{i-1}},t{2^i}]$ where $t{2^{J-1}}\le t_k< t{2^{J}}$.
Using the bound \eqref{eq:bootstrap} for each such interval and using
$t_1 <10t_k$,
\begin{multline}
  \label{eq:sum_bk}
  |\E(\prod_{i} \hat h_{t_i}(z_i)) - \E(\prod_{i} \hat h_{t}(z_i))|
\le c\sum_{i=1}^\infty
(t{2^{i}})^{k/2}e^{-c't{2^{i-1}}} \le ct^{k/2} e^{-c't}
\end{multline}
from which  \cref{lem:quantitative_two_pt} follows.\ The bound for the term involving $h_t$ follows from that of
$\hat h_t$ using \cref{E:onepointmoment}, H\"older's inequality (generalised for
$k$ terms) and the exponential bound on the probability of events $\cA(t_i,z_i)$.

To prove \eqref{eq:bootstrap}, the idea is to consider several
cases depending on how close $\gamma_{z_i}$ gets to the other
points. If it gets very close, the distortion of the conformal map becomes more pronounced and the
estimate in \cref{lem:pointwise} carries large errors. But $\gamma_{z_i}$ getting
close to $z_j$ for some $j \neq i$ is unlikely by \cref{cor:SLE_distance} and comes at a price. So there is a tradeoff between
these two situations. Let $d_i = \inf_{j \neq i}
\dist(z_i, \gamma_{z_j}[0, \infty)) \wedge r$ and $d_{\min} : = \min_i(d_i)$.

\paragraph{Case 1: $-\log (d_{\min}) > t/4 $}
By \cref{cor:SLE_distance} and a union bound, $\P(-\log d_{\min} > t/4 ) \le
ck(\frac{e^{-t/4}}{r})^{c_0} $ for universal constants $c,c_0$. Using the
fact
that $t>-10\log r+1$, we see that $\P(-\log d _{\min}> t/4 ) \le
cke^{-c't}$ for some $c'>0$. Using the one-point moment bounds \eqref{E:onepointmoment} and H\"older's (generalised) inequality,
\begin{equation}
 |\E(\prod_{i} \hat h_{t_i}(z_i)) - \E(\prod_{i} \hat h_{t}(z_i))\mathbbm{1}_{-\log (d_{\min}) > t/4 } )|
< ct_1^{k/2}e^{-c't}\label{eq:case1}
\end{equation}
for some positive universal constants $c,c'$ since $t >1$.

\paragraph{Case 2: $-\log( d_{\min}) \le t/4 $}
Let $A_{i}:= \{\gamma_{z_j} [ 0 , \infty): j \neq i\}$. First we observe that it
is enough to prove
\begin{equation}
  \label{eq:28}
  |\E(h_{t_i}(z_i) - h_t(z_i)|A_{i})\mathbbm{1}_{-\log (d_{\min})<t/4 }|
\le
ct_i e^{-c't}
\end{equation}
since we can use the decomposition
\begin{multline}
  \label{eq:breakup}
|\E(\prod_{i} \hat h_{t_i}(z_i)) - \E(\prod_{i} \hat h_{t}(z_i))| \le\\
\sum_{i=1}^k |\E\Big( \E(\hat h_{t_i}(z_i) - \hat h_t(z_i) |A_{i}
)\hat h_t(z_1)\ldots \hat h_t(z_{i-1}) \hat h_{t_{i+1}}(z_{i+1})\ldots
\hat h_{t_n}(z_k)   \Big)   |
\end{multline}
and then use H\"older's inequality, \eqref{eq:28} and the one-point moment bounds \eqref{E:onepointmoment} to obtain the required
bound.

We now concentrate on the proof of \eqref{eq:28}.
We wish to use
\cref{lem:pointwise} and map out $A_{i}$ by a conformal map $\ph$ mapping $z_i$ to $0$ and $1$ to $1$ and record
the
change in winding of $\gamma_{z_i}$.
By \eqref{E:changecap},
$$
\frac{e^{-t_i}}{4} \le R(\ph(z_i), \D \setminus \ph(\gamma_{z_i} [0, t_i] )) \le e^{- (t_i- t/4 )}.
$$
since $-\log( d_{\min}) \le t/4$.
Therefore, using
\cref{lem:pointwise} for an independent copy $\tilde h$ of $h$ (note that the $\arg_{\varphi' (\D \setminus A_i)}$ term cancels), we have
\begin{equation}
 \E\big((\hat h_{t_i}(z_i) - \hat h_t(z_i))\mathbbm{1}_{\cA(t,z_i)} |A_{z_i}
\big) = \E\Big(\big(\tilde h_{t_i'}(z_i) - \tilde h_{t'}(z_i)
+ \epsilon(t_i) - \epsilon(t) \big)\mathbbm{1}_{\cA(t,z_i)}\Big)
\end{equation}
where $|t_i' -t_i|<t_i/2$ and $|t' - t| <t/2$ and $|\epsilon(t_i)|
\vee |\epsilon(t)| \le e^{-ct}$ on $\cA(t,z_i)$. Now notice that by symmetry,
$\E\Big(\big(\tilde h_{t_i'}(z_i) - \tilde h_{t'}(z_i)
\big)\Big) = 0$. We conclude using Cauchy--Schwarz, the moment bound
\eqref{E:onepointmoment} and the bound on the probability on $\cA(t,z_i)^c$.
\end{proof}

We also need the
following estimate which says that the $k$-point function blows up at most like a power of $\log(r)$
as the points come close.

\begin{lemma}[Logarithmic divergence]\label{lem:log_div}
For any $k \ge 1$ and any  $k$ distinct points $z_1,z_2,\ldots,z_k \in \D$ and
$t>0$,
$$|\E(\prod_{i=1}^k h_t(z_i))| \le  c(1+ \log^{k}(1/r)) \quad ; \quad
|\E(\prod_{i=1}^k \hat h_t(z_i))| \le c(1+ \log^{k}(1/r))$$
where $r =
\min_i \dist(z_i ,\partial \D) \wedge \min _{i \neq j}
|z_i-z_j| $ and $c=c(k)>0$ is a constant.
\end{lemma}
\begin{proof}
We only check the first of these inequalities as the proof of the other is identical.  Let $t = -10 \log r+1$. By \cref{lem:one_point}, for $t' \le t$, we obtain
$\E(\prod_{i=1}^k| h_{t'}(z_i)|) \le C(1+ (t')^{k/2}) \le C( 1+ t^{k/2})$ which is  what we wanted for $t'\le t$. On the other hand if $t'\ge t$, by
\cref{lem:quantitative_two_pt}, $|\E(\prod_{i=1}^k  h_{t'}(z_i)
-\prod_{i=1}^k  h_t(z_i))|<ct^ke^{-ct}$.  Combining with the result for $t' =t$ we obtained the desired bound also for $t'\ge t$.
%
%
\end{proof}

\begin{proof}[Proof of \cref{lem:two_pt_conv}]
   Notice that \cref{lem:quantitative_two_pt}
implies that the quantity $\E(\prod_{i=1}^k  h_{t}(z_i))$ (resp. $\E(\prod_{i=1}^k  \hat h_{t}(z_i))$) is a Cauchy
sequence and hence converges. Moreover,
  \begin{equation}
    \label{eq:25}
   | \E (\prod_{i=1}^k  h_{t}(z_i) - \prod_{i=1}^k  \hat h_{t}(z_i))|
   \le \E(\prod_{i=1}^k  |h_{t}(z_i)| \mathbbm{1}_{\cup_i \cA(t,z_i)^c})
   \le c(1+ t^k)e^{-c't} \to 0
  \end{equation}
Hence the limits are the same, proving the first point.

To simplify notations, we write $g$ in place of $g_A$ and $\arg$ in place of
$\arg_{g_A'(\D\setminus A)}$.
Let $r =
\min_i \dist (z_i, A \cup \partial \D ) $ and take $t > -11\log r +1$. From
\cref{lem:pointwise}, we see (using the obvious domain Markov property and conformal invariance of the UST) that given $A$, we have the equality in distribution
\begin{equation}\label{eq:cond_0}
h_t(z) = \tilde h_{t_i} (g(z)) - \arg g'(z_i) + \epsilon_i(t),
\end{equation}
where
$$t_i = -\log R \big( g(z_i), \D\setminus g(\gamma_{z_i} [0,t] ) \big)  + \log R\big( g(z_i), \D \big).
$$
Hence
\begin{equation}
 \E^A(\prod_{i=1}^k\hat h_{t}(z_i)) = \E(\prod_{i=1}^k (\tilde h_{t_i}(g(z_i))
- \arg(g'(z_i)) +\epsilon_i(t)
)\mathbbm{1}_{\cA(t,z_i)})\label{eq:cond_1}
\end{equation}
By \cref{E:changecap},
$
|t_i - t| \le \log 4 +  \log (1/r).
$
Therefore almost surely,
$9t/10\le t_i \le 11t/10 $ for all $i$ from the choice of
$t$. Thus $t_i\to \infty$ as $t\to \infty$. Further
$|\epsilon_i(t)| = O( e^{-t/2}/r ) = O( e^{-t/2+t/10}) \to 0$ for all $i$
on the event $\cA(t,z_i)$ from \cref{lem:pointwise}. Using all this
information, Cauchy--Schwarz,
\cref{lem:quantitative_two_pt,lem:log_div}, we obtain
\begin{multline}
 | \E(\prod_{i=1}^k (\tilde h_{t_i}(g(z_i))- \arg(g'(z_i)) +\epsilon_i(t)
)\mathbbm{1}_{\cA(t,z_i)}) \\- \E(\prod_{i=1}^k (\tilde h_{9t/10}(g(z_i)) -
\arg(g'(z_i))) | \le ct^ke^{-c't}\nonumber
\end{multline}
almost surely given $A$. The second item of the proposition now follows from the first item.
\end{proof}

To prepare for the proof of convergence in the Sobolev space
$H^{-1-\eta}$ for all $\eta >0$ we need the following convergence of $h_t$ integrated
against test functions.
\begin{lemma}
  \label{lem:multi_point_conv}
Let $\{f_i\}_{1 \le i \le n}$ be smooth compactly supported functions in $\D$. Then for any sequence of integers $k_1,\ldots,k_n$,
\begin{multline*}
\lim_{t \to \infty}\E\prod_{i=1}^n( \int_\D h_t(z)f_i(z)dz)^{k_i}
=
\int_{\D^{\sum_{i=1}^nk_i}} H(z_{11},\ldots,z_{nk_n})\prod_{i=1}^n\prod_{j=1}^{k_i}f_i(z_{ij}
) dz_{ij}
\end{multline*}
where $H$ is as in \cref{lem:two_pt_conv}.
\end{lemma}
\begin{proof}
  Straightforward expansion and Fubini's theorem yield
  \begin{equation}
    \label{eq:19}
    \E\prod_{i=1}^n( \int_\D  h_t(z)f_i(z)dz)^{k_i}
    = \int_{\D^{\sum_{i=1}^nk_i}}
    \E(\prod_{i=1}^n\prod_{j=1}^{k_i}
h_{t}(z_{ij}))\prod_{i=1}^n\prod_{j=1}^{k_i}f_i(z_{ij})dz_{ij}
  \end{equation}
We can apply Fubini because the term inside the integral is integrable from
the moment bounds in \cref{lem:one_point}. We want to take the limit as $t \to
\infty$ on both sides of \eqref{eq:19} and apply dominated convergence theorem
and \cref{lem:two_pt_conv} to complete the proof.
To justify the application of dominated convergence theorem note
that by \cref{lem:log_div}, $\E|\prod_{i=1}^n\prod_{j=1}^{k_i}h_{t}(z_{ij}))| \le
\log^{\sum_{i=1}^nk_i/2}(r)$ where $r  = \min_{(i,j),(i',j')}|z_{ij} -
z_{i'j'}| \wedge \min_{ij}|z_{ij} - \partial \D|$,which is integrable. Further the functions $f_i$'s are
uniformly bounded.
\end{proof}

\subsection{Identifying winding as the GFF: imaginary geometry}\label{sec:IG}

At this stage, we have proven that $h_t$ converges as $t$ goes to infinity, even if we are yet to state a precise meaning for this convergence. However we do not have any information about the limit law and in particular the $k$-point function $H$ is unknown. In this section we will identify the limit with the GFF, using the imaginary geometry coupling.

Recall from \cref{sec:UST} that imaginary geometry provides a coupling between a UST and $\hg$, such that conditionally on some branches $\gamma_i$ of the UST, $\hg$ is a GFF in $\D \setminus \cup \gamma_i$, plus an argument term. Note that this argument term
is exactly the same as the one for the conditional law of the regularised winding $h_t$ (see \cref{lem:two_pt_conv}). The key idea will be to say that if we take a large but finite number of branches, then in $\D \setminus \cup \gamma_i$ all points are close to the boundary and therefore both $\hg$ and $h_t$ have a small conditional variance. The means are essentially the argument terms so they match up to small errors. This will show that $\hg$ and $h_t$ are close in $L^2$, hence identifying the limit.

Note that the only non-trivial fact about imaginary geometry that we need to use is the existence of a field $\hg$ with such a conditional law.

\medskip We first need the fact that the centred two-point function $G(x,y)$ (defined below) is small when one of the points, say $x$, is near
the boundary. For this we start by a deterministic lemma about the argument of conformal maps that remove a small set.

\begin{lemma}[distortion of argument]
 \label{lem:distortion2}
 Let $K$ be a closed subset of $\bar \D$ such that $H = \D \setminus K$ is simply connected (i.e. $K$ is a hull).
 Further assume that the diameter of $K$
is smaller than some $\delta < 1/2$ and $1 \notin \cup_{x \in K} B(x,\delta^{1/2})$. Let
$\tilde g$ denote the conformal map sending $H$ to $\D$ with $\tilde g(0) = 0$ and $\tilde g(1) = 1$. Then
$$|\arg_{\tilde g'  (H)}(\tilde
g'(0)) - \arg_{\tilde g'(H)} (\tilde g'(1))|<C\delta^{1/2},
$$
 where $C$ is a universal constant. Here
$\arg_{\tilde g'  (H)} (\cdot)$ is the argument in $\tilde g'  (H)$ (which does not contain 0), defined up to a global unimportant additive constant.
\end{lemma}
The proof is given in the supplementary, Lemma B.3.

\medskip
We define
\begin{equation}\label{E:G}
G(z_1,z_2,\ldots,z_k) = \lim_{t \to \infty}\E(\prod _{i=1}^k(
 h_t(z_i) - \E( h_t(z_i) ) ))
\end{equation}
 to be the $k$-point covariance function which exists by
\cref{lem:two_pt_conv}. Using \cref{lem:distortion2} we can show that the two-point function is small close to the boundary (i.e., the field has Dirichlet boundary conditions):

\begin{lemma}\label{lem:G0}
For all $|z| \ge 3/4$ and $t_2\ge t_1 >  -10 \log \dist(z,\partial \D) +1$ such that $t_2<10t_1$,
\[
\abs{ \E\Big(h_{t_1}(z)h_{t_2}(0) - \E(h_{t_1}(z))\E(h_{t_2}(0))\Big)} \le  c \dist(z, \partial \D)^{c'}
\]
In particular, as $z \to \partial \D$, $G(0,z) \to 0$.
\end{lemma}
\begin{proof}
Let $r =  \dist(z,\partial \D) = 1- |z|$. Set $t = -10 \log r +1$.
 By
\cref{lem:quantitative_two_pt},
\begin{gather*}
|\E\big(h_{t_1}(z)h_{t_2}(0) - h_t(z)h_t(0)  \big)| \le cte^{-c't} ,\\
|\E h_{t_1}(z)\E h_{t_2}(0) - \E h_t(z)\E h_t(0)   | \le cte^{-c't} ,
\end{gather*}
and observe that $ct e^{- c't} \le c r^{c'}$.  Let us define the event $\cG:=
\cA(t,z) \cap \{\gamma_z\subset
B(z,\sqrt{r})\}$ where here and in the rest of the proof by $\gamma_z$ we mean $\gamma_z[0,\infty)$. From the exponential bound on the probability of
$\cA(t,z)$, $\cA(t,0)$ and \cref{cor:SLE_distance}, we have
\begin{equation}
 \P(\cG \cap \cA(t,0)) \ge 1-ce^{-c't}\label{eq:g}.
\end{equation}
By
\cref{lem:one_point} and Cauchy--Schwarz, we see that
\begin{equation}
 \label{eq:target}
| \E (h_t(z)h_t(0)) -  \E (h_t(z)h_t(0)\mathbbm{1}_{\cG,\cA(t,0)}) | \le cte^{-c't}.
\end{equation}
 Let $g :\D
\setminus \gamma_z \mapsto \D$ be a conformal map fixing $0$ and $1$.
Then from \cref{lem:pointwise}, we have for some independent copy $\tilde
h_t(0)$ of $h_t(0)$,
\begin{equation}
 \E\Big(h_t(0)\mathbbm{1}_{\cA(t,0)} \Big| \gamma_z \Big)\mathbbm{1}_\cG = \E\left(\big(\tilde
h_s(0) - \arg_{g'(\D \setminus \gamma_z)}(g'(0)) + \epsilon(t) \big)\mathbbm{1}_{\cA(t,0)}
\right)\mathbbm{1}_{\cG}
\end{equation}
where $|\epsilon(t)|<ce^{-c't}$ on $\cA(t,0)$  and $t+\log(|z|-\sqrt{r})\le
s \le t+\log 4$ (as in \cref{E:changecap}), and where $\arg_{g'(\D \setminus \gamma_z)}$ is chosen as in \cref{lem:pointwise}, i.e., $\arg_{g'(\D\setminus \gamma_z)} (g'(1)) = 0$. Note also that $|z| - \sqrt{r}
\ge 1 - 2\sqrt{r}>0$. Thus
we obtain using \cref{eq:target},
\begin{multline}
 \E(h_t(z)h_t(0) ) =   O( e^{-ct}) + \\
 \E \Big( h_t(z)\E \Big( (\tilde h_s(0) -
\arg_{g' ( \D \setminus \gamma_z)}(g'(0)) +
\epsilon(t) )\mathbbm{1}_{\cA(t,0)}  \Big| \gamma_z \Big)
\mathbbm{1}_{\cG} \Big)
\end{multline}
We now expand the terms in the right hand side and treat each of them separately. Observe that while $\tilde h$ is independent of $h$, $s$ is still measurable with respect to $\gamma_z$. Hence by symmetry, $\E( \tilde h_s(0) | \gamma_z) = H(0)$ a.s. and hence
$$
\E\big ( h_t(z) \E( \tilde h_s(0) | \gamma_z) \big ) =  H(0) \E( h_t(z)) \to H(0) H(z)
$$
as $t \to \infty$. (In fact a symmetry argument holds here as well and there is no need to let $t \to \infty$).

Regarding the second term, we claim that
$\E (h_t(z)\arg_{g' ( \D \setminus \gamma_z)}(g'(0))\mathbbm{1}_{\cG})$
converges to $0$. Indeed, this follows from the distortion estimate on the
argument we did in
\cref{lem:distortion2} and the fact that on $\cG$, $\diam(\gamma_z(t)) <
\sqrt{r}$. Hence by Cauchy--Schwarz we conclude
$$\E|h_t(z)\arg_{g' (\D \setminus \gamma_z) }(g'(0))\mathbbm{1}_{\cG}| <cte^{-c't} .$$

Finally, for the third term, since $|\epsilon(t)| \le e^{-c't}$ on $\cA(t,0)$,
we deduce that $\E( h_t(0) \epsilon(t) \mathbbm{1}_{\cA(t,0) , \cG} ) \le
te^{-ct}$ by Cauchy--Schwarz and the moment bound. Consequently, we have proved
$$
|\E( h_t (z) h_t(0) ) - H(z) H(0) | \le c e^{-c't}.
$$
Using \cref{lem:quantitative_two_pt} we deduce that $|\E( h_t (z) h_t(0) ) - \E( h_t (z) ) \E (h_t(0) ) | \le c e^{- c't}$.
This proves the lemma.
\end{proof}

 \begin{lemma}
   \label{cor:wind_var}
Let $\{ w_1,w_2, \ldots, w_{k}\}$ be a set of points
in $\D$ all of which are distinct. Let $A=\{\gamma_{w_1}, \ldots,
\gamma_{w_{k}}\}$ be the corresponding set of branches of $\cT$ in $\D$. Let $g:
\D \setminus A \to \D$ be a conformal map fixing $1$. Let $g_z: \D
\setminus A \to \D$ be a conformal
map which maps $z$ to $0$ and $1$ to $1$. Then for any test
function $f$
in $C^\infty(\bar \D)$,
 \begin{multline*}
 \lim_{t \to \infty}  \E \left[\int_{\D} \hat h_t(z)f(z)\mathbbm{1}_{\dist(z,A)
 >e^{-t/10}}dz \Big| A \right] = \\ \int _\D
(2\arg_{\D;1}(g(z)) - \frac \pi 2 -
  \arg_{g' (\D \setminus A)}(g'(z)))f(z)dz
  \end{multline*}
  where $ \arg_{g' (\D \setminus A)}$ is chosen so that $ \arg_{g' (\D \setminus A)} (g'(1)) = 0$.
  \begin{align*}
  \lim_{t \to \infty} \var \left[\int_{\D} \hat h_t (z)f(z) \mathbbm{1}_{\dist(z,A)
 >e^{-t/10} }dz \Big| A \right] & =
  \int_{\D \times \D}G(0,g_z(w)) f(z)f(w)dzdw
\end{align*}
almost surely, where $G(\cdot, \cdot)$ is the two-point covariance function defined in \cref{E:G}.
 \end{lemma}
 \begin{proof}

This proof is an application of dominated convergence theorem. For the
first item, note that for a fixed $t$ we can take the expectation inside by
Fubini and the moment bounds of $\hat h_t$.  Again observe
that from \eqref{eq:cond_1}, we have for an independent copy $\tilde h_t$ of
$ h_t$ in $\D$, and if we write $\cD(z,A) = \{ \dist(z, A)
 >e^{-t/10}\}$,
\begin{align*}
& \int_\D \E\left[\hat h_t(z)f(z)\mathbbm{1}_{ \cD(z,A)}dz  \Big| A\right]
  \\
 & = \int_{\D} \mathbbm{1}_{\dist(z, A)
 >e^{-t/10}} \E\left(\left.\mathbbm{1}_{\cA(t,z)}(\tilde h_{t'}(g(z)) - \arg_{g'(\D \setminus A)} g'(z))
\right| A \right) f(z)dz + O( e^{-ct})
\end{align*}
where $9t/10<t' = t'(z)<11t/10$ almost surely by \cref{lem:pointwise}.
Therefore, the first item follows by taking limit on both
sides and using dominated convergence theorem (whose application is justified by
say \cref{lem:log_div}).

For the variance computation, recall that we write $\E^A$ for the
conditional expectation given $A$. Then, applying the conformal map $g_z$ and using \eqref{eq:cond_0},
\begin{multline}
 \E^A\Big[\big(\hat h_t (z) - \E^A\hat
h_t(z)\big) \mathbbm{1}_{\cD(z,A)}  \big(\hat h_t (w) - \E^A \hat h_t(w)\big)\mathbbm{1}_{\cD(w,A)}   \Big] \\
=\E^A\Big[\prod_{y \in \{z,w\}} \Big(\tilde h_{t_y}(g_z(y))
+\epsilon_y(t) - \E\tilde
h_{t_x}(g_z(y))   \Big)\mathbbm{1}_{\cA(t,y);\cD(y,A)}\Big] \label{eq:arg_cancel}
\end{multline}
because once we condition on $A$, the term $\arg_{g'_z( \D\setminus A)} (g'_z(y))$ is nonrandom and hence cancels out in $\hat h_t(y) - \E^A \hat h_t (y)$. Note that since $z,w$ are at least at a distance $e^{-t/10}$ away from $A$, we have $9t/10 \le t_y \le 11t/10$ for $y\in
\{z,w\}$, and that
$|\epsilon_y(t)| \le c e^{-c't}$ on $\cA(t,y)$.
 By Cauchy--Schwarz and \cref{lem:one_point,lem:quantitative_two_pt}, note that in the right hand side we can replace $t_z, t_w$ by $t$ provided that we add an error term bounded by $ce^{-c't}$, uniformly in $z$ and $w$.

Hence, by Fubini and \cref{eq:arg_cancel},
\begin{align}
 & \text{Var} \left[\int_{\D} \hat h_t(z)f(z) \mathbbm{1}_{\cD(z,A)}dz )\Big| A\right] = \nonumber\\
&\int_{\D^2}\E \Big[\prod_{y \in \{z,w\}} \Big(\tilde h_{t}(g_z(y))
+\epsilon_y(t) - \E\tilde
h_{t}(g_z(y))   \Big)\mathbbm{1}_{\cA(t,y);\cD(y,A)}\Big]f(z)f(w)dzdw   \nonumber \\
& \quad +\textsf{error}(t)\nonumber\\
&
 = \int_{\D^2} \text{Cov}\Big(\tilde h_{t}(0) , \tilde
h_{t}(g_z(w)) \Big) f(z)f(w)dzdw +\textsf{error}(t)\nonumber\\
   & = \int_{z \in \D} \!\!\! f(z)\int_{y \in \D}\text{Cov}\big(\tilde h_{t}(0)
\tilde
h_{t}(y)\big)|(g_z^{-1})'(y)|^2f(g_z^{-1}(y))dydz
+\textsf{error}(t) \nonumber 
\end{align}
where  the
error term satisfies $|\textsf{error}(t)| \le cte^{-c't}$.

Pointwise convergence of the integrand comes from the definition of the two-point correlation function $G$. To conclude we check that we can apply the dominated convergence theorem. By \cref{lem:G0}, one can find a $\delta$ such that for all $y \in \D$ with
$|y| > 1-\delta$ and all $t > -12 \log \delta+1$, $|\text{Cov}(\tilde h_{t}(0)
\tilde
h_{t}(y))| < 1$ almost surely. Therefore, on the set $\{|y | \ge 1- \delta\}$,  the integrand in the last equality is bounded by
 $a(y,z) := |(g_z^{-1})'(y)|^2\|f\|_\infty^2$. On the other hand, by \cref{lem:log_div} the integrand
is bounded by $b(y,z):=\log(|y| \wedge
(1-|y| ))|(g_z^{-1})'(y)|^2\|f\|_\infty^2$ when $|y | \leq 1-\delta$.

 Note that
$$\int_{\{|y| >1-\delta\}} |(g_z^{-1})'(y)|^2 dy = \Leb (\{ w: |g_z(w)
>1-\delta |\}) \le \pi,$$
therefore $a(y,z)$ is integrable on $\{|y|
>1-\delta\}$. Note also that if $|y| \le 1-\delta$, $|(g_z^{-1})'(y)| =
R(g_z^{-1}(y),
g_z^{-1}(\D))/R(y, \D)<c\delta^{-1}$, so $b$ too is integrable on $\{|y|\le
1-\delta\}$. Thus we can take limit inside
the integral. Finally, one can use
\cref{lem:two_pt_conv}, item $1$ to conclude.
 \end{proof}

 We are now going to use the  imaginary geometry coupling of
 \cref{thm:coupling_intro} to prove the following consequence.

\begin{thm}\label{thm:pointwise_GFF}
 Let $f$ be any test function in $C^\infty(\bar \D)$. Let $ h =   \hg^0  + \chi u_{\D,1}$
be the Gaussian free field coupled with the UST according to \cref{thm:coupling_intro}, and let $\hg = (1/\chi ) h + \pi/2.$ Then we have
\[
 \lim_{t \to \infty} \E\Big(( h_t,f) - \big(\hg ,f\big)\Big)^2 = 0
\]
In particular, $( h_t,f) $ converges to $ ( \hg,f)$ in $L^2(\P)$ and in
probability as $t\to \infty$.
\end{thm}

\begin{proof}

Fix $\ve>0$. Using \cref{lem:G0}, pick
$\delta$ such that we have $G(0,y) +2\log(1/|y|) <\ve$ if $|y| \in (1-\delta,1)$.
Fix $\eta$ to be chosen suitably later (in a way which is allowed to depend on
$\eps$ and $\delta$).
Let $A$ be the
set of branches of $\cT$ from a ``dense" set of points, $\frac \eta 4\Z^2 \cap \D$, to $1$. Note that
that $R(z,\D \setminus
A) <\eta$ for any $z \in \D$ by Koebe's 1/4 theorem. Let
$D'= \D \setminus A$. Define $\bar h_t(z) = \hat
h_t\mathbbm{1}_{\dist(z,A)>e^{-t/10}}$.  
First of all notice that
\begin{equation}
 \E(( h_t,f) - ( h_{\text{GFF}},f))^2 \le 2 \E((\bar h_t,f) - ( h_{\text{GFF}},f))^2 + 2 \E((\bar h_t,f) - (h_t,f))^2 \label{eq:hat_bar}
\end{equation}
By adding and removing $\hat h_t$ the last expression on the right hand side of \eqref{eq:hat_bar} converges
to $0$ by \eqref{eq:25} and the following fact: using Cauchy--Schwarz and the moment bounds on $\hat h_t(z)$,
\begin{align*}
 \E((\bar h_t,f) &- (\hat h_t,f))^2  = \E(\int_{\D} \hat
h_t(z)f(z)\mathbbm{1}_{\dist(z,A) \le e^{-t/10}})^2 \\
& \le \int_{\D^2}
\E|\hat h_t(z) \hat h_t(w)|\|f\|^2_\infty \mathbbm{1}_{\dist(z,A)\vee \dist(w,A)
\le e^{-t/10}}dzdw\\
& \le  c\|f\|^2_\infty(1+t) \int_{\D^2} \P(\dist(z,A) \vee \dist(w, A) \le e^{-t/10})^{1/2}dzdw\\
& \le c\|f\|^2_\infty(1+t)\P(\dist(U_1,A)\vee
\dist(U_2,A) \le e^{-t/10})^{1/2}  \label{eq:area}
\end{align*}
where $U_i \sim \text{Unif}(\D)$ and are independent of everything else and
each other. Now if $z$ is a point in $\tfrac{\eta}{4} \Z^2 \cap \D$, then
$$
\P( \dist(U, \gamma_z) \le e^{-t/10}) \le c_\eta e^{-c't}
$$
by  \cref{cor:SLE_distance}. Hence summing up over $O(1/\eta^2)$ points and using a union bound, we see that (for every fixed $\eta$), $\P(\dist(U_1 ,A)\vee
\dist(U_2,A) \le e^{-t/10})^{1/2}  \to 0$ exponentially fast and thus the second term on the right hand side of \cref{eq:hat_bar} tends to 0.


Let $\E^A$ and $\var^A$ denote the conditional expectation and
variance given $A$.  It is easy to
see
\begin{equation}
 \E^A((\bar h_t,f) - ( h_{\text{GFF}},f))^2 \le 3\text{Var}^A (\bar h
_t,f)+3\text{Var}^A ( h_{\text{GFF}},f) + 3(\E^A(\bar h
_t,f) - \E^A( h_{\text{GFF}},f))^2\label{eq:tri_bk}
\end{equation}
Note that it is enough to show that as $t\to \infty$ the left hand side of \eqref{eq:tri_bk} can be made smaller than $\eps$ (in expectation) by choosing $\eta$ suitably since this implies that the first term in the right hand
side of \eqref{eq:hat_bar} is smaller than $\eps$ plus a term converging to zero, which completes the proof.

The last term of \eqref{eq:tri_bk} converges to $0$ for every $\eta$ from the convergence of
expectations in \cref{cor:wind_var} and the fact that $  h_{\text{GFF}}$ satisfies the correct boundary conditions given $A$ (which is a consequence of the imaginary geometry coupling).


For the other terms, recall that conditionally on $A$, $
h_{\text{GFF}}$ is just a free field in $D'$ with variance $1/\chi^2$ and with Dirichlet boundary
condition plus a harmonic function. Recall that the variance of a GFF
integrated against a test function is given by an integral of the Green's function
in the domain. Also recall that the Green's function is conformally
invariant. In particular if $g_z$ is the conformal map from $D'$ to $\D$ sending $z$ to $0$ and $1$ to $1$,
using the explicit formula for the Green's function in the unit disc, we have
\[
 \var^A(
h_{\text{GFF}}) =  -\int_{\D\times \D} \frac1{\chi^2}\log|g_z(w)|f(z)f(w)dzdw.
\]
Plugging in the variance formula derived in \cref{cor:wind_var} and since $\chi = 1/\sqrt{2}$,
\begin{align}
  \label{eq:23}
& \E(\text{Var} ^A(( h_{\text{GFF}},f))) + \E(\lim_{t \to
\infty}\text{Var}^A (\bar
  h_t,f))) \\& \quad  = \E(\int_{\D\times
\D}(G(0,g_z(w))-2\log|g_z(w)|)f(z)f(w)dzdw). \nonumber
\end{align}
By a change of variable $y=g_z(w)$,
\begin{equation}
 \int_{\D } (G(0,g_z(w))-2\log|g_z(w)|)f(w)dw \leq \|f\|_\infty \int_{\D }
|(G(0,y)-2\log|y|)| \, |(g_z^{-1})'(y)|^2dy. \label{eq:int1}
\end{equation}
As in \cref{cor:wind_var}, we are going to estimate the integral on two domains, $B_\delta =
\{|y|<1-\delta\}$ and $B'_\delta := \D \setminus B_\delta$. Recall that
by the choice of $\delta$, $G(0,y)+2\log1/|y|<\ve$ if $y \in B'_\delta$. Hence
\begin{equation}
  \label{eq:5}
  \int_{B'_\delta}
|(G(0,y)-2\log|y|)| \, |(g_z^{-1})'(y)|^2dy< \ve
\Leb(g_z^{-1}(B'_\delta)) <\pi \ve \nonumber
\end{equation}
To estimate the integral in $B_\delta$, notice that $|(g_z^{-1})'(y)| =
R(g_z^{-1}(y), D')/R(y,\D)<\eta/\delta$ if $y \in B_\delta$, by Koebe's 1/4 theorem. Hence
\begin{equation}
  \label{eq:12}
  \int_{B_\delta}
|(G(0,y)-2\log|y|)| \, |(g_z^{-1})'(y)|^2dy  <
\frac{\eta^2}{\delta^2} \int_{B_\delta}
|G(0,y)-2\log|y|| dy
\end{equation}
The integral on the right hand side is finite via the bound
\cref{lem:log_div}. After bounding the integral over $w$, it remains to bound
the integral over $z$ in \eqref{eq:23} by $\|f\|_\infty$ times the area. This completes the proof since we can choose $\eta$ such that
$\eta/\delta <\ve$
where $\ve$ is arbitrary.
\end{proof}

	
\begin{corollary}
In the same setup as \cref{thm:pointwise_GFF}, for any $p>0$, and any sequence of test functions $(f_1, \ldots, f_n) \in
C^\infty (\bar \D)$ and integers $k_1, \ldots, k_n$, as $t \to \infty$,
\begin{equation}
 \E[|\prod_{i=1}^n( h_t,f_i)^{k_i} - \prod_{i=1}^n( \hg,f_i)^{k_i}|^p] \to 0.
\end{equation}
\end{corollary}
\begin{proof}
It is enough to prove this fact when $p$ is an even integer.
 For $n=1$ this follows from the fact that $(h_t, f)$ converges in $L^2(\P)$ towards $( h_{\text{GFF}}, f)$, and $(h_t,f)$ is bounded in $L^p$ for any $p>1$ by Lemma \ref{lem:log_div}.

For general $n\ge 1$ we proceed by induction, and note that by the triangle inequality in $L^p$ (i.e., Minkowski's inequality), if $X_n \to X$ and $Y_n \to Y$ in $L^p$ for every $p>1$ then $X_n Y_n \to XY$ in $L^p$ for every $p>1$.
\end{proof}

\subsection{General domains}
\label{S:generaldomain}

In this section we state our result when $D$ is a bounded domain with a locally connected boundary and we defer the proof to Theorem F.1 of the supplementary.
 Recall that our definition of $u_{(D,x)}$ in \eqref{E:uDx} only makes sense when the boundary is smooth in a neighbourhood of a marked point $x \in \partial D$ (while otherwise it is only defined up to a global additive constant, see \cref{R:rough_everywhere}). The general idea is to show that in the limit one has
$$
  h^D \circ \ph(z) = h^\D(z)  + \arg_{\ph'(\D)}(\ph'(z)), \ \ z \in \D
$$
which is the imaginary geometry change of coordinates (see \cite{IG1,IG4}).

\begin{thm}\label{L:change_coord}
Let $D$ be as above.
Let $f$ be any bounded Borel test function defined on $\bar D$. Let $ h = \hg^0 + \chi
u_{(D,x)}$
 be the GFF coupled to the UST according to the imaginary geometry coupling of
\cref{thm:coupling_intro} and $u_{(D,x)}$ is as in \eqref{E:uDx}. Then
 $( h_t^D ,f) $ converges to $(\hg^D,f)$ in $L^2(\P)$ and in
probability as $t\to \infty$, where $\hg^D = \chi^{-1}h+\pi/2$.
\end{thm}

\subsection{Convergence in $H^{-1-\eta}(D)$}\label{S:cont_conv_H}

Let $D$ be a domain with locally connected boundary and now assume also that $D$ is bounded.
Let $(e_j)_{j \ge 1}$ denote the orthonormal basis of $L^2(D)$ given
by the eigenfunctions of $-\Delta$ in $D$. Let $\hg$ denote the process defined in \cref{thm:pointwise_GFF}, which is $(1/\chi)$ times a GFF with winding boundary conditions multiplied by $\chi$.
We now strengthen the convergence from a convergence in probability or $L^2(\P)$ for finite-dimensional marginals to a convergence in the Sobolev space $H^{-1 - \eta}$.

\begin{prop}\label{P:conv_GFF_H}
For every $\eta>0$, the field $h_t$ converges to $\hg$ in
$H^{-1-\eta}$ in probability as $t \to \infty$. Further, $\{h_n\}_{n
  \in \N}$ converges almost surely to $\hg$ as $n \to \infty$ along positive integers. Also for all $1 \le k <\infty $, $\E [ \| h_{u} - h_\infty \|^k_{H^{-1 - \eta}} ]\to 0$.
\end{prop}
\begin{proof}
The basic idea is to show that $h_t$ is a Cauchy sequence in $H^{-1 - \eta}$. Let $u \ge t$.
We start by getting bounds on $\E[(h_{u} - h_t,e_j)^2]$.
By Fubini's theorem and Cauchy--Schwarz,
 \begin{align}
  &\left(\E [ (h_{u} - h_t,e_j)^2]  \right)^2  = \left( \int_{D^2} \E [ (h_{u}(z)-h_t(z))
(h_{u}(w)-h_t(w))] e_j(z)e_j(w)dzdw \right)^2 \nonumber \\
& \le \int_{D^2} \Big(\E[(h_{u}(z)-h_t(z))
(h_{u}(w)-h_t(w))]\Big)^2dzdw \int_{D^2} e_j^2(z)e_j^2(w)dzdw\nonumber\\
&=  \int_{D^2} \Big(\E[(h_{u}(z)-h_t(z))
(h_{u}(w)-h_t(w))]\Big)^2dzdw
\label{eq:27}
  \end{align}
since $e_j$ forms an orthonormal basis of $L^2(D)$. Let $r(z,w) =  |z-w|\wedge
\dist(z,\partial D) \wedge \dist(w,\partial D)$. We are going to break up the
integral in \eqref{eq:27} into two cases, either $t
\le -10\log (r(z,w)) +1 $ (i.e, $r^{10} \le e^{1-t}$) or otherwise. In the first case Cauchy--Schwarz and
the bound on moment of order two yield
\begin{multline}
  \label{eq:29}
 \int_{D^2} \Big(\E[(h_{u}(z)-h_t(z))
(h_{u}(w)-h_t(w))]\Big)^2  \mathbbm{1}_{t \le  -10 \log
  (r(z,w)) +1 }dzdw \\ \le c (1+u^2) \int_{D^2} \mathbbm{1}_{ r(z,w)^{10} \le e^{1-t}} dzdw \le c(1+u^2)e^{-c't}
\end{multline}
On the other hand,
\begin{multline}
  \int_{D^2} \Big(\E[(h_{u}(z)-h_t(z))
(h_{u}(w)-h_t(w))]\Big)^2  \mathbbm{1}_{t >  -10 \log
  (r(z,w)) +1}dzdw  \\
  \le cu^2e^{-ct}\int_{D^2} \mathbbm{1}_{t >  -10 \log
  (r(z,w)) +1 }dzdw \le c(1+u^2)e^{-c't}\label{eq:30}
\end{multline}
where the second inequality above follows from (F.1) of the supplementary to formulate the correlation in the unit disc, and \cref{lem:quantitative_two_pt} to control this correlation by $r(z,w)$
(note that the bound of \cref{lem:quantitative_two_pt} holds also if $\hat h$ is replaced by $h$, because of the control on moments of $h_t$ in \cref{lem:one_point} and the exponential bound on the probability of $\cA(t,z)^c$). Combining
\eqref{eq:29} and \eqref{eq:30}, we obtain
\begin{equation}
\E[(h_{u} - h_t,e_j)^2] \le c(1+u^2)e^{-ct}.\label{eq:bootstrap2}
\end{equation}
Now let $t = n$ and $n \le u \le n+1$. Then by Jensen's inequality
\begin{multline}\label{E:controlnorm}
[\E  \| h_u - h_n\|_{H^{-1 - \eta}} ]^2 \le \E [ \| h_{u} - h_n \|^2_{H^{-1 - \eta}} ]  = \\
 \sum_{j \ge 1} \E[(h_{u} - h_n,e_j)^2]\lambda_j^{-1-\eta} \le c(1+ n^2) e^{- c'n } \sum_{j=1}^\infty \lambda_{j}^{-1- \eta}
\end{multline}
and since $\sum_{j=1}^\infty \lambda_{j}^{-1- \eta}< \infty$ by Weyl's law
we deduce (by applying Markov's inequality and the Borel--Cantelli lemma) that $h_n$ is almost surely a Cauchy
sequence in $H^{-1 - \eta}$ along the integers, and hence converge to a limit $h_\infty$ almost
surely in $H^{-1 - \eta}$. Furthermore by the triangle inequality and
\eqref{E:controlnorm} we get $ \E ( \| h_\infty - h_n \|_{H^{-1 - \eta}} ) \le
C e^{-c'n}$. Then
using \eqref{E:controlnorm} again, we deduce
$$
\E [ \| h_{u} - h_\infty \|_{H^{-1 - \eta}} ] \le C (1+u^2) e^{- c'u}
$$
and hence $h_u$ converges in probability in $H^{-1 - \eta}$ to $h_\infty$. To get convergence of $\E [ \| h_{u} - h_\infty \|^k_{H^{-1 - \eta}}$ for any $k \ge 1$ a similar argument would work: one needs to consider $\E(h_t-h_u,e_j)^{2k}$ and hence there would be $k$ terms inside the integrals around  \eqref{eq:27}. We skip this here because an exact similar argument with minor modifications is done in the proof of \cref{thm:main_details} later. Furthermore we have that $h_\infty = \hg$ by considering the action on test functions and \cref{thm:pointwise_GFF}. This finishes the proof of \cref{P:conv_GFF_H} and hence also of \cref{T:winding_continuum}.
\end{proof}

\section{Discrete estimates on uniform spanning trees}
\label{S:UI}

The goal of this section is to gather the lemmas needed in \cref{S:proof_discrete} for the proof of the main result, \cref{T:winding_intro}. To make the purpose of the results in this section more clear, it will be useful for the reader to recall the general overview of the proof in \cref{S:overview}.
Recall that one additional difficulty comes from the fact that we need to deal with convergence of moments and not just convergence in law. Therefore we also need a priori estimates on the tails of our variables to use Cauchy--Schwarz bounds and dominated convergence theorems. In particular a bound on the tail of the winding of loop-erased random walk is derived in \cref{SS:winding}.

\subsection{Assumptions on the graph}\label{sec:assumption}
Let $\{G^\d\}_{\delta >0}$ be a sequence of planar infinite (directed, weighted) graphs embedded properly in the plane. This means that for any $\delta>0$ the embedding is such that
no two edges cross each other. (The reader may think of $\delta$ usefully as a ``mesh size" or microscopic scale). Vertices of the graph are identified
with some points in $\C$ given by the embedding. We allow $G^\d$ to have
oriented edges with weights. A continuous time simple random walk
$\{X_t\}_{t \ge 0}$ on such a graph
$G^\d$ is defined in the usual way: the walker jumps from $u$ to $v$ at
rate $w(u,v)$ where $w(u,v)$ denotes the weight of the oriented edge
$(u,v)$. Given a vertex $u $ in $G^\d$, let $\P_u$ denote the law of
continuous time simple random walk on $G^\d$ started from $u$.
For $A \subset \C$,
we denote by $A^\d$ the set of vertices of $\Gd$ in $A$.

 In this section,
$B(a,r)$ will denote the set $\{z:|z-a|<r\}$. For $A \subset \C$, denote by $A+z
:= \{z+x:x \in A\}$ to be the translation of $A$ by $z$. We assume $G^\d $ has the
following
properties.

\begin{enumerate}[{(}i{)}]

  \item \label{boundeddensity} \textbf{(Bounded density)} There exists $C$ such that for any $x \in \C$, the number of vertices of $G^\d$ in the square $x + [0,\delta]^2$ is smaller than $C$.

  \item \label{embedding} \textbf{(Good embedding)} The edges of the graph are embedded
  in such a way that they are piecewise smooth, do not cross each other and have uniformly bounded winding. Also, $0$ is a vertex.


\item\label{irreducibility}  \textbf{(Irreducible)} The continuous time random walk on $G^\d$ is irreducible in the
  sense that for any two vertices $u$ and $v$ in $G^\d$, $\P_u(X_1 =
  v)>0$.
\item \label{InvP} \textbf{(Invariance principle)}
The continuous time random walk $\{X_t\}_{t \ge 0}$ on $G^\d$
  started from $0$ satisfies:
\[
 {( X_{t/\delta^2})_{t \ge 0}} \xrightarrow[\delta \to 0 ]{(d) }
(B_{\phi(t)})_{t \ge 0}
\]
where $(B_t, t \ge 0)$ is a two dimensional standard Brownian motion in $\C$ started from
$0$, and $\phi$ is a nondecreasing, continuous, possibly random function
satisfying $\phi(0) = 0 $ and $\phi(\infty) = \infty$. The above convergence
holds in law in Skorokhod topology.

\begin{figure}
\centering
\includegraphics[scale = 0.4]{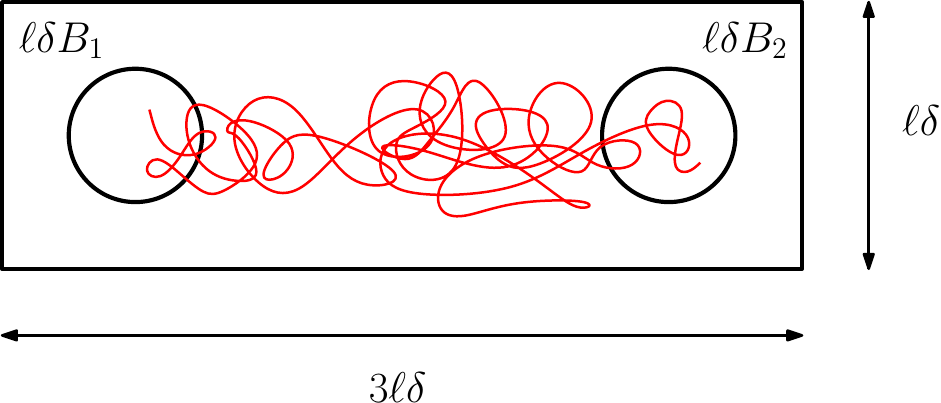}
\caption{An illustration of the crossing condition.}\label{fig:crossing}
\end{figure}

\item \label{crossingestimate} \textbf{(Uniform crossing estimate).}  
Let
    $\cR$
   be the horizontal rectangle $[0,3]\times [0,1]$ and $\cR'$ be the vertical
rectangle with same dimensions, and let $B_1 := B((1/2,1/2),1/4)$ be the
\emph{starting ball} and $B_2:= B((5/2,1/2),1/4) $ be the \emph{target ball}
(see \cref{fig:crossing}). There
exist constants $\delta_0 >0$ and $\alpha_0>0$ such that for all $z
\in \C, \delta>0$, $\ell \ge 1/\delta_0$, $v \in \ell \delta B_1$ such that $v+z \in \Gd$,
\begin{equation}
  \P_{v+z}(X \text{ hits }(\ell \delta  B_2+z) \text{ before exiting }  (\ell \delta \cR+z))
>\alpha_0.\label{eq:cross_left_right}
\end{equation}

The same statement as above holds for crossing from right to left, i.e., for
any $v \in \ell \delta B_2$, \eqref{eq:cross_left_right} holds if we replace $B_2$ by
$B_1$.
Also, the corresponding statements hold for the vertical rectangle $\cR'$.
\end{enumerate}

We point out that the invariance principle starting from zero, together with the crossing estimate, imply an invariance principle starting from arbitrary vertices $v^\d$ in $G^\delta$ converging to a point $z \in \C$ as $\delta \to 0$. Briefly, by the invariance principle a walk starting from zero has a positive probability of making a small circuit around $z$; so that a walk starting from $z$ can be coupled to the walk starting from zero the first time it hits this circuit (note that this time is a.s. finite by the crossing assumption).

However, we also point out that that the crossing estimate would not follow from the
invariance principle even if we assume it for all starting points; this would require a uniformity in the rate of convergence which does not necessarily hold in interesting examples for applications (in particular in the case of T-graphs which motivates our work).

\begin{remark}\label{R:conv_lerw}
In this paper we make crucial use of a result of Yadin and Yehudayoff \cite{YY} showing convergence of loop-erased random walk to SLE$_2$.
This holds under an assumption of invariance principle for simple random walk. However we need to spare a few words on their paper since they do not state their main theorems with quite the level of generality that we need here. Here are the points to note in order to check that their proofs extend to our setting.

1. They considered \emph{backward} loop
  erased random walk, whereas we consider forward LERW (where loops
  are erased in chronological order). However, these have the same
  law, even when the graph is directed as is the case here (see
  \cite{Lawler}).

2. They consider scaled versions of a single infinite graphs instead of an arbitrary family of graphs with a scale parameter $\delta$. However this is just for ease of notation as the proofs never use the relation between the graphs at different scales and all estimates are uniform over the underlying graphs.

 3. The result of \cite{YY} is stated with $D = \D$ and $z = 0$, but this does not play a role in the proof. Notice in particular  that the key estimate on the Poisson kernel (\cite{YY}, Lemma 1.2) is stated with the generality we require, namely on an arbitrary domain and an arbitrary target point. More precisely, the target point is $0$ but the domain is an arbitrary domain which contains $0$: this of course amounts to the same thing as fixing and choosing an arbitrary fixed point inside a given domain. They also require that the inner radius of the domain (with respect to the target point $0$) is greater than $1/2$. Up to a change of scale, this amounts to requiring that the point $z$ is at positive distance from the boundary. The convergence in Lemma 1.2 of \cite{YY} is hence uniform in the domain if we assume that the distance from the boundary is bounded below. In our case we will only use the result of \cite{YY} at a finite number of points which sit in the support of a compactly supported function $f$ on $D$, so this assumption is certainly verified.

 See in particular \cite{YY}, Proposition 6.4 for a statement about the convergence of the driving function to Brownian motion in the general setup we require. Note that planarity of the graph plays a crucial role to prove this estimate. Also, the proof of tightness in the sense of Lemma 6.17 in \cite{YY} follows through in our situation with no significant modification.

\end{remark}

\begin{remark}
Let us briefly discuss the role of these assumptions. The invariance principle should be essentially a minimal assumption for the convergence. Indeed the Gaussian free field depends on the Euclidean structure of the plane and it is difficult to imagine any graph converging in a sense to the Euclidean plane without satisfying an invariance principle. In practice the invariance principle and irreducibility, together with the fact that there is no accumulation point, are exactly the assumptions needed for the convergence of the loop-erased random walk to SLE$_2$ from \cite{YY}.

Our main additional assumption is the uniform crossing estimate. It is used extensively to derive various \emph{a priori} estimates on the behaviour of the random walk, the uniformity over starting points and scale being a key factor for different multi-scale arguments. We believe however that there should be some room in our proofs to weaken this assumption.

The bounded density assumption is actually only needed for a union bound in the proof of \cref{lem:Schramm_finiteness}. It is clear from that proof that it would not be needed if the uniform crossing assumption was allowed to ``scale" with the local density of the graph.

For future reference, we note that the uniform crossing assumption can be rephrased equivalently by saying that there exist $\delta_0,\alpha_0$ such that for any $r>0$, for any $\delta \le r \delta_0$, the probability to cross $r\cR$ from $rB_1$ to $rB_2$ (left to right) and from $rB_2$ to $rB_1$ (right to left) is at least $\alpha_0$, and likewise for the vertical rectangle.

\end{remark}


\subsection{Russo--Seymour--Welsh type estimates }\label{sec:regularized_estimate}

Let $D$ be a domain with locally connected boundary. To define the wired UST in the discrete domain, we perform the following surgery.
For every oriented edge $(xy)$ which intersect $\partial D$, we add an extra auxiliary vertex at the first intersection point (when following the embedded edge $(xy)$). We then replace $(xy)$ by an oriented edge from $x$ to this auxiliary vertex, keeping the same weight.
 The \emph{wired graph} is the graph induced by all the vertices in $D^\d$ along with all the auxiliary vertices and then wiring (or gluing) together all the auxiliary vertices. We denote by $\partial D^\d$ all the edges with one endpoint being an auxiliary vertex and another endpoint inside $D$. The wired UST $\cT^\d$ is defined to be  a uniform spanning tree on the wired graph. It is useful to think of the wired tree being sampled by Wilson's algorithm with the wired vertex being the initial root vertex. All the results in this section hold without the assumption
of CLT (just assumptions \ref{boundeddensity} and \ref{crossingestimate} from \cref{sec:assumption} are needed).

We denote by $A(x,r,R)$ the annulus $\{z\in \C:r<|z-x|<R\}$. Let $v \in
A(x,r,R)^\d$. The random walk trajectory from a vertex $v$  is the
union of the edges it crosses (viewed as embedded in $\C$). We say random walk
from $v$ does a \textbf{full turn} in $A(x,r,R)$
if the random walk trajectory intersects every curve in the plane starting
from $\{|z|=r\}$ and ending in $\{|z| = R\}$. We will write $X[a,b]$ for the random walk trajectory between times $a$ and $b$. We will allow ourselves to see $X[a,b]$ and the loop-erased walk $Y$ either as sequences of vertices, continuous paths in $\mathbb{C}$, or as sets depending on the place but this should not lead to any confusion.
For any continuous curve $\lambda \in \C$, with a slight abuse of terminology we will freely say that ``$(X_t, t \ge 0)$ \textbf{crosses (or hits)} $\lambda$ at time $t>0$" to mean that $X_t \neq X_{t^-}$ and the half-open edge $(X_{t^-}, X_t]$ intersects the range of $\lambda$.

{In this section and the next, we will always assume that the loop-erased walk is generated by erasing loops chronologically from a simple random walk. We will allow ourselves to refer to the simple random walk associated to a loop-erased walk without further mention of this.}

\begin{lemma}
  \label{lem:full_turn}
Fix $0<r<R$, $\Delta = R-r$. There exists constants $c>0$ and $\alpha > 0$
depending only on $R/r$ such that the following holds. For all $\delta \le c r \delta_0$ where $\delta_0$
is as in \cref{crossingestimate}, for
all $x \in \C$ and $v \in A(x,r+\Delta/3,R-\Delta/3)^\d$, the
probability that the random walk starting at $v$ does a full turn before
exiting $A(x,r,R)$ is at least $\alpha$.


\end{lemma}
\begin{proof}
We use the uniform crossing assumption here and use the notations and
terminology as described in \cref{sec:assumption}. It is easy to see that we can
find a sequence of rectangles $\cR_1,\cR_2,\ldots,\cR_k \subset
A(x,r+\Delta/4,R-\Delta/4)^\d$ where each such rectangle is a rectangle of the
form $\ve \cR+z$ or $\ve \cR' +z$ (i.e. a scaling and translation of $\cR$ or $\cR'$)
such that the starting ball of $\cR_i$ coincides with the target ball of
$\cR_{i-1}$, $v $ is in the starting ball of $\cR_1$ and the following holds.
If the simple random walk iteratively moves from the starting ball to the
target ball of $\cR_i$ for each $i =1,\ldots,k$ such that the starting vertex of
$\cR_{i+1}$ is the vertex where the walk enters the target ball of $\cR_i$, then
the walker accomplishes a full turn in $A(x,r,R)$. Here we can choose the
scaling $\ve$ as a function of the ratio $r,R$ and the number $k$ to be bounded
above by a constant $k_0(R/r)$. Applying the uniform crossing estimate and
the Markov property of the walk, we see that this probability is bounded below
by $\alpha_0^{k_0}$, thus completing the proof.
\end{proof}

Actually we will need estimates such as \cref{lem:full_turn} to hold even when we condition on the exit point of the annulus which we will prove now. The first step is to prove a conditional version of the uniform crossing estimate.

\begin{lemma}\label{lem:cross_conditioned}
Fix $0 < r < R$ and $\epsilon <(R-r)/3$. There exist constants $c = c(R/r, \epsilon/r) >0, \alpha= \alpha( R/r, \epsilon/r) > 0 $ such that if $\delta \le c r \delta_0$ and $x,y \in \C$, the following holds. Let $\tau$ be the stopping time when the random walk exits $A(x, r, R)^\d$. Let $\cR$ be a rectangle of the form $y +[0,3\eta]\times[0,\eta]$ such that $\cR \subset A(x, r+\epsilon, R-\epsilon)$ and $\eta \ge \epsilon $. Let $B_1$ and $B_2$ be balls defined as in the uniform crossing estimate, i.e $B_1 = y + B((\frac{\eta}{2}, \frac{\eta}{2}), \frac{\eta}{4})$ and $B_2 = y + B((\frac{5\eta}{2}, \frac{\eta}{2}), \frac{\eta}{4})$. For all $v \in B_1^\d$ and $u \in \partial A(x,r, R)^\d$ such that $\P_v[X_\tau = u] > 0$,
\[
\P_v[\text{$X$ hits $B_2$ before exiting $\cR$} | X_\tau = u] > \alpha .
\]
\end{lemma}
\begin{proof}
The following argument is inspired by \cite{YY}. Let $h(v) = \P_v[X_\tau = u]$. We start by giving a rough bound on $h$ restricted to $A(x, r+\epsilon, R-\epsilon)^\d$. Let us fix $v, v' \in A(x, r+\epsilon, R-\epsilon)^\d$. Since $h$ is harmonic, there exists a path $\gamma = \{v, v_1, \ldots\}$ from $v$ to $\partial (A(x,r, R)^\d)$ along which $h$ is nondecreasing. Also since $h$ is harmonic and bounded, if $\tau_\gamma$ denotes the hitting time of $\gamma \cup \partial A(x, r, R)$ by a simple random walk, we have
\[
h(v') = \E_{v'}[ h(X_{\tau_\gamma}) ]\geq h(v) \P_{v'}[X_{\tau_\gamma} \in \gamma].
\]
Using the crossing estimate a bounded number of times as in the proof of \cref{lem:full_turn}, it is clear that there exists a constant $\beta = \beta(R/r, \epsilon/r)$ independent of $\delta$ and $v'$ such that
\begin{multline*}
\P_{v'}[\text{$X$ does a full turn in $A(x, r, r+\epsilon)$ and}\\
 \text{in $A(x, R-\epsilon, R)$ before exiting $A(x, r, R)$}] \geq \beta .
\end{multline*}
We see that on the above event we have $X_{\tau_\gamma} \in \gamma$ so we have proved the Harnack inequality
\[
 \forall v, v' \in A(x, r+\epsilon, R-\epsilon)^\d,\, \beta h(v) \leq h(v') \leq (1/\beta )h(v).
\]
Now together with the Markov property and the uniform crossing estimate this gives
\begin{multline}
\P_v[\text{$X$ hits $B_2$ before exiting $\cR$ and $X_\tau = u$}] \\ \geq \P_v(\text{$X$ hits $B_2$ before exiting $\cR$}) \inf_{v' \in B_2^\d} h(v')\geq \alpha \beta h(v).
\end{multline}
Dividing by $h(v)$, the proof is complete.
\end{proof}

Using a bounded number of rectangles to surround the center $x$ in $A(x,r,R)$
as in \cref{lem:full_turn}, we get the following corollaries:
\begin{corollary}
  \label{lem:full_turn_conditioned}
Suppose we are in the setup of \cref{lem:cross_conditioned}. Let $\tau$ be the
stopping
time when the random walk exits $
A(x,r,R)^\d$. Let $v \in A(x, r+ \eps, R- \eps)$ and $u \in \partial A(x,r,R)^\d$ such that
$\P_v(X_\tau = u) > 0$. Then if $\delta < cr\delta_0$, $$\P_v( X \text{ does a full turn in }A(x,r,R) | X_\tau = u)\ge \alpha= \alpha(R/r) >0.$$
\end{corollary}

\begin{corollary}\label{lem:hitting_cond}
Fix $0 < r < R$. There exists a constant $c = c(R/r) >0$ and $\alpha(R/r)>0$ such that if $\delta \le cr\delta_0$  and $x \in \C$, $v \in A(x, r + \frac{R-r}{3}, R - \frac{R-r}{3})^\d$ and if $u$ is such that $\P_v(X_\tau = u) > 0$, where $\tau$ is the exit time of $B(x, R)$,
\[
	\P_v[\text{$X$ enters $B(x, r)$ before exiting $B(x, R)$} | X_{\tau} = u] \ge \alpha.
\]
\end{corollary}

The next lemma establishes an exponential tail for the winding of the simple random walk in an annulus conditioned to exit at a vertex, a key estimate to get an exponential tail on the winding of loop-erased random walk (\cref{prop:tail_winding2}).
Recall that we write $X[t, t']$ for the random walk path between times $t$ and $t'$ and $W(\gamma, x)$ for the topological winding of a path $\gamma$ around $x$.
\begin{lemma}\label{lem:cond_winding}
Fix $0<r <R$.
There exists
$\alpha =\alpha(R/r) \in (0,1)$ and $c$ such that for all $x \in \C$, $\delta \in (0, c r
\delta_0)$, $v \in A(x, r + \frac{R-r}{3}, R - \frac{R-r}{3})^\d$,
and for all $u$ such
that $\P_v(X_\tau = u) > 0$ where $\tau$ is the exit time of $A(x,r, R)$,
\[
\forall n \ge 1, \quad \quad \P_v\big(\sup_{
 \cY \subset X[0, \tau]}
\abs{W( \cY, x) }\geq n \big| X_\tau = u
\big) \leq C (1-\alpha)^n .
\]
where the supremum is over all continuous paths $\cY$ obtained by erasing portions from $X[0, \tau]$.
\end{lemma}

\begin{proof} The proof is technical and can be found in Lemma D.1 of the supplementary. The main idea is similar to the Harnack inequality in Lemma \ref{lem:cross_conditioned}: we show that there is a curve cutting the annulus such that every time we wind around $x$ in the annulus we hit this curve and there is then a positive probability to escape the annulus without increasing the winding number.
\end{proof}

\cref{lem:cond_winding} will allow us to control the random walk at all scales except the biggest one (the annulus around the center which intersects the complement of $D$) since in reality we stop the walk when it exits $D$; the following lemma allows us to control this largest scale.

\begin{lemma}\label{lem:winding_macro}
Fix $0<r <R$.
There exists
$\alpha =\alpha(R/r) >0$ and $c$ such that for all $x \in \C$, $\delta \in (0, c r
\delta_0)$, $v \in A(x, r + \frac{R-r}{3}, R - \frac{R-r}{3})^\d$,
writing $\tau$ for the exit time of $D\setminus B(x, r)$,
\[
\forall n \ge 1, \quad \quad \P_v\big(\sup_{
 \cY \subset X[0, \tau]}
\abs{W( \cY, x) }\geq n \big| X_\tau \in \partial D^\d
\big) \leq C (1-\alpha)^n ,
\]
and for all $u \in B(x, r)$ such that $\P( X_\tau = u ) > 0$,
\[
\forall n \ge 1, \quad \quad \P_v\big(\sup_{
 \cY \subset X[0, \tau]}
\abs{W( \cY, x) }\geq n \big| X_\tau= u
\big) \leq C (1- \alpha)^n .
\]
In both cases, the supremum is over all continuous paths $\cY$ obtained by erasing portions from $X[0, \tau]$.
\end{lemma}

Note that when we condition on exiting the domain, it is essential that we do not condition on the precise exit point. Indeed the stronger statement where we condition on this exit point does not necessarily hold if the domain itself winds many times around $x$.

\begin{proof}
Details are similar to \cref{lem:cond_winding} and can be found in Lemma D.2.. \end{proof}


Finally recalling the definition of $\gamma_z$ from \eqref{eq:general_D}, we need to bound the winding along the boundary of a domain $D^\d$ from an arbitrary marked point to $X_\tau$ where $\tau$ is the hitting time of $\partial D^\d$. Once again, as in \cref{lem:winding_macro}, it is important here to note that we can only hope for a statement that is not conditional on $X_\tau$ since a priori the boundary winding can be arbitrarily large in some places -- but these will have small harmonic measure.

\begin{lemma}\label{lem:winding_boundary}
Let $\gamma_v^\d$ denote the branch of the tree between $v$ and the marked boundary vertex $x^\d$. There exists $\alpha> 0$ and $c,C>0$ depending only on the constants in the uniform crossing conditions \cref{crossingestimate}
such that the following holds. For all $\delta <c\delta_0R$,
\[
\forall n \ge 1, \quad \quad \P_v\big( | W(\gamma^\d_v[-1, 0], v) - \E  W(\gamma^\d_v[-1, 0], v) | \geq n \big) \leq C (1-\alpha)^n .
\]
\end{lemma}
\begin{proof}
The idea is to use the replica method, i.e., to consider two independent samples of $\gamma^\d_v$ and compare their boundary winding.
Roughly, we can form a simple loop by considering the two paths after their last intersection and the portion of the boundary joining them. Since any simple loop has winding bounded by $2\pi$, the winding of the portion of the boundary between the two paths is controlled by the winding of these two paths after their last intersection, whose tails are given by \cref{lem:winding_macro}. The details are in Lemma D.3 of the supplementary.
\end{proof}

\begin{remark}
We emphasise that in general the result of \cref{lem:winding_boundary} does not hold if we remove the centering, in the same way that the result cannot hold conditionally on $\gamma_v(0)$.
\end{remark}

Recall that we aim to decompose the winding into $h^\d = h_t^\d + \eps^\d$ (see \cref{E:overview}).
We have already dealt with the limit of $h_t^\d$ as $\delta \to 0$ in the continuum part. It now remains to say that $\ve^\d$ does not contribute because when $x \neq x'$ and $t$ is large, $\ve^\d(x)$ and $\ve^\d(x')$ are nearly independent. This is proved in \cref{sec:coupling} by constructing a coupling of the sub tree around $x$ and $x'$ with independent variables. This coupling is built by sampling tree branches in the right order using Wilson algorithm and analysing carefully which part of the graph the random walks visits while performing the algorithm. In particular, a crucial step is to control the probability that the loop-erased walk from $x$ comes close to $x'$ which we now prove.

\begin{prop}
 \label{prop:discrete_coming_close}
 There exists constants $c > 0$ and $\alpha >0 $ depending only on the constants $\alpha_0, \delta_0$ in the uniform crossing assumption \cref{crossingestimate} such that the following holds.
Let $D \subset \C$ be a domain and let $u,v \in D$. Let $r = |u-v| \wedge
\dist(v,\partial D) \wedge \dist(u,\partial D)$. Let $v^\d$ be the closest vertex to $v$ in $G^\d$. Let $\gamma$ be a loop erased walk
starting from $v^\d$ until it exits $D^\d$. For all $\delta \in (0,c\delta_0]$, for all $n \le \log_4( cr\delta_0 /\delta)-1$
in $\N$,
\[
 \P(\dist (u, \gamma) <4^{-n}r) < (1-\alpha)^n.
\]
\end{prop}

\begin{proof}

We assume $|u-v|=r$ for otherwise we can wait until the simple
random walk comes closer to $u$. The idea for the proof is the following. If
the loop erased walk comes within distance $4^{-k}r$ to $u$, then after the \emph{last} time the random walk was within distance  $4^{-k}r$ to $u$, it crossed $k$ annuli without performing a full turn. The probability of this event is exponentially small in $k$ via \cref{lem:full_turn_conditioned}. Some care is needed since the above time is obviously not a stopping time.

Now we write the details.
  Let $i_{\max}(\delta) =
\lfloor \log_4(c\delta_0r/\delta) \rfloor$. Let $\{C_i\}_{0\le i \le i_{\max} }
$ denote the circle of radius $4^{-i}r$ around $u$ and define $C_{-1} =
\partial D$. We inductively define a sequence of times $\{\tau_k\}_{k \ge 0}$ as
follows. We have $\tau_0=0$. Having defined $\tau_k$ to be a time
when the random walk crosses (or hits) some circle $C_{i(k)}$, we define
$\tau_{k+1}$ to be the smallest time when
 leaving the annulus defined by $C_{i(k)-1}$ and $C_{i(k)+1}$, and
define $i(k+1)$ to be the index of the circle by which the random walk leaves the annulus. If $i (k) = i_{\max}$,
we define $\tau_{k+1}$ to be smallest time after $\tau_k$ when the walk leaves the ball defined by $C_{i_{\max} -1}$. We stop if we leave $ D$ and let $N$ be the
largest index of $\tau$ after which we stop.

Let $\cS := (X_{\tau_k})_{0\le k \le N}$ denote the sequence of crossing positions of the $C_i$.
Notice that conditioned on any realisation of $\cS$, if $k<N$, the
simple random walk between $X_{\tau_k}$ and $X_{\tau_{k+1}}$ is a simple random
walk in the annulus $A(u,C_{i(k)-1},C_{i(k)+1})$ conditioned to exit at
$X_{\tau_{k+1}}$. Furthermore by the Markov property of the walk, conditioned on $\cS$, the portions of random walk $(X[\tau_k, \tau_{k+1}])_{0 \le k \le N-1}$ are independent.

On the event $\dist( u, \gamma )<4^{-n}r$, the sequence of positions
$\cS := \{X_{\tau_k}\}_{0\le k \le N}$ contains an index when the random walk
crosses $C_n$ since the loop-erased walk is obviously a subset of the walk. Let $\kappa$ be the index of the last crossing of $C_n$ by the walk and let $\gamma'$ be the path obtained by erasing loops from $X[0, \tau_\kappa]$, i.e the ``current loop-erased path" at time $\tau_\kappa$. On the event $\dist(u, \gamma)<4^{-n}r$, necessarily the random walk did not hit $\gamma'$ after $\tau_\kappa$, otherwise the part of the path closer to $u$ would have been erased. In particular the random walk did not do any full turn after time $\tau_\kappa$. Also by construction we have $N - \kappa \geq n$.

Therefore we see that conditioned on the sequence of positions $\cS$, the event $\dist(u , \gamma)<4^{-n}r$ is included in the event that there was no full turn in the last $n$ intervals $[\tau_k, \tau_{k+1}]$. Choosing the constant $c$ to be the constant from  \cref{lem:full_turn_conditioned} associated with annuli of aspect ratio $R/r = 16$,  we can apply that result for each annulus,
since we condition on $\cS$ and for each $i\le i_{\max}$, if $r_i$ is the radius of $C_i$, we have $\delta \le c r_i\delta_0$ by our choice of $i_{\max}(\delta)$. Hence using the independence noted above, the conditional probability of this is at most $(1-\alpha)^n$ for some $\alpha >0 $, which concludes the proof.
%
\end{proof}


\subsection{Tail estimate for winding of loop-erased random walk}
\label{SS:winding}

Let $\gamma_v^\d[0,\infty)$ denote the branch of the wired UST in $D^\d$ from $v$ to the wired vertex. As in the continuum we parametrise it by its capacity plus $\log R(v,D)$, so that time 0 corresponds to being on $\partial D$, and time $\infty$ to hitting $v$. 
We
prove that $W(\gamma_v[t,t+1]^\d,v)$ has exponential tail uniformly in $\delta$ and $t$. (Note that here we do not consider the contribution coming from the winding of the boundary between the marked boundary vertex $x^\d$ and $\gamma^\d_v(0)$.)

%
%


\begin{prop}
 \label{prop:tail_winding2}
There exist constants $C, c > 0$ depending only on the constant in the uniform crossing assumption \cref{crossingestimate} such that the following holds.
 For all $v \in D^\d$, for all $t \geq 0$,
for all $\delta < c e^{-t} d(v, \partial D^\d) \delta_0 $ and $n \ge 1$,
\[
 \P\Big (\sup_{t \leq t_1, t_2 \leq t+1}\abs{W(\gamma_v[0,t_1]^\d,v) - W(\gamma_v[0,t_2]^\d,v)} >n \Big) < Ce^{-cn} .
\]
\end{prop}

We first set up a few notations; these are analogous to the ones in the proof of \cref{prop:discrete_coming_close} except that here we are considering circles of growing size.

Let $r_i = (4e)^{i-1}e^{-t}R(v, D)$ for $i \ge -1$ and
$r_{-2} = 0$. Let $C_i$ be the circle of radius $r_i$ centered at $v$ as long
as $C_i \subset D$. As soon as  $C_i $ is not a subset of $D$, define $C_i =
\partial D$ (call the maximum index $i_{\max}$ {and allow us to make small abuse of notations such as calling $C_{i_{\max}}$ a circle or writing $D \setminus B(v, r_{i_{\max}-1}) = A( r_{i_{\max}-1}, r_{i_{\max}})$}). Let $X$ be a random walk from $v$ run until it leaves the domain $D$.
Let $Y(t)$ be the loop-erasure of $X$, reparametrised by capacity seen from $v$  plus $\log R(v,D)$ (so in particular $Y(\infty) = v$). We emphasise that we are indexing circles starting from $i=-2$.

We inductively define a sequence of times $\{\tau_k\}_{k \ge 0}$ as
follows. We have $\tau_0=0$ and $i(0) = -2$, $\tau_1$ is the time the random walk crosses $C_{-1}$ and $i(1) = -1$. Having defined $\tau_k$ to be the smallest time when the random walk crosses (or hits) some circle $C_{i(k)}$, we define
$\tau_{k+1}$ to be the smallest time when it hits $C_{i(k)-1}$ or $C_{i(k)+1}$ and
define $i(k+1)$ to be the index of the circle it crosses. When $i(k) = -1$ we define $\tau_{k+1}$ only as the next crossing of $C_0$.

Let $k_{\text{exit}}$ be the first index such that $i( k ) = i_{\max}$, i.e the index corresponding to the exit from $D^\d$ and we let
 $\cS = (X_{\tau_k})_{ 0 \leq k < k_{\text{exit}} }$ the sequence of crossing positions, not including the exit position.
For any $j$, we also define $\cV_{j}$ to be the sequence of crossings of the circle $C_{j}$, i.e $\cV_{j} = \{ k | i(k) = j\}$.
We will call sets of the form $X[\tau_k, \tau_{k+1}]$ \textbf{elementary piece of random walk} and we note that conditionally on $\cS$ they are all independent. In this proof we call \textbf{crossing} a time of the form $\tau_k$.


Actually we prove something stronger than the statement in the theorem. Let $B_0$ be the disc with boundary $C_0$. We look at the portion of $Y$ \emph{outside} $B_0$ counted from the first time it enters $C_1$ until the last time it enters $C_0$ and bound the maximal winding of any sub-portion of this part of the path. Note that from Koebe's 1/4 theorem for any $t \geq 0 $, the circle $C_0$ does not intersect $\partial D$ and the times $t$ and $t+1$ happen in the interval above. 

\begin{lemma}\label{lem:tail_winding3}
There exist constants $C, c > 0$ depending only on the constant in the uniform crossing assumption \cref{crossingestimate} such that the following holds.
 For all $v \in D^\d$, for all $t \geq 0$, for all $\delta < C e^{-t}d(v, \partial D^\d) \delta_0$, let $\cY$ denote the portion of $Y$ from the first time it enters $C_1$ until the last time it enters $C_0$, then
\[
 \forall n \geq 1, \quad \quad \P\Big (\sup_{\tilde \cY \subset \cY \cap B_0^c}\abs{W(\tilde \cY,v)}>n \Big) < Ce^{-cn} .
\]
\end{lemma}

We now prove \cref{lem:tail_winding3} which immediately implies \cref{prop:tail_winding2}.
 The main idea of the proof of \cref{lem:tail_winding3} will be to show that this portion of the loop erased walk can be generated by erasing loops from a small number of elementary pieces of random walk. Then \cref{lem:cond_winding} will show that each piece does not contribute too much to the winding.

To control the number of pieces, we will use two elements. First conditionally on $\cS$ we will argue that we only need to look at the last few visits of the simple random walk to any circle because everything else was erased by a loop. Secondly we will show that the sequence $\cS$ is not too badly behaved even when we are looking close to the last visit to $C_0$. Note that this is non trivial because the last visit to $C_0$ is very far from being a stopping time. We now proceed to the actual proof, writing each of these steps as lemmas.

\medskip

We first note some deterministic facts about the loop erasure (\cref{claim:erasure}), therefore until further notice we work on a given realisation of the random walk. Let $k_{\max}-1$ be the index of the last crossing of $C_1$ by the walk (or $k_{\text{exit}}$ if $C_1 = \partial D$).
Only indices less than $k_{\max}$ will be of interest to us, as only those can contribute to the loop-erasure in the range we are considering.


  Let $\kappa_{-1}$ be the last $k$ such that $i(k) = -1$ and in the interval $[\tau_{k}, \tau_{k+1}]$ the random walk did a full turn in $A(v,r_{-1}, r_0)$. If there was no such full turn, set $\kappa_{-1} = 0$. Now we define $\kappa_i$ inductively as follows.
  \begin{itemize}
  \item The time $\kappa_i$ is the last time after $\kappa_{i-1}$ (but still before $k_{\max}$) where a full turn occurs in the annulus $A(v, r_i, r_{i+1})$.

\item   In case no such full turn occurs, $\kappa_i$ is the index of the first crossing of $C_i$ after $\kappa_{i-1}$ (but before $k_{\max}$).

\item Finally, if there is no crossing of $C_i$ between $\kappa_{i-1}$ and $k_{\max}$ define $\kappa_i = +\infty$.

    \end{itemize}

Finally we define $I = \max \{  i : \kappa_i < \infty \}$.
The idea is that a full turn erases a lot of the past. After a full turn we may still visit larger scales without doing a full turn in those scales and these might not get erased, so we will need to consider them. This is the role of the random variable $I$ which gives a bound on the largest such scale.

For every $i \leq I$, we let $\cG_i$ be the set of visits to $C_i$ after $\kappa_{i-1}$ but before $k_{\max}$, i.e, $\cG_i = \{ k \in \cV_i | \kappa_{i-1} \leq k \leq k_{\max} \}$. Observe that $\cG_i$ are ``good" indices which matter for the loop-erasure.

\begin{lemma}\label{claim:erasure}
Recall $\cY, B_0$ from \cref{lem:tail_winding3}. Then
\[
\cY \cap B_0^c \subset  \bigcup_{0\leq i \leq I} \bigcup_{k \in \cG_i} X[\tau_{k}, \tau_{k+1}].
\]
Furthermore, one can write $\cY {\cap B_0^c} = \cup_{i \le I} \cup_{k \in \cup \cG_i} \cY_{k}$ where
$\cY_k$ are disjoint intervals of the loop erased
random walk of the form $\cY_k = (Y_{j_k},Y_{j_k+1}, \ldots, Y_{j_k+i_k}) $ and $\cY_k \subset
X[\tau_k,\tau_{k+1}]$.
\end{lemma}

\begin{proof}
  The proof is by inspection, see Lemma D.4.
\end{proof}

The next step is to control the law of the size of the sets $\cG_j$, therefore we go back to considering $X$ as random.

\begin{lemma}\label{claim:size_G}
 There exist constants $C,c,c'>0$, such that for all $\delta \leq ce^{-t} \dist(v, \partial D^\d) \delta_0$, for all $n>0$,
 \[
 \P(\sum_{0 \le i \le I} \abs{\cG_i} \geq n) \leq C \exp(-c'n).
\]
\end{lemma}

\begin{proof}
This is the most delicate part of the proof of \cref{prop:tail_winding2} and is included in Lemma D.5 of the supplementary. To explain briefly:

-- it is easy to check that $I$ itself has geometric tail (conditionally on $\cS$, each $i$ such that $\kappa_i < \infty$ requires not making a full turn immediately after the last visit to $C_i$).

-- it is also immediate to see with a similar argument that the number of crossings of $C_i$ after $\tau_{\kappa_i}$ has geometric tail.

-- Therefore, to get an exponential tail on $|\cG_i|$ it remains to exclude the possibility that the walk oscillates many times between $C_i$ and $C_{i+1}$ before the next visit to $C_{i-1}$. Of course, the idea is to exploit the fact that every time the walk visits $C_i$ there is a positive chance to hit $C_{i-1}$ first rather than $C_{i+1}$. However this is technically tedious to implement since these visits are not stopping times and we can not directly condition on $\cS$ this time.
Instead we
choose to discover $\cS$ step by step by revealing only the portion of $\cS$ that is outside of $C_{i}$ and using \cref{lem:hitting_cond}.
\end{proof}

Now it is easy to complete the proof of \cref{lem:tail_winding3} using \cref{lem:cond_winding}. By \cref{claim:erasure} we can write $\cY \cap B_0^c = \bigcup_{k} \cY_k$ with for all $k$, $\cY_k \subset X[\tau_k, \tau_{k+1}]$. Therefore the winding around $v$ of any $\cY_k$ is bounded by the maximal winding difference between two times in $[\tau_k, \tau_{k+1}]$ which has uniform exponential tail by \cref{lem:cond_winding}
or \cref{lem:winding_macro} for the pieces in $\cG_{i_{\max}}$ and are independent since the walks are independent conditionally on $\cS$. Note that crucially this independence holds even when we do not condition on $\gamma^\d_v(0) = X(\tau_{k_{\text{exit}}})$, making \cref{lem:winding_macro} applicable. By \cref{claim:size_G} the number of terms in the union has exponential tail, so the proposition follows.
\qed


We now put together a consequence of the above estimates in a single lemma for ease of reference later on. Pick $v \in D^\d$ and $t \geq 0$. Recall the definition of $\gamma^\d_v[0,t]$ from \cref{prop:tail_winding2}. Now add to it a portion of the boundary $\partial D^\d$ from the marked boundary point on $\partial D^\d$ to $\gamma_v^\d(0)$
(i.e., the point where the branch hits the wired boundary). Parametrise the resulting curve as $\gamma^\d_v[-1,t]$.

\begin{lemma}\label{lem:truncated_UI}
For any $k \ge 1$, there exists a constant $A = A_k>0$ depending only on the constants in the crossing assumptions (\cref{crossingestimate}) and $k$ such that for all $\delta < ce^{-t}\dist(v, \partial D^\d)\delta_0$, we have
$$
\E\Big(\Big|W(\gamma_v^\d[-1,t),v  ) - \E [W( \gamma_v^\d[-1,t),v)] \Big|^k\Big) \le A(t + 1)^k
$$
\end{lemma}
\begin{proof}
Recall the definition of the circles $C_i$ used in the proof of \cref{lem:tail_winding3}. Note that \cref{lem:tail_winding3} provides exponential tail for the winding of $\gamma_v^\d$ from the first entry into $C_{i+1}$ until the first entry into $C_{i}$ for $i \ge i_{\max}-2$ (note that this is a continuous subportion of the loop erasure seen until last entry into $C_i$ and this portion is completely outside $C_i$), which immediately implies the same bound for the same centered random variable. Likewise, the first item of \cref{lem:winding_macro} provides exponential tail on the (centered) winding from $\gamma_v^\d$ until the first entry into $C_{i_{\max } -1}$. Finally \cref{lem:winding_boundary} provides exponential tail of the centered winding of $\gamma_v^\d[-1,0]$. Putting all of this together, we see that we need to bound the $k$th moment of a sum of $O(t+1)$ random variables with uniform exponential tails, which is elementary.
\end{proof}

\subsection{Local coupling of spanning trees}\label{sec:coupling}

Let
$z_1,\ldots,z_k \in D^\d$.  The goal of this section is to
establish a coupling between a wired uniform spanning tree $\cT^\d$ in
$D^\d$ and $k$ independent copies of full plane
spanning tree measure $\cT_i$ such that  for all $i$, there is a neighbourhood $N_i$ around $z_i$ on which $\cT^\d$ matches with $\cT_i^\d$.
The diameter of the
neighbourhoods $N_i$ are going to be random but nevertheless we will have a
good bound on the probability of the diameter being very small; typically the neighbourhoods will be of the optimal scale (i.e., not much smaller than the distance to the nearest vertex $v_j$ with $j \neq i$ or to the boundary). Note that \emph{a priori} it is
not even clear that the full plane local weak limit of a wired
spanning tree exists. For undirected graphs, the existence of this
limit follows from the theory of electrical networks \cite{pemantle}. However our setting
includes directed graphs where the electrical network theory no longer
applies. The existence of this limit will actually come out
of our coupling procedure.

The overall strategy will be to sample the spanning tree $\cT^\d$ in $D^\d$ and $(\cT_i^\d)_{1\le i \le k}$ using Wilson's algorithm. The coupling will mostly be achieved by using the same random walks for $\cT^\d$ and $(\cT_i^\d)_{1 \le i \le k}$. To achieve independence and obtain the tail estimate for the diameters of the neighbourhoods $N_i$, we will choose carefully the points from which we sample loop-erased walks and keep track of the distances from $\{z_i\}$ to the sub-tree discovered at any step. 


We start with a simple lemma regarding the hitting probability of random
walk. This is a reformulation of Lemma 2.1 from Schramm \cite{SLE} in our
setting.
\begin{lemma}
  \label{lem:SRW_hit}
  There exist constants $C, c, c' > 0$ such that
for all connected set $K \subset \C$ such that the diameter (in the metric
inherited from the Euclidean plane) of $K$ is at least $R$, for all $\delta \in (0,C \dist(v,K)\delta_0)$,
\[
\P_v(X \text{ exits $B(v,R)^\d$ before hitting $K^\d$})
\le c \left(\frac{ \dist(v,K)}{R}\right)^{c'}.
\]
\end{lemma}
\begin{proof}
  Let $C_i$ denote the circle of radius $2^{-i}$ around $v$ for $i \in
  \Z$. Consider a sequence of stopping times $\{T_k\}_{k \ge 0}$
  defined as in \cref{prop:discrete_coming_close}: if $T_k$ is the time when the
walk crosses
  $C_i$ then
  $T_{k+1}$ is the smallest time after $T_k$ when the simple random
  walk crosses $C_{i+1}$ or $C_{i-1}$. The number of circles which
  intersect $K$ is at least $c\log _2(\frac{R}{ \dist(v,K)})$ for some $c>0$. The
choice of $\delta$ is small enough for \cref{lem:full_turn} to apply for the
annuli bounded by these circles. Whenever the walk
  at $T_k$
  is in a circle $C_i$ such that both $C_{i-1}$ and $C_{i+1}$ are
  subsets of $D$ and
  intersect $K$, then the walk has probability at least $\alpha >0$ of
 performing a full turn in $A(v,2^{-i-1},2^{-i+1})$
  via \cref{lem:full_turn}. But doing such a full turn
 implies the walk must hit $K$. Hence the probability of the walk
 exiting $D^\d$ without hitting $K$ has probability at most
 $(1-\alpha)^{c'\log_2 (\frac{R}{ \dist(v,K)})}$ for some $c'>0$
which
 concludes the proof.
\end{proof}

Let $D$ be a fixed bounded domain, let $v \in D^\d$, and let $r$ be such that $B(v,r) \subset D$.
Using Wilson's algorithm, we now prescribe a way to sample the portion
of the wired
uniform spanning tree $\cT^\d$ of $D^\d$ which contains all the branches emanating from vertices in $ B(v,r/2)^\d$.
Consider
$\{\frac{r}{2} 6^{-j} \Z^2\}_{j
\ge 0}$, a sequence of scalings of the square lattice $\Z^2$ which
divides the plane into square cells.
At
step $j$, pick a vertex from each cell $f$ of $\frac{r}{2} 6^{-j} \Z^2$
which is farthest from $v$ in $B(v,\frac{r}{2}(1+2^{-j}))^\d \cap f$  (break ties arbitrarily) and is not chosen
in any previous step.
Call  $\cQ_j$ the set of vertices picked in step $j$. Now sample branches of $\cT^\d$ from each of these
vertices in any prescribed order via Wilson's algorithm, resulting in a partial tree $\cT^\d_j$. We continue
until we exhaust all the vertices in $B(v,r/2)^\d$. We call this algorithm the
\textbf{good algorithm} $GA_D^\d(r,v)$ to sample the portion of $\cT^\d$
containing all branches emanating from vertices in $B(v,r/2)^\d$ (and in
particular, the restriction of $\cT^\d$ to $B(v,r/2)^\d$). Note in particular
that $GA_D^\d(r,v)$ is sure to terminate after step $j=  \log_6( Cr/\delta)$,
where $C$ depends only on the constant appearing in the bounded density assumption
(assumption \ref{boundeddensity}).

\medskip The next lemma is similar to Schramm's finiteness theorem from \cite{SLE}. Roughly, this says that for all $\eps>0$, if we fix a $\rho$ sufficiently small depending only on $\eps$, and reveal the branches of the spanning tree at a finite number of points with density approximately $1/\rho$, then with high probability none of the remaining branches would have diameter greater than $\eps$. Schramm's finiteness theorem is originally stated for the diameter of the remaining branches of the spanning tree (which are loop-erased paths) but in fact the result holds for the underlying random walks themselves. Also the original theorem is interested in sampling the whole tree while we only want to sample $\cT^\d \cap B(v, r)$ for some $r$. Since we will need these properties later on, we write it for completeness, but the proof is exactly the
same as in \cite{SLE}.

\begin{lemma}[Schramm's finiteness theorem]\label{lem:Schramm_finiteness}
 Fix $\ve>0$ and let $D,v,r$ be as above. Then there exists a $j_0 =
j_0(\ve)$ depending solely on $\eps$ such that for all $j \ge j_0$ and all $\delta \le 6^{-j_0}
\delta_0 r$, where $\delta_0$ is as in \cref{crossingestimate}, the following holds with probability at least $1-\ve$:
\begin{itemize}
 \item The random walks emanating from all vertices in $\cQ_j$ for $j > j_0$ stay in
$B(v,r)$.
\item All the branches of $\cT^\d$ sampled from vertices in $\cQ_j
\cap B(v,r/2)$ for $j > j_0$ until they hit $\cT^\d_{j_0} \cup \partial D^\d$ have Euclidean
diameter at most $\ve r$. More precisely, the connected components of $\cT^\d \setminus \cT^\d_{j_0}$ within $B(v, r/2)$ have Euclidean diameter at most $\eps r$.
\end{itemize}

\end{lemma}


\begin{proof}
  For $ j \ge 1$, the
  number of vertices in $\cQ_j$ is at most $c6^j$ where $c$ is a universal
constant. Let $  j_{\max}:= \lfloor \log_6(\frac{C\delta_0 r}{\delta})
\rfloor $. The choice of $j_{\max}$ is such that for $j \le j_{\max}$ our
uniform crossing assumption holds and in particular we can apply
\cref{lem:SRW_hit}.
Notice each vertex in $D^\d$ is within
  Euclidean distance
  $4\cdot 6^{-j}r$ from a vertex in $\cQ_{j-1}$.
   By
  \cref{lem:SRW_hit} and the choice of $\delta$, for $j \le j_{\max}$, there exists
  a $C_0$ such that the probability that the simple random walk from a vertex in $\cQ_j$ reaches Euclidean
  distance $C_06^{-j}r$ from its starting point without hitting $\cT^\d_{j-1}$ is at most $1/2$. Notice that
$j^26^{-j} < 2^{-j}$ for all $j \in \N$ and hence using the Markov property, we
can iteratively apply the same bound for
  the walk $j^2/C_0$ times (this is the reason why in the good algorithm we
sample from balls of decreasing size at each step).  This shows that the
probability that the random walk emanating from
a vertex $w$
  in $\cQ_j$ has diameter greater than $j^2 6^{-j}r$ (call this event
$\cD(w,j)$) is at most
  $ (1/2)^{j^2/C_0}$.



\def\jm{j_{\max}}

Recall that the bounds above hold for $j \le j_{\max}$. When $j >\jm$ and $w \in \cQ_j$ we define $\cD(w,j)$ to be the event that the random walk emanating from $w$ reaches distance $\jm^2 6^{-\jm}r$ without hitting $\cT_{j}^\d$ (and in particular $\cT^\d_{\jm}$). Then observe that we still have
$\P(\cD(w,j)) \le (1/2)^{j_{\max}^2/C_0}$ in this case. Notice that the number of lattice points in
$\cup_{j >j_{\max}} \cQ_j$ is at most $B6^{j_{\max}}$ for some constant $B$
depending only on $\delta_0$ in the uniform crossing estimate assumption and
bounded density assumptions (\cref{crossingestimate,boundeddensity}).

Notice that on the complement of $$ \cD:= \cup _{j_0 \le j }\cup_{w\in \cQ_j}
  \cD(w,j) $$
  no random walk emanating from a vertex $w\in \cQ_j$ can reach distance $j^2 6^{-j} r$ from its starting point and hence stays in $B(v,r)$. Furthermore, by a union bound,
\begin{equation} \label{D2}
\P(\cD) \le \sum_{j \ge j_0}c6^j(1/2)^{j^2/C_0} <\ve
\end{equation}
for large enough choice of $j_0=j_0(\ve)$ which shows the first property of the lemma.

Also, let $\cE(w,j)$ be the event that a connection from $w$ to $\cT_{j-1}^\d$ has diameter at least $j^2 6^{-j}r$.
Observe that conditionally on $\cT_{j-1}^\d$, the probability of the event $\cE(w,j)$ does not depend on the order of the points in $\cQ_j$ and hence we may assume that $w$ is the first point in $\cQ_j$ when we compute this probability. In that case, the probability in question is at most the one we computed above, and we deduce that
 $\P( \cE(w,j) | \cT^\d_{j-1})  \le (1/2)^{(j\wedge \jm)^2/C_0}.$ On the complement of
 \begin{equation}
 \label{E}
 \cE:= \cup _{j_0 \le j }\cup_{w\in \cQ_j}\cE(w,j)
 \end{equation}
each point $w\in \cQ_j$ is connected to a point in $\cT_{j_0}^\d$ by a path of diameter at most $\sum_{j>j_0} j^2 6^{-j} r \le \eps r$ provided that $j_0$ is large enough. As $\P(\cE) \le \eps$ if $j_0$ is large enough, the proof is complete.
\end{proof}

Now we shall describe the coupling between a wired spanning tree in $D^\d$ and a full plane spanning tree around a single point, which we call \textbf{base coupling}. (The final coupling will be nothing more but an iteration of this procedure with an extra step initially which we call cutset exploration). Recall that a priori it is
not even clear that the full plane local weak limit of a wired
spanning tree exists (the existence of this limit will actually come out
of our coupling procedure).

Basically, the idea is the following. We wish to couple a uniform spanning in tree in $D^\d$ to a uniform spanning tree in $\tilde D^\d$ within a neighbourhood of some fixed vertex $v$. To do this, we first make sure that the branches from a finite number of vertices coincide for UST's in $D^\d$ and $\tilde D^\d$ in some neighbourhood, and then we apply Schramm's finiteness theorem. We now explain this in detail.

\paragraph{Base coupling} The base coupling we describe now takes
the following parameters as input: two domains $D^\d,\tilde D^\d$ and a neighbourhood
$B(v,10r)$ of a vertex $v$ such that $B(v,10r)^\d \subset D^\d \cap \tilde D^\d$. Let $\cT$ and $\tilde \cT$ denote a sample of uniform
spanning tree in $D^\d$ and $\tilde D^\d$ respectively. For any vertex $u$ in $D^\d$ (resp. $\tilde D^\d$), let
$\gamma(u)$ (resp. $\tilde \gamma(u)$) denote the branch of $\cT $ (resp. $\tilde \cT$) from $u$
to the boundary of $D^\d$ (resp. $\tilde D^\d$).
\begin{enumerate}[{(}i{)}]
\item Pick a point $u_1$ in $A^\d(v,8 r,9 r)$ and sample
  $\gamma(u_1), \tilde \gamma(u_1)$ independently (any joint law could work but we take them independent for concreteness). Let $\cE_1$ be the event that both
  $\gamma(u_1)$ and $\tilde \gamma(u_1)$ stay outside $B(v,7
  r)$ and suppose $\cE_1$ holds.

\item Let $u_2 \in A^\d(v, 2 r,3 r)$. We will use the same underlying random walk to couple $\gamma(u_2)$ and $\tilde \gamma(u_2)$. More precisely, start a simple random walk
  from $u_2$ until it is in one of $\gamma(u_1) \cup \partial D^\d$ or $\tilde \gamma(u_1) \cup \partial \tilde D^\d$ at
  time $t_1$. Suppose without loss of generality that the walk hits $\gamma(u_1) \cup \partial D^\d$ at $t_1$. Then
  we continue the walk from that point until it hits the
  other path or the boundary at time $t_2$. We then define $\gamma(u_2)$ to be the loop erased path up to time $t_1$ and
  $\tilde \gamma(u_2)$
  to be the loop erased
  path from time $0$ to $t_2$. Let $\cE_2$ be the event that $\gamma(u_2)$ and $\tilde \gamma(u_2)$ agree
in $B(v,6r)$, and suppose $\cE_2$ holds.

\item Fix a $j_0=j_0(1/2)$ as defined in
  \cref{lem:Schramm_finiteness}. Let $\cQ_j$ be a set of points in $B(v, (r/2)(1+ 2^{-j})^\d$, one in
  each cell of $6^{-j} (r/2)\Z^2 $ chosen that it is furthest away from $v$ within that cell,
as described in the good algorithm above. Let $\cE_3$ be the event that the
  branches from all the vertices in $ \cup_{j \le j_0} \cQ_j$ of $\cT$, $\tilde \cT$
  agree in $B(v,5r)$, and suppose that $\cE_3$ holds.

  \item \label{GAfail} Finally let $\cE_4$ be the event that all the branches
  from vertices in $\cup_i \cQ_i$ of $\cT$, $\tilde \cT$ agree in $B(v,r/2)$.
\end{enumerate}

In the steps above if $\cap_i \cE_i$ does not occur, we say that the
\textbf{base coupling has failed}. We think of the above process as
sampling branches one by one from the prescribed vertices. We stop
this process of sampling branches
at the first time a sampled branch causes the base coupling to fail.

\begin{lemma}\label{claim:fail_prob}
  There exist constants $0<p_0<p_0'<1$ and $c>0$ such that for all
  $\delta \le c \delta_0r$,
  \[
p_0<\P(\text{base coupling succeeds in }D^\d) <p_0'
\]
\end{lemma}

\begin{proof}
  The proof essentially follows from
  \cref{prop:discrete_coming_close}, which says that loop erased random walk does not come too close to a particular vertex, and \cref{lem:full_turn} which says that random walk makes a full turn in a given annulus with positive probability. To start with, it follows from
  \cref{prop:discrete_coming_close} (possibly replacing $4$ there by some other number) that $\cE_1$ has a positive
  probability $p_1$. Also using the crossing estimate it is easy to
  see that $\P(\cE_1) < p_0' <1$ which completes the proof of the upper bound. Moreover, independently of $\cE_1$, the walk started from $u_2$ after exiting $B(v,7r)$ has probability at least $p_2$ to make a
  full turn in $A(v,9r ,10r)$ without first hitting $B(v,6r)$ by \cref{lem:full_turn}. In particular, this implies $\cE_2$ has probability at least $p_2$, conditionally on $\cE_1$.

Now assume  $\cE_1 \cap \cE_2$ holds. Let $w \in \cQ_j$ for $j \le j_0$, and assume that revealing previous branches did not make the coupling fail. Then
  the walk started from $w$ has a positive
  probability $p_3$ to do a full turn in $A(v,3r,5r)$ before leaving $B(v, 5r)$. If this occurs then the corresponding branches $\gamma(w)$ and $\tilde \gamma(w)$ will agree in $B(v, 5r)$ (since the walk is then certain to hit at least both $\gamma(u_2)$ and $\tilde \gamma(u_2)$ in that ball). Iterating this bound over a bounded number of points (of order $6^{j_0}$) shows that, conditionally on $\cE_1 \cap \cE_2$, the probability of $\cE_3$ is uniformly bounded below.

  Finally, conditionally on $\cap_{1\le i \le 3} \cE_i$, $\cE_4$ has probability at least $1/2$ by Schramm's finiteness theorem (\cref{lem:Schramm_finiteness}), which finishes the proof.
\end{proof}

The general idea for the full coupling around one point $v$ will be that when the base coupling fails there is a not too small neighbourhood around $v$ which was not intersected by any of the paths we sampled so far. Therefore we will be able to retry the coupling in a new smaller neighbourhood. To implement this strategy we first show that if the base coupling fails, $v$ remains reasonably isolated at
the point when we stop the process with high probability.
We say a vertex $u$ (possibly different from the vertex $v$ around which we make the coupling) has
\textbf{isolation radius} $6^{-k}$ at scale $r$ at any step in the above base coupling (centered around $v$) if
$$ B(u,6^{-k}r)
\text{ does not contain any vertex from a sampled branch}.$$
We then set $I_u$ to be the minimal such $k\ge 1$ at the time when the base coupling fails.

\begin{lemma}
  \label{lem:isolation_rad}
Let $I_u$ be as above and suppose that either $u = v$ or $|u - v | \ge 10r$. Then there exist constants $\delta',c,c'>0$
such that for all $\delta \in (0,\delta'r)$ and for all $i \in (0,\log_6(\delta'r/\delta) -1)$,
\[
\P(I_u \ge i| \text{ coupling fails }) \le ce^{-c'i}
\]
\end{lemma}

\begin{proof}
  From \cref{prop:discrete_coming_close}, if one of $\cE_k$ fails for $k
  =1,2,3$, the probability that the
  isolation radius is at least $i$ is at most $C(1-\alpha)^i$ for some $C>0$ and $\alpha >0$
  (since it is the maximum of a finite number of variables
  each with exponential tail). It remains to consider is
  the branches drawn while doing the good algorithm in \cref{GAfail} of the base coupling.

Notice that the number
of vertices in  $\cQ_j$ for $j\ge j_0$ is at most $C_0 6^j$ and each of them is at
a distance at least
$r6^{-j-1}$ from $u$ (note that this holds both when $u = v$ or $|u - v | \ge 10r$). Let $\cA(i,j)$ be the
event that coupling fails in step $j\ge j_0$ and $I_u$ after this step is at
least $i$. For $i \in
(j^2,\log_6(\delta_0r/\delta) )$, the probability of $\cA(i,j)$ is at
most the probability that the branch $\gamma(w)$ from one of the vertices $w\in \cQ_j$ comes within distance $6^{-i}r$ of $u$. By  \cref{prop:discrete_coming_close} and a union bound, this is bounded by $C_0 6^j (1-\alpha)^{i-j} \le C_0 6^{\sqrt{i}}(1- \alpha)^{i - \sqrt{i}}, $ which decays exponentially even when we sum over $j$ such that $j \le \sqrt{i}$.

 Finally if $i <j^2$, we bound the
probability of $\cA(i,j)$ by using the explicit error bound \cref{D2} which we obtained in the proof of Schramm's finiteness theorem: indeed, we showed that the probability one of the branches emanating from a vertex $w \in \cQ_j$ leaves the ball of radius $ j^2 6^{-j}r$ around $w$ is less than $(1/2)^{j^2/C_0} = (1- \alpha)^{j^2}$. In particular this is also a bound on the probability that one of these branches leaves $B(v,r)$. Hence we conclude that $\P(\cA(i,j)) \le C_0 6^j (1-\alpha)^{j^2} \le C(1-\alpha)'^{j^2}$ for some $\alpha'>0$. Summing over $j\ge \sqrt{i}$ we get $\P( I_u \ge i) \le C(1-\alpha')^i$.

It remains to condition on the event that the coupling fails. But since the coupling fails with probability bounded below by \cref{claim:fail_prob}, the result follows.
\end{proof}



\paragraph{Iteration of base coupling around a single point}

We now describe how to iterate
the base coupling at different scales which is the key step to construct the full coupling. We start with a
domain $D \subset \C$ and suppose $v \in D^\d$. Suppose $r<1$
is small enough such that $B(v,10r)$ is disjoint and contained in
$D \cap \tilde D$. Fix a small constant $c$ so that \cref{claim:fail_prob} holds and
assume that $\delta \le c\delta_0r$.
\begin{enumerate}[{(}i{)}]
\item Perform a base coupling with $D^\d,\tilde D^
\d$ and $B(v,r)$. If the
  coupling succeeds, we are done.
\item If the coupling fails, let
  $6^{-I_{v,1}}r$ be the isolation radius of $v$ at scale $r$ at the step the
  coupling has failed. If $ I_{v,1} \ge
  \log_6 (c\delta_0r/\delta)$, we \textit{abort} the whole process and we say that the full coupling failed.

\item If the base coupling has failed but we haven't aborted, we move to a smaller scale. Let $\cT_{1}$ (resp. $\tilde \cT_1$) be the portion of
  the uniform spanning tree in $D^\d$ (resp. $\tilde D^\d$) sampled up to this point. We perform the base
  coupling in the domains $D^\d \setminus \cT_{1},\tilde D^\d \setminus \tilde \cT_1^\d$ in the ball
  $B(v,6^{-I_{v,1}}r)$ around $v$.


\item If the coupling fails, let $6^{-I_{v,1} - I_{v,2}}r$  be the isolation radius
  at scale $r$ around $v$ at the step the coupling has failed. Let
  $\cT_{2} \supset \cT_{1}$ (resp. $\tilde \cT_2 \supset \tilde \cT_1$) be the uniform spanning tree of $D^\d$ (resp. $\tilde D^\d$)
  sampled up to this point. If
  $ (I_{v,1}  +I_{v,2})\ge
  \log_6 (c\delta_0r/\delta)$, we abort the whole process. Otherwise we
  perform the base coupling with $D^\d \setminus \cT_{2} , \tilde
  D^\d  \setminus \tilde \cT_2$ and $B(v_1,6^{-(I_{v,1}  +I_{v,2})}r)$.
\item We continue in this way until we either abort or the base coupling succeeds at the
  $N$th iteration. If we haven't aborted the process along the way, we have
  obtained a partial tree $\cT_{N}$ which is coupled with a
  uniform spanning tree $\tilde \cT_N$ in $\tilde D^\d$ so that they are the same  in $B(v,6^{-(I_{v,1} +
    I_{v,2} +\ldots + I_{v,N-1})}r)$.
%
%
%

\item  If we abort the process at step $m$, define $N = m$ by convention.
\end{enumerate}

We call $I =  \sum_{\ell=1}^{N-1} I_{v,\ell}$ , so that $6^{-I} r$ is the isolation radius at scale $r$ when we have succeeded in coupling the trees around $v$.

\paragraph{Full coupling}
We now describe how to perform the full coupling around a fixed number of given points.
For this we introduce a new idea. We first sample all the branches from a cutset separating each of the vertices from the rest. Conditioned on these sampled branches, the neighbourhood of the vertices are now independent. We then show that the ``unexplored" neighbourhoods around the points are still big with high probability and apply our one-point iterated base coupling for each such neighbourhood.

Let us start with some notations. Let $v_1,v_2,\ldots,v_k$ be $k$ distinct points and let $r$ be chosen so that so that $B(v_i,10r) $ are disjoint and are all contained in $D \cap \tilde D$. Let $H_i$ be a set of vertices in $A(v_i,9r/2,5r)$ which disconnect $v_i$ from $\partial  D $ and $\partial \tilde D $, and let $H = \cup_i H_i$. We simply reveal the branches emanating from $H_i, 1 \le i \le k$ in some arbitrary order by Wilson's algorithm, resulting in a subgraph $\cT_{H}^\d$. We call this step a \textbf{cutset exploration}. Let $J_{ v_i}$ be the minimum $k$ such that $B(v_i, 6^{-k} r) \cap \cT_H^\d = \emptyset$. Let $J = \max_i J_{ v_i}$, and let $D^\d_i$ be the remaining unexplored domain around $v_i$, i.e., the connected component containing $v_i$ in $D^\d \setminus \cT_H^\d$. Say that we abort if $c 6^{-J} r \delta_0 \le \delta$, where $c$ is as in \cref{claim:fail_prob}.


Conditionally on $\cT_H^\d$, on the event that we haven't aborted, the component of $\cT$ containing $v_i$ is distributed as a wired uniform spanning tree in $D_i^\d$. We perform the iterated base coupling of this wired spanning tree with a uniform spanning tree $\cT^\d(i)$ of $\tilde \cD^\d$ with a base neighbourhood $B(v_i, 6^{-J}r)$. We also do this coupling around each point to obtain conditionally independent subtrees $ (\cT^\d(i))_{1 \le i \le k}$ given the cutset exploration $\cT_H^\d$. We say that we abort the full coupling either if we aborted at the cutset exploration step or if we abort in any of the iterated base couplings. Let $I_i$ be the isolation radius at scale $r$ around $v_i$ after performing the iterated base coupling around $v_i$.



\begin{thm}
  \label{lem:exp_tail}
On the event $\cA$ that we do not abort the full coupling, we obtain a coupling between
$\cT^\d$ and
independent copies of uniform spanning trees $\tilde \cT^\d(i)$ in $\tilde D^\d$ for $1 \le
j \le k$ such that
\[
\cT^\d
\cap B(v_i, 6^{- I_i} r) = \tilde \cT^\d(i)
\cap B(v_i, 6^{- I_i} r)
\]
Furthermore, there exists a universal constant $c>0$ and $C>0$ such that for all $\delta \le c \delta_0r$
and
$1\le i\le k$,
\begin{equation}
  \label{eq:32}
  \P (I_i \ge n; \cA) \le Ce^{-cn}.
\end{equation}
\end{thm}

\begin{proof}

Observe that $I_i$ is a sum of the form $\sum_{\ell=1}^{N_i-1} I_{v_i, \ell} + J_{v_i}$.
By exactly the same proof as \cref{lem:isolation_rad}, $J_{v_i}$ has an exponential tail so we concentrate on the sum. Observe that by \cref{claim:fail_prob}, $N_i$ has geometric tail (since the base coupling has probability uniformly bounded below to succeed at every step, independently of the past). Moreover, each $I_{v_j,\ell}$ has uniform exponential tail conditionally on all previous steps by \cref{lem:isolation_rad}. Hence the situation is similar to Lemma D.5 : by Markov's inequality
\begin{align*}
\P(   \sum_{\ell=1}^{N_i-1} I_{v_i, \ell} \ge n) & \le \P( N_i \ge \eps n) + P ( \sum_{\ell=1}^{\eps n} I_{v_i, \ell} \ge n)\\
& \le e^{ - c \eps n} + \E( e^{c'\sum_{\ell=1}^{\eps n} I_{v_i, \ell} } ) e^{-c' n}
\end{align*}
where $c'$ is as in \cref{lem:isolation_rad}.
Now from that lemma we see that, even when we condition on $I_{v_i, 1}, \ldots, I_{v_i, \ell_1}$,  $\E(e^{c' I_{v_i, \ell}} | I_{v_i,1}, \ldots , I_{v_i, \ell-1}) \le C_1$ for some $C_1$. Hence the right hand side above is less than
$$
e^{- c\eps n} + (C_1)^{\eps n}e^{- c' n }
$$
so by choosing $\eps$ sufficiently small this decays exponentially, as desired.
\end{proof}
The following two consequences are immediate:
\begin{corollary}
  \label{cor:abort}
  There exist constants $C,c>0$ such that
the probability of aborting the above process is at most
$C(\delta/ r\delta_0)^{c}$.
\end{corollary}


\begin{proof}
The event $\cA^c$ occurs precisely when one of the variables $I $ in step (vii) exceeds $\log_6 (c r \delta_0 /\delta)$. As above, this has exponential tail.
\end{proof}

\begin{corollary}\label{cor:existence_tree}
The wired uniform spanning tree measures in $D^\d$ has a local limit when $\tilde D \to \C$ and this limit is independent of the exhaustion taken. We call this measure the \textbf{whole plane spanning tree}. In \cref{lem:exp_tail} and all the above statements, the spanning tree measure on $\tilde D^\d$ can be replaced by a whole plane spanning tree.
\end{corollary}

\begin{proof}
For the first sentence, consider an exhaustion $D_n$ i.e., an increasing
sequence such that $\cup_n D_n = \C$. Using \cref{lem:exp_tail} along with
the control of the abortion probability from \cref{cor:abort} we conclude that
the spanning tree measures form a Cauchy sequence in total variation, therefore
it converges. For the second one, it then follows immediately from the fact
that all the statements are uniform on domains $\tilde D$.
\end{proof}


\begin{remark}\label{R:abscontI}
Suppose that $(\gamma, \tilde \gamma)$ are branches emanating from $z_0$ in $\cT$ and $\tilde \cT$ coupled by a global coupling as
above. For $i \in \mathbb{Z}$, let $T_i$ be the smallest time $t$ such that $\gamma(t) \in B(z_0,
r 6^{-i})$, and let $T = T_I$ (here we view $\gamma$ as parametrised towards $z_0$). Then given $I = i$ the law of $(\gamma_t, t \ge T)$
is absolutely continuous with respect to the unconditional law of $(\gamma_t, t
\ge T_i)$ with Radon-Nikodym derivative bounded above by $C$ for some universal
constant $C>0$ as we vary $\delta$ (indeed, it is just the law of a
loop-erased random walk conditioned on some event of uniformly positive
probability, where these events are described in the base coupling).
%
\end{remark}

\paragraph{Cutting coupled paths}
At this point, we have proved that each branch $\gamma$ of the UST has a portion where it can be coupled with a path $\tilde \gamma$ in the whole plane independent UST. Also the part far away from the starting point can be approximated by an SLE because of the convergence of loop erased random walk to SLE. It will therefore be natural to cut the branches into two parts and to use a different approximation for each piece.

A subtle issue arises here because of the choice of approximation. Indeed for the portion approximated by SLE we want to cut the path at a fixed capacity to be in the setup of \cref{sec:continuum}, while for the discrete part we want to cut $\gamma$ and $\tilde \gamma$ exactly at the same point so that their (diverging) windings cancel exactly. This is a problem because the capacity of a curve depends on the whole curve and therefore will never agree exactly between $\gamma$ and $\tilde \gamma$. Our solution to this issue is to parametrise by capacity but cut at a randomised time in both $\gamma$ and $\tilde \gamma$ in such a way that the corresponding positions match exactly. The key to doing this will be to observe that not only are the capacities of the two paths close to one another, but also their derivatives.

Recall the full coupling $(\gamma, \tilde \gamma)$ in \cref{lem:exp_tail}, where $\gamma$ is a loop-erased random walk in $D^\d$ and $\tilde \gamma$ is a loop-erased random walk in $\tilde D^\d$ starting from a vertex $v$ where $\tilde D$ is arbitrary (we can think of $\tilde D$ as the full plane). Suppose they are parametrised by capacity plus log of the conformal radius seen from $v$ in their respective domains.

\begin{lemma}\label{lemma:coupling_Poisson}
There exist constants $C,c > 0$ such that the following holds. For any $t>0$ there exists $\delta = \delta(t)$ such that for any $\delta \in (0, \delta(t))$,
we can find a pair of random variables $(X,\tilde X)$ such that individually, $X$ and $\tilde X$ are each independent of $(\gamma, \tilde\gamma )$
and
\[
\P[\gamma(t + X) = \tilde \gamma(t +  \tilde X)] \geq 1- Ce^{-ct }.
\]
Furthermore, both $X$ and $\tilde X$ are random variables which are bounded (by $1/20$).
\end{lemma}

The proof of this lemma is quite technical so we only give a short description of the ideas, see the supplementary, Section E for the details.

The idea behind the coupling is to see $\gamma(t+X)$ where $X$ is an exponential variable as the first point in a Poisson point process on $\gamma$ with an intensity given by a multiple of the derivative of the capacity over $[ t , t+ 1/20]$ (or more formally a multiple of the push forward of the Lebesgue measure by $s \mapsto \gamma(s)$ over that interval), where the multiple is itself exponentially large in $t$. This derivative measure will be almost the same in $\gamma$ and $\tilde \gamma$ with exponentially small error in $t-I$ (this is the technical part of the proof). Therefore the point processes can be coupled so that their first points are the same with high probability and this preserves some independence because of the independence inherent to a Poisson process.

\section{Convergence of the winding of uniform spanning tree to GFF}
\label{S:proof_discrete}



\subsection{Discrete winding, definitions and notations}\label{sec:discrete_wind_def}

We define the discrete winding fields in a finite domain properly here. This is completely analogous to the definition in the continuum from \cref{sec:continuum}. Let us fix a bounded domain $D \subset \C$ with a locally connected boundary and a marked point $x \in \partial D$. Using \cite{Pommerenke} Theorem 2.1, $\partial D$ is a curve. Let $\delta > 0$ and let $\cT^\d$ be a wired spanning tree of $D^\d$. Let $\gamma_v^\d$ be the branch connecting $v$ to the wired boundary in $\cT^\d$. 
As in the continuum definition, we add to each $\gamma_v^\d$ a path following $\partial D$ clockwise to $x$. More precisely, one endpoint of $\gamma_v^\d$ is some auxiliary vertex (a point on the continuum boundary $\partial D$, see \cref{sec:regularized_estimate}). We add to $\gamma^\d_v$ the \emph{continuum} curve joining this vertex and the marked boundary point $x$ in the clockwise direction. For simplicity of notation we still call this path $\gamma_v^\d$.

 We parametrise the part of the curve in $\partial D$ by $[-1, 0]$ and the rest by capacity in $D$ plus $\log R(z,D)$. Observe that this definition is analogous to the continuum definition in \cref{sec:continuum}, so that $t=0$ will always correspond to hitting the boundary. Motivated by the formula connecting intrinsic and topological winding (cf. \cref{lem:wind_start}), we define, in the smooth case first:
\begin{equation}\label{E:hd}
h^\d(v) := W_{\i} (\gamma^\d_v[-1, \infty));
\end{equation}
and
\begin{equation}
h^\d_t (v) :  = W(\gamma^\d_{v}[-1,t],z) - \Arg(-(\gamma^\d_v)'(-1)) + \arg_{D;x}(v).
\end{equation}
If $D$ is not smooth near $x$, then we can still define $h^\d(v)$ via \eqref{E:hd} up to a global constant,
\begin{equation}
h^\d_t (v) :  = W(\gamma^\d_{v}[-1,t],z) + \arg_{D;x}(v)
\end{equation}
where $\arg_{D;x}(v)$ is defined up to a global constant.

\def\fh{ \tilde h}

\smallskip Consider a full plane discrete UST on $G^\d$, $\tilde \cT = (\tilde  \gamma_v )_{v \in G^\d}$. (We write $\tilde \gamma_v$ instead of $\tilde \gamma^\d_v$ for simplicity).
We parametrise the paths $ \tilde \gamma_v$ by full plane capacity plus $\log R(v, D)$, going from $-\infty$ far away to $+\infty$ at $v$ (we need to add $\log R(v, D)$ to match its parametrisation with the parametrisation of $\gamma_v$ as much as possible).
We extend the definition of the regularised winding to that setting by defining $\fh_T(v) - \fh_S(v) := W(\tilde \gamma_v(S, T), v)$ and $\fh(v) - \fh_T(v) := W(\tilde \gamma_v(T, \infty), v)$. Note that the definition of the increments (in $T$) of $\fh$ make sense even though we cannot define $\fh$ pointwise, so this definition is a slight abuse of notation. Finally
\begin{equation}\label{def:m}
m^\d(v) := \E[\fh^\d(v) - \fh^\d_{\log R(v,D)}(v)]
\end{equation}
and note that random variable $$\fh^\d(v) - \fh^\d_{\log R(v,D)}(v) = W ( \tilde \gamma_v(\log(R(v,D)), \infty); v)$$ is well defined (and its expectation is well defined too) and that it does not depend on $D$ since $\tilde{\gamma} ( \log R(v, D))$ is just the point of full plane capacity $0$. We point out that as $\delta \to 0$, the path $\tilde \gamma_v$ converges
(uniformly on compacts) towards a full plane SLE$_2$; this follows by a
combination of results from Lawler, Schramm and Werner; Yadin and
Yehudayoff \cite{LSW, YY} as well as Field and Lawler \cite{FieldLawler}. As a
consequence, the arbitrary choice of truncating at capacity $0$ is irrelevant:
this is because  the asymptotics of $m^\d(v)$ as $\delta \to 0$ is independent
of the choice of truncation (indeed, for a full plane SLE path, the expected
winding between capacity 0 and 1 is zero by symmetry). Readers who are
uncomfortable with the notion of full plane SLE can replace the full plane by a
disc of some large radius in the definition of $m^\d$; in which case the above
remark relies just on the convergence result of Lawler, Schramm and Werner
\cite{LSW} and Yadin and Yehudayoff \cite{YY}.

To help with the intuition, recall from the introduction that we need to remove by hand microscopic contributions to the expected winding. This is the purpose of $m^\d$. Subtracting $m^\d$ also allows us to take into account the possible contribution to winding coming from intermediate (mesoscopic) scales.
\subsection{Statement of the main result}
\label{S:main_winding_full}
We now state the main theorem in this section, which is a stronger version of  \cref{T:winding_intro} in the introduction. Since we are going to integrate
the discrete winding field against test functions, we need to make precise what
we mean by this. There are two natural choices to do this integral: one which
takes into account the geometry of the underlying graph and another one which
accounts only for the ambient Euclidean space in which the graph is embedded. The latter turns out to be slightly more natural since the limit in that case is a (conformally invariant) Gaussian free field, i.e., does not depend on the limiting density of vertices.

We proceed as follows. Given $h^\d$ defined on the vertices of the graph, we can extend $h^\d$ to a function on the whole domain $D$ using
various forms of interpolation. One way to do this is to linearly extend the
value of $h^\d$ to the edges and then take a harmonic extension on the faces (this includes the outer face, and then we restrict this extension to $D$).
However for concreteness we will look at the following extension: consider the Voronoi tesselation of $D$ defined by the vertices
of the graph. We then define the extension $h^\d_{\text{ext}}$ to be constant on each Voronoi cell, equal to $h^\d (v)$ where $v$ is the centre of the cell. This allows us to use the regular $L^2$ product to integrate $h^\d$ against test functions.
\[
 (h^\d,f) := \int_D h^\d_{\text{ext}}(z) f(z) dz.
\]
This extension procedure can also be applied to $m^\d$, leading to a function defined on all of $D$. We then have the following theorem.

\begin{thm}\label{thm:main_details}
Let $G$ be a graph satisfying the assumptions of \cref{sec:assumption}, let $D\subset \C$ be a simply connected domain with a locally connected boundary and a marked point $x \in \partial D$. Let $\cT$ be a uniform spanning tree of $D^\d$ and let $h^\d(v)$ denote the intrinsic winding as above, and let $m^\d$ be defined as above.
Then
\[
	h^{\d} - m^\d \xrightarrow[\delta \to 0]{} \hg.
\]
The convergence is in law in the Sobolev space $H^{-1-\eta}(D)$ for any $\eta>0$. The limit $\hg$ is a free field with intrinsic winding boundary conditions, that is, $\hg = (1/\chi) h + \pi/2$ where $ h =   \hg^0  + \chi u_{(D,x)}$. Here $\hg^0$ is a standard GFF with Dirichlet boundary conditions and $u_{(D,x)}$ is defined as in \cref{E:uDx} and \cref{R:rough_everywhere}.

Moreover, for all $n\ge 1$, for all $f_1, \ldots, f_n \in H^{1+ \eta}$, and for all positive reals $k_1, \ldots, k_n $ we have
\[
  \E \prod_i (h^{\d} - m^\d, f_i)^{k_i} \xrightarrow[\delta \to 0]{} \E \prod_i (\hg , f_i)^{k_i}
\]
and for all $k\ge 1$
\[
 \E  \left(\norm{h^{\d} - m^\d}^k_{H^{-1-\eta}} \right)  \xrightarrow[\delta \to 0]{}  \E \left( \norm{\hg }^k_{H^{-1- \eta}}\right) .
\]
\end{thm}

\begin{remark}
We emphasise that the function $m^\d$ is a deterministic function which depends only on the point in the graph, and in particular does not depend on the domain $D$. Note also that it follows from this result that $\E(h^\d - m^\d) \to \E(\hg)$ and hence we deduce
$$
h^\d - \E h^\d \to\frac1{\chi} \hg^0
$$
in the same sense as above, where $\hg^0$ is a Gaussian free field with Dirichlet boundary conditions.
\end{remark}

\subsection{Other notions of integration}
We now comment briefly on other possible definitions of integration against test functions. Another definition which is a priori natural is to consider
$$
(h^\d - m^\d,f)_{\d} : = \frac{1}{\mu^\d(D)}\sum_v (h^\d(v) - m^\d) f(v).
$$
In that case, $h^\d - m^\d$ is viewed as a random measure which is a sum of weighted Dirac masses.
We can first ask about convergence of this object as a stochastic process indexed by test functions (see e.g. \cite[Definition 1.10]{NBnotes}).
It can be checked that if the uniform distribution on the vertices of the graph converges to a measure
$\mu$ in $\C$, we have that
$$
(h^\d - m^\d,f)_{\d} \xrightarrow[\delta \to 0]{} \hg^\mu
$$
where now $\hg^\mu$ is a Gaussian stochastic process indexed by test functions such that $(\hg^\mu, \phi) = (\hg, \phi \frac{d\mu}{d\Leb})$ where $\Leb$ denotes the Lebesgue measure.

Note that in most cases, for example in any periodic lattice or isoradial graphs or T-graphs where
our results apply, $\mu$ is just the Lebesgue measure. In that case, note that $h^\d$ lies a priori within $H^{-1 - \eps}$ for any $\eps>0$ (this is the Sobolev regularity of any Dirac mass) and it is easy to check that our proof implies convergence of $h^\d - m^\d$ towards $\hg$ in the stronger sense of Sobolev spaces $H^{-1 - \eta}(D)$ for any $\eta$, as in Theorem \ref{thm:main_details}.

However there are also
interesting examples of graphs where the convergence of random walk to a (time-changed) Brownian motion holds but $\mu$ is different from Lebesgue measure. An exotic example of such a situation is
a conformally embedded random planar map where such a convergence is expected to hold and the measure $\mu$
is a variant of Gaussian multiplicative chaos (see \cite{GarbanBourbaki} and \cite{LBM13, ber2013diffusion} for an introduction to this topic).

\subsection{Proof of the main result}

Now we collect the results of the two previous sections to prove \cref{thm:main_details}.

\begin{lemma}
Fix a domain $D \subset \C$ with locally connected boundary and let $x_1,\ldots, x_k \in D$. For all $v^\delta_1, \ldots, v^\delta_k \in D^{\d}$ converging to $x_1, \ldots, x_k$, for all $T_1, \ldots, T_k$,
\[
	\E\left[ \prod_i h^{\d}_{T_i} (v_i) \right] \xrightarrow[\delta \to 0]{} \E \left(\prod_i h_{T_i}(x_i) \right).
\]
where $h_T$ is the regularised winding field of a continuum UST in $D$ as defined in \cref{E:uDx,R:rough_everywhere}.
\end{lemma}
\begin{proof}
By our assumptions (see \cref{R:conv_lerw}) and by Wilson's algorithm, the paths $(\gamma_{v_i}^\d)_{i=1}^k$ converge to $(\gamma_{x_i})_{i=1}^k$ where the $\gamma_{x_i}$ are the paths connecting $x_i$ to $\partial D$ in a continuous UST. Furthermore, observe that the function $h_T (v) $ is a continuous function of $\gamma_v^\d$ (this is because the topological winding up to capacity $t + \log(R(v,D))$ is continuous in the curve). Hence
$ \prod h^{\d}_T (v_i)$ converges in distribution to $\prod h_T(x_i)$.
Using \cref{lem:truncated_UI}, we see that it is also uniformly integrable and hence the expectation converges.
\end{proof}

Combining the above lemma with \cref{T:winding_continuum} and \cref{lem:quantitative_two_pt}, we can find a sequence $T(\delta)$ (depending on the $v_i$'s) going to infinity slowly enough such that whenever $T(\delta) \le T_i(\delta) \le T(\delta) + 1/20$,
\[
\E\left[ \prod h^{\d}_{T_i(\delta)} (v_i) \right] \xrightarrow[\delta \to 0]{} \E\left( \prod \hg(x_i) \right).
\]
Using the previous lemma with the full coupling of \cref{lem:exp_tail}, we can control the $k$-point function for the winding up to the endpoint, which is the key step in the proof of \cref{thm:main_details}.
\begin{prop} \label{P:pointwisemoments}
For all $k$, for all bounded domains $D$ with locally connected boundary, for all $v^\delta_1, \ldots, v^\delta_k \in D$ converging to $x_1, \ldots, x_k$,
\[
	\E\left[ \prod_i ( h^{\d} (v_i^\delta) - m^\d(v_i^\delta) )\right]  \xrightarrow[\delta \to 0]{} \E \left(\prod_i \hg(x_i)\right).
\]
Recall that the right hand side is a notation for the k-point function of a GFF with winding boundary condition, as in \cref{thm:main_details}.
\end{prop}

\begin{proof}
By definition of our extension of $h^\d$ to $D$, we may assume without loss of generality that $v_i^\delta \in D^\d$ (this is one of the advantages of working with the Voronoi extension of $h^\d$ to $D$).
We write $\gamma_i$ for $\gamma^\d_{v_i}$. If $\delta$ is small enough, one can apply the coupling of \cref{lem:exp_tail}. Focusing only on the paths from the $v_i$, we obtain random variables $I_1, \ldots, I_k$, all with exponential tails, and independent full plane loop-erased paths $\tilde\gamma_1, \ldots, \tilde\gamma_k$ such that
\[
 \forall j, \quad \quad  \gamma_{j} \cap B(v_j, 6^{-I_j} r) = \tilde\gamma_j \cap B(v_j, 6^{-I_j} r)
\]
on the event that the coupling succeeds, where $r =(1/10) ( \min_{i \neq j} |v_i - v_j| \wedge \min_{i } d(v_i, \partial D) \wedge 1)$. Here $r  $ is a constant and $\delta \to 0$ so we will not worry about $r$ or the offset in the parametrisation (which recall is capacity plus $\log R(v_i, D)$). Let $\tilde h^\d$ be the associated field defined as in \cref{def:m} (recall that only its increments in $T$ are defined).
Let $T(\delta)$ be some sequence such that for any $T_j$ such that $T(\delta) \le  T_j \le T(\delta) +  1/20$ we have
\begin{align}
& \E[ \prod h^{\d}_{ T_j} (v_j) ]\xrightarrow[\delta \to 0]{}  \E \prod \hg(x_j), \label{GFFTj}\\
 & \E[\tilde h^\d_{T(\delta)}(v_j) - \tilde h_{\log (R(v,D))}^\d(v_j)] \xrightarrow[\delta \to 0]{} 0.\label{eq:C}
 \end{align}
Recall that
\cref{lemma:coupling_Poisson} holds for $t = T(\delta)$. We can further modify our choice of $T(\delta)$ if needed so that $$\P( \gamma_{i}(T(\delta), \infty) \not\subset B(v_i, e^{- T(\delta)/2} )) \le ce^{- cT(\delta)} $$
and the same statement holds with $\tilde \gamma_{i}$ instead of $\gamma_{i}$.
This is possible because $T(\delta)$ can be chosen to grow to infinity arbitrarily slowly and these estimates hold in the continuum. Indeed, $\gamma_{j}$ converges to radial SLE$_2$ for which we can apply \eqref{eq:Schramm_dist} and $\tilde \gamma_{j}$ converges to whole plane SLE$_2$ which is symmetric with respect to conjugation hence has zero expected winding.

For each $1\le j\le k $ we use \cref{lemma:coupling_Poisson} and we write $T_j = T(\delta) + X_j$ and $\tilde T_j = T(\delta) + \tilde X_j$ the resulting times. Let $\cC$ denote the $\sigma$-algebra generated by the cutset exploration in Theorem \ref{lem:exp_tail} as well as $\tilde T_j$, $1\le j \le k$. Let $G$ be the good event that $\tilde \gamma_i (\tilde T_i, \infty) = \gamma_i (T_i, \infty) \subset B(v_i, r e^{-I_i})$ for all $1\le i \le k$. In other words, $G$ is the event that $\gamma_i(T_i) = \tilde \gamma_i (\tilde T_i)$ and $T_i \ge \Lambda_i$ where $\Lambda_i$ is the last time at which $\gamma_i$ enters $B(v_i, re^{-I_i})$. Unfortunately $G$ is not measurable with respect to $\cC$, but fortunately we will see that its complement has exponential small probability in $T(\delta)$.

 For clarity we remove the superscripts $\cdot^\d$ from notations, hence for instance $h(v_i)$ means $h^\d (v_i)$. Note that on the good event $G$, since $T_i\ge  \Lambda_i$ and $\tilde T_i \ge \tilde \Lambda_i $ (where $\tilde \Lambda_i$ is defined in the obvious analogue way to $\Lambda_i$), and since $\gamma$ and $\tilde \gamma$ also agree on $B(v_i, r e^{- I_i})$, we can write
\begin{equation}\label{dec:simple}
h(v_i) = h_{T_i} (v_i)  + \Delta \tilde h(v_i),
\end{equation}
where $ \Delta \tilde h(v_i) = \tilde h (v_i) - \tilde h_{\tilde T_i} (v_i)$.
The first term is the ``main" term (well described by SLE), see \eqref{GFFTj}. The second term is a mesoscopic term: when we substract $m(v_i)$ we will get independent terms with mean approximately zero.

We wish to take the conditional expectation given $\cC$, but since $G$ is not measurable with respect to $\cC$, some care is needed. We introduce the following notations: if $s \le t $ we write $h(s,t)$ for the winding around $v_i$ of $\gamma_{i}$ during $[s,t]$; i.e., $h(s,t) = W ( \gamma_i([s,t]), v_i)$. By an abuse of notation, we write $h(v_i) = h(-1, \infty)$. When $s \ge t $ we put $h(s,t) = - h(t,s)$.
We then write
$$
h(- 1, \infty) = h (-1, T_i) + h(T_i, \Lambda_i) + h(\Lambda_i, \infty).
$$
As before, the main term is $h(- 1, T_i) = h_{T_i} (v_i)$ which is thus governed by \eqref{GFFTj}. Note that we also have $ h(\Lambda_i, \infty) = \tilde h ( \tilde \Lambda_i, \infty)$
(this is also trivially true even when the coupling fails).

Moreover, the middle term $h(T_i, \Lambda_i)$ may be rewritten by adding and taking away $\tilde h ( \tilde T_i, \tilde \Lambda_i)$ as
$$
h(T_i, \Lambda_i) =   \tilde h (\tilde T_i, \tilde \Lambda_i) + \xi_i
$$
where
$$
\xi_i =  h ( T_i, \Lambda_i ) - \tilde h ( \tilde T_i, \tilde \Lambda_i)
$$
is an error that is typically zero except on an event of very small probability.

Consequently, we obtain the decomposition:
$$
h(v_i) - m_i = \underbrace{h_{T_i}(v_i)}_{\text{type 1}} + \underbrace{[\tilde h ( \tilde T_i, \infty) - m _i]}_{\text{type 2}} + \underbrace{\xi_i}_{\text{type 3}}.
$$
We now explain how we will finish the proof. We need to compute the expected value of the product over $i$ of the expression in the left hand side above. We expand the product in the right hand side above to get a finite sum of products of terms involving one of the three types of terms above for each $1\le i \le k$. We take the conditional expectation given $\cC$ and then the total expectation. If only the first type of terms occur in the sum, we can simply use \eqref{GFFTj} as already mentioned. To finish the proof, we simply make the following observations:
\begin{itemize}

\item Terms of type two, $B_i = [\tilde h ( \tilde T_i, \infty) - m_i] $, satisfy $\E(B_i) = o(1)$ and $B_i$ is independent of $\cC$ and of any of the terms in the product given $\cC$.

\item Since $\xi_i = 0$ except if $G$ does not hold, we have $\P( \xi_i \neq 0) \le C e^{- c T}$. Indeed,
    \begin{align*}
\P( G^c) &\le  \P( \gamma(T_i) \neq \tilde \gamma_i( \tilde T_i)) + \P( I_i \ge T(\delta)/2 - \log r+ \log R(v_i, D) )\\
&\ \  +  \P\left( I_i \le T(\delta)/2 - \log r + \log R(v_i, D) \text{ and } \gamma_i( T(\delta), \infty) \not\subset B(v_i,  e^{-T(\delta)/2})\right) \\
& \le C e^{- cT(\delta)}
\end{align*}
 where we have used Lemma \ref{lemma:coupling_Poisson} for the first term, \cref{lem:exp_tail} for the second (and the fact that $r$ is fixed), and \eqref{eq:Schramm_dist} and the choice of $T(\delta)$ for the third.
 Hence, using \cref{lem:truncated_UI},
    $$\E( \xi_i^k) \le  ce^{-cT} \E( T_i^k + \tilde T_i^k) \le C e^{-cT}.$$


 \item Moreover, $\E( h_{T_i}^k)  \le C (1+T^k)$ by \cref{lem:truncated_UI}.
\end{itemize}

When we take the conditional expectation, all the terms of the type $2$ contribute $\E(B_i |\cC) = \E(B_i) = o(1)$ to the product since they are independent of any other term in the product. Hence if the product contains only terms of types $1$ and $2$ then this contributes $ o(1) \E( \prod h_{T_i}(v_i) ) = o(1)$ by \eqref{GFFTj}.

Otherwise if the product contains any term involving $\xi_i$ (i.e. type 3), using H\"older's inequality and \eqref{GFFTj}, this contributes at most $O (e^{- cT}) = o(1)$. Hence the result is proved.
\end{proof}

%

The above proposition gives a pointwise convergence of the $k$-point function. We now need some a priori bounds to allow us to integrate these moments against test functions via the dominated convergence theorem.

\begin{lemma}\label{L:domconv}
For all $k \geq 2$, for all bounded domains $D$ with locally connected boundary, there exist constants $C=C_k,c>0$ such that for
all $\delta < c\delta_0$, for all $v^\d_1, \ldots, v^\d_k \in D^{\d}$,
\[
   \left|\E\big [ \prod \big( h^{\d} (v_i^\d) -
m^\d(v_i^\d)\big) \big] \right| \leq C  (1+ \log^{2k} r  )
\]
where $r   =  (1/10) ( \min_{i \neq j} |v_i^\d - v_j^\d| \wedge  \min_j
d(v_j^\d, \partial D) \wedge 1) $. The same holds even if we replace $v_i^\d$
by any point in its Voronoi cell as in our extended
definition of $h_{ext}^\d (z)$.
\end{lemma}

\begin{proof}
Let us assume $r \ge \delta$ for now and consider the full coupling of
\cref{lem:exp_tail}. We use the notations from the proof of \cref{P:pointwisemoments}. We exploit the following decomposition which is analogous to the decomposition in \cref{P:pointwisemoments} except we do away with $T_i$ (indeed, if the points $v_i$ are really close to each other, $e^{-T_i}$ might be much greater that this distance):
\begin{equation}
h(v_i) - m_i = h(-1, \Lambda_i) + \tilde h (\tilde \Lambda_i,\log R(v, D)) + [\tilde h(\log R(v, D),\infty) - m_i].
\end{equation}
Let $\cC$ be the sigma algebra generated by the cutset exploration. Note that conditionally on $\cC$, the third term is independent of any of the above terms involving $j \neq i$ and has 0 mean. Thus we can ignore the terms in the expansion of the product which has at least one term of the third type.

We now provide an argument on how to bound the first term and the same argument can be used to bound the second term by the same quantity. Let $\Phi_i$ be the time of first entry of $\gamma_i$ into $B(v_i, re^{-I_i})$ (in contrast with $\Lambda_i$ which is the last entry into  $B(v_i, re^{-I_i})$).
Note that
$$
h( - 1, \Lambda_i) = h ( - 1, \Phi_i) + h ( \Phi_i, \Lambda_i)
$$
and the second term is deterministically bounded above in absolute value by $2\pi$ for elementary topological reasons (essentially, the winding number of a Jordan curve is either 0 or $\pm 2\pi$). Furthermore, note that
$$|h ( - 1, \Phi_i)| \le  \sum_{j=-1 }^{\lfloor \Phi_i \rfloor -1  } |h (j, j +1) | +| h ( \lfloor \Phi_i \rfloor, \Phi_i)|.  $$
Now each of these terms have exponential tail by \cref{prop:tail_winding2}. Moreover, $ \Phi_i  -\log R(v, D) \le  - \log r + I_i  $ by monotonicity of conformal radius, and $I_i$ has exponential tails by \cref{eq:32} in \cref{lem:exp_tail}, hence by convexity of the function $x \mapsto x^k$,
\begin{align*}
\E ( |h (- 1, \Phi_i)|^k) & \le
 \E \Big( ( - \log r + I_i + \log R(v_i, D) )^{k-1} \\
 & \quad \quad\times
 \sum_{j= -1}^{ - \log r + I_i + \log R(v_i, D)} |h (j, j + 1)|^k \Big) + C\\
& \le C ( -  \log r + 1)^{2k}  + C
\end{align*}
by Cauchy--Schwarz, as desired (the above bound is not optimal, but this is unimportant).


Finally if $r<\delta$, we can use H\"older to bound the moment of $h^\d -
m^\d$ by $C(1+\log^{2k}\delta)$ as above which is at most the required bound.
\end{proof}

%
We can now prove our main theorem.
\begin{proof}[Proof of \cref{thm:main_details}]
Fix $f_1, \ldots, f_n$ to be smooth functions in $\bar D$ and $k_1, \ldots, k_n \ge 1$. Combining \cref{P:pointwisemoments} and \cref{L:domconv}, we see that we can apply the dominated convergence theorem to $ \E \left[ \prod_{i=1}^n ( h^\d - m^\d , f_i)^{k_i} \right] $ and therefore
$$
\E \left[ \prod_{i=1}^n ( h^\d - m^\d , f_i)^{k_i} \right] \to \E \left[ \prod_{i=1}^n ( \hg, f_i)^{k_i}\right].
$$
Since the right hand side is Gaussian (and therefore moments characterise the distribution), $ (h^\d - m^\d , f_i)_{i=1}^n $ converges in distribution to $ (\hg, f_i)_{i=1}^n$. In other words, at this point we already know $h^\d - m^\d$ converges to $\hg$ in the sense of finite dimensional marginals (when viewed as a stochastic process indexed by smooth functions with compact support, say). We now check tightness in the Sobolev space $H^{-1- \eta}$, from which convergence in $H^{-1 - \eta}$ will follow.

Note that by the Rellich--Kondrachov embedding theorem, to get tightness in $H^{-1- \eta}$ it suffices to prove that $\E(\|h^\d - m^\d\|_{H^{-1 - \eta'}}^2) < C$ for some constant $C$, for any $\eta'< \eta$.  More generally we will check that $\E(\|h^\d - m^\d\|_{H^{-1 - \eta}}^{2k}) < C_k$ for any $k \ge 1$ and any $\eta>0$.

Let $(e_j)$ be an orthonormal basis of eigenfunctions of $- \Delta$ in $L^2(D)$, corresponding to eigenvalues $\lambda_j > 0$. Then writing $h = h^\d - m^\d$ for convenience,
\begin{align*}
 \E(      \| h \|^{2k}_{H^{-1 - \eta}})
 &=  \E  \Big(  \sum_{j=1}^\infty  (h, e_j)_{L^2}^2  \lambda_j^{-1- \eta} \Big)^k \le C \sum_{j=1}^\infty \E( (h, e_j)^{2k}_{L^2} )   \lambda_j^{-1- \eta}
\end{align*}
by Fubini's theorem and Jensen's inequality (since by Weyl's law, $ \sum_j \lambda_j^{-1- \eta} < \infty$ is summable).

Furthermore,
\begin{align*}
 \E( (h, e_j)_{L^2}^{2k} ) &= \int_{D^{2k}} \E( h (z_1) \ldots h (z_{2k}) ) e_j(z_1)\ldots  e_j(z_{2k}) dz_1 \ldots dz_{2k}
\end{align*}
and note that by \cref{L:domconv},
$ \E( h(z_1)\ldots  h(z_{2k}) ) \le C ( 1+ \log^{4k} r)$ where $r = r(z_1, \ldots, z_{2k})$ is as in that lemma. Note also that $( \log r)^{a}$ is integrable for any $a>0$ and $D$ is bounded hence using Cauchy--Schwarz, we deduce that $\E( (h, e_j)_{L^2}^{2k} )  \le C (\int_{D^{2k}} e_j(z_1)^2\ldots  e_j(z_{2k})^2 dz_1\ldots dz_{2k} )^{1/2}= C$ since $e_j$ is orthonormal in $L^2$. Consequently,
 $$
 \E( \| h \|_{H^{-1 - \eta}}^{2k}) \le \sum_{j=1}^{\infty} C  \lambda_j^{-1- \eta}  < \infty
 $$
 by Weyl's law. This finishes the proof of \cref{thm:main_details} and hence \cref{T:winding_intro}. Let us remind the reader here that the proofs of moment bounds in $H^{-1-\eta}$ follows through in exactly the same way in the continuum proof in \cref{P:conv_GFF_H}.
\end{proof}

\section{Joint Convergence}\label{sec:jt_conv}

In this section we prove the joint convergence of (dimer height function, wired UST) to (GFF, continuum wired UST) where the latter is coupled together through the imaginary geometry coupling in \cref{thm:coupling_intro}.

Let us first introduce the setup. The topology on the height function is the Sobolev space $H^{-1-\eta}$ (recall this is a complete, separable Hilbert space). The topology on the tree is the Schramm topology $\Omega_1$ introduced in \cite{SLE} (and described in the supplementary). As usual we have a bounded domain $(D,x)$ with a marked point $x \in \partial D$ and locally connected boundary. We work with the space $\Omega:= \Omega_1 \times H^{-1-\eta}(D)$ equipped with the product topology. We also view $\Omega$ also as a metric space with metric defined by $d_1+d_2$ where $d_1$ and $d_2$ are the metrics in each coordinate.
Let $\cT^d,h^\d$ be as in \cref{S:main_winding_full}. Let $(\cT,\hg(\cT))$ denote the continuum wired UST in $D$ with $\hg(\cT)$ denoting the GFF which is coupled with $\cT$ using the imaginary geometry coupling (cf. \cref{thm:coupling_intro}). The point here is again that the height function is not continuous as a function of the discrete tree, hence we have to use the results about the truncated winding we proved in this paper.
\begin{thm}\label{thm:jt_conv}
In the above setup, we have the following joint convergence in law in the product topology described above:
\[
(\cT^\d,h^\d - m^\d) \xrightarrow[\delta \to 0]{(d)} (\cT,\hg(\cT)).
\]
\end{thm}
\begin{proof}
To simplify notations, we write $h^\d$ for $h^\d - m^\d$ admitting a slight abuse of notation.
Let $h_t$ denote the continuum winding truncated at capacity $t $ plus log conformal radius seen from the point as before.
Notice that from \cref{thm:pointwise_GFF},
$
(\cT,h_t) \xrightarrow[t \to \infty]{P} (\cT,\hg(\cT)),
$
where the convergence is in probability in the metric space defined above.
Note that for any fixed $t$, $h^\d_t$ and $h_t$ are obtained by applying the same continuous function to respectively $\cT^\d$ and $\cT$.
Hence we have
\begin{equation*}
(\cT^\d,h_t^\d) \xrightarrow[\delta \to 0]{(d)} (\cT,h_t).
\end{equation*}
Thus there exists a sequence $t(\delta)$ growing slow enough  such that
\begin{equation*}
(\cT^\d,h_{t(\delta)}^\d) \xrightarrow[\delta \to 0]{(d)} (\cT,\hg(\cT)).
\end{equation*}
Now recall that we proved in \cref{P:pointwisemoments,L:domconv} that $\E(\|h^\d_{t(\delta)} - h^\d \|_{H^{{-1-\eta}}}) \to 0$ which implies that $h^\d_{t(\delta)} - h^\d $ converges to $0$ in probability. Hence the result follows from Slutsky's lemma (e.g., Theorem 3.1 in \cite{Billingsley}).
%
\end{proof}


  \bibliographystyle{abbrv}
\bibliography{winding.bib}

\newpage

\appendix

\renewcommand{\appendixname}{}




\title{Supplementary material for Dimers and Imaginary Geometry}

\begin{frontmatter}

\runtitle{Dimers and Imaginary geometry: supplement}

%
%
%
%
%
%
%

\end{frontmatter}

\section{Background Materials}\label{sec:background_supp}\
Here we recall some background materials mentioned in Section 2 of the main file.
\subsection{Schramm--Loewner evolution}\label{sec:SLE_supp}
SLE, or Schramm--Loewner evolution, is a family of conformally
invariant random curves in the unit disc $\D$ which are supposed to describe
several aspects of statistical physics models. In this paper we are concerned
with the radial version of SLE, see \cite{Lawler} for more details. SLE$_\kappa$ with $\kappa \le 4$ in the unit disc $\D$ starting from $1$ and targeted at $0$ is the curve $\gamma$ (parametrised by $[0, \infty)$) described via the (unique) family of conformal maps  $g_t: \D  \to \D \setminus \gamma[0, t]$ with
$g_t(0) = 0$ and $g_t'(0) >0$, where $g_t$ satisfies the following differential
equation (called the radial Loewner equation) for each $z \in \D \setminus \gamma[0,t]$:
\begin{equation}
  \label{eq:rad_SLE_supp}
  \frac{\partial g_t(z)}{\partial t} = g_t(z)
  \frac{e^{i\sqrt{\kappa}B_t}+ g_t(z)}{e^{i\sqrt{\kappa}B_t}- g_t(z)}
  ; \quad g_0(z) = z.
\end{equation}
SLE$_\kappa$ enjoys conformal invariance, hence one can obtain radial SLE$_\kappa$ curves in other domains and/or for other starting and target points simply by applying conformal maps to the curve described by \eqref{eq:rad_SLE_supp}.
It is worthwhile to note here that radial SLE$_\kappa$ in $\D$ starting at 1 and targeted at $0$ is symmetric under conjugation in distribution.

\subsection{Gaussian Free Field (GFF)}
\label{sec:gff_supp}
In this section, we recall the definition of the GFF mostly to fix normalisation. See \cite{GFFShe} and \cite{NBnotes}
for more thorough definitions. Let $D$ be a domain in $\C$. Let $G_D$ be the Green function in $D$ (with Dirichlet boundary conditions), i.e., $ G_D(x,y) = \pi \int_{0}^\infty
p_t^D(x,y)dt $ where $p_t^D$ is the transition kernel for a Brownian motion
killed upon exiting $D$. The Gaussian free field, viewed as a stochastic process indexed by test functions $f\in C^\infty(\bar D)$, is the centered Gaussian process such that $(h,f)$ is a normal random variable and covariance between $(h,f)$ and $(h,g)$ given by $\iint G_D(x,y) f(x) g(y) dx dy$.
Alternatively, if $f_n$ is an orthonormal basis for the Sobolev space $H_0^1(D)$ induced by the Dirichlet inner product
\begin{equation}
  \label{eq:2_supp}
  (f,g)_\nabla = \frac{1}{2\pi}\int_D \nabla f \cdot \nabla g,
\end{equation}
then $h = \sum_n X_n f_n$ where $X_n$ are i.i.d. standard real Gaussian random variables. As a consequence of Weyl's law, it can be seen that this sum converges almost surely in the Sobolev space $H^{- \eta}(D)$ for any $\eta>0$.
Note that even though the GFF is not defined pointwise we will sometimes freely abuse notations and write $\E( h(x_1)  \ldots h(x_k))$ for its $k$-point moments (which are unambiguously defined if the $x_i$ are distinct).

Finally, for any function $u$ on the boundary of $D$, a GFF in $D$ with boundary condition given by $u$ is defined to be a Dirichlet GFF in $D$ plus the harmonic extension of $u$ in $D$. (In fact this also makes sense when $u$ is rougher than a function, provided that it can be integrated against harmonic measure).

\subsection{Wilson's  algorithm}\label{sec:Wilson_supp}


Recall from the introduction of the main file that given a finite weighted, oriented graph $G$ with a marked ``root'' vertex (and such that there exists a path from any point to the root), a \emph{spanning tree} of $G$ is an oriented subgraph $T$ containing all vertices of $G$, exactly one outgoing edge for each non-root vertex, and such that if we forget about the orientations $T$ is connected and has no cycle.
Note that starting from any point $x$, by following the outgoing path one obtains a self-avoiding path from $x$ to the root which will be called the branch of the tree starting at (or emanating from) $x$ which we will often write as $\gamma_x$.
The measure defined in picking a tree with a probability proportional to the edges is called (with a small abuse of language) the Uniform Spanning Tree measure or UST for short.


A crucial tool for studying the UST is Wilson's celebrated algorithm, which we recall here for convenience (since we need it in the context of directed and weighted graphs). This relies on the notion of loop-erased random walk, which we briefly explain now. Consider a simple random walk $(X_s)_{0 \le s \le T}$ run until some (possibly random) time $T$. The loop-erasure $Y$ of $(X_s)_{0 \le s \le T}$ is obtained from $X$ by chronologically erasing the
loops. By chronologically erasing the loops, we mean that we keep track of the current loop-erasure $Y^t$ which we update as follows. Given $Y^t$ and the next position of the walk $X_{t+1}$, if $X_{t+1} \notin Y^t$ we set $Y^{t+1} = Y^t \cup X_{t+1}$, and if $X_{t+1} \in Y^t$ we set $Y^{t+1}$ as the part of $Y^t$ up to and including $X_{t+1}$. The latter case is what we call erasing a loop. See \cite{Lawler_RW} for general background on loop-erased random walks. We also observe that the law of the path when loops are erased in a chronological order, and the law of the path when loops are erased in a reverse chronological order coincide, even if the graph is oriented (i.e., the random walk is nonreversible). See Lemma 7.2.1 in \cite{Lawler_RW} for a proof.

We now describe Wilson's algorithm for sampling a UST.
In the first step, we fix
a vertex, the root. To sample a uniform
spanning tree rooted at that vertex, we now pick
any other vertex and perform a loop-erased
random walk until the walk hits the root.
In the next step we update the root by
adding this loop-erased path to the old root, and iterate this procedure, (each time stopping when we hit the updated root), until
there are no vertices remaining. It
is easy to see that such a process in the end produces a tree, where the outgoing edge is taken to be towards the initial root or equivalently the first step of the loop-erased walk out of this vertex. This
tree has the law of the uniform spanning tree (\cite{LP:book}). Furthermore, and crucially, the law of this resulting tree does not depend on
the choice of the ordering of vertices from which we draw loop-erased paths.

It follows that from every vertex,
one can generate branches of the uniform spanning tree from that
vertex to the root by sampling loop erased random walk from those
vertices to the root. As mentioned before, often we will be interested in the \textbf{wired} uniform spanning tree, which is the UST on the graph where we have identified together a given set of ``boundary" vertices to be the root.

\subsection{Continuum Uniform spanning tree}\label{SS:contUST_supp}
The Continuum (Wired) Uniform Spanning Tree, in a simply connected domain $D$, is the scaling limit of the discrete UST in an approximation $D^\d$ of $D$ where we have wired all vertices of $\partial D^\d$. The convergence is in the following underlying topology. For any metric space
$X$, let $\mathcal H (X)$ denote the space of subsets of $X$ equipped with the
Hausdorff metric. Let $\cP(z,w,D)$ be the space of all continuous paths in $\overline D$
from a point $z\in \overline D$ to $w\in  \bar D$. We consider the Schramm space
$D \times D \times \cup_{z,w \in \overline{D}} \cP(z,w)$ which we view as a subset of $\mathcal H (\overline D \times \overline D \times \cH
(\overline D))$ equipped with the induced distance. A discrete wired UST is embedded into this space by considering $\{(x, y, \gamma_{xy}) | x, y \in D^\d \}$ where $\gamma_{xy}$ is the unique path connecting $x$ and $y$ in the tree, possibly including a part of the boundary if $x$ and $y$ belong to different connected components. 
We view the Schramm space as a subset of the
compact space $\mathcal H (\overline D \times \overline D \times \cH
(\overline D))$ equipped with its metric. This is the \textbf{Schramm
topology}.

All we will need to recall for now is the following fact.
 Let $z_1,z_2,\ldots,z_k$ be distinct points in $D$. A landmark result of Lawler,
Schramm and Werner \cite{LSW} shows that loop erased random walk in a domain
$D^\d$ of $\delta \Z^2$ approximating $D$ from the vertex closest to $z_i$ converges as
$\delta \to 0$ to a radial SLE$_2$ curve (denote it by $\gamma_{z_i}$). This SLE$_2$ curve starts from a point on the
boundary picked according to harmonic measure from $z_i$, and is targeted at $z_i$.
This convergence is in the Hausdorff sense. 
 From this it was deduced (cf. \cite{SLE}, Theorem 11.3 or Corollary 1.2 in \cite{LSW}) that the UST converges in the Schramm sense to an object which we call the \emph{continuum UST}. We call the curve $\gamma_{z_i}$ obtained as the scaling limit of the discrete UST branch the \emph{branch from $z_i$} of the continuum UST (note that this is a typical observable for the Schramm topology).
It is shown in Theorem 1.5 of \cite{SLE} that the branch
from a point in a domain is almost surely unique (i.e., there is a unique path connecting $z$ to $\partial D$) and hence this branch is unique a.s. for Lebesgue-almost every point $z \in D$.
(Actually Theorem 1.5
in \cite{SLE} deals with UST in the sphere but the same
argument works in a domain, see \cite{SLE}, Theorem 11.1.)
It is therefore easy to deduce:

\begin{prop}[Wilson's algorithm in the continuum]\label{prop:Wilson_cont_supp}
Let $D \subset \C$ be a simply connected domain and $z_1,\ldots, z_k \in D$.
We can sample the (a.s. unique) branches of the continuum wired UST in a domain
$D$ from
$z_1,\ldots, z_k$ as follows. Given the branches $\eta_i$ from $z_i$ for $1\le
i < j$, we inductively sample the branch from $z_j$ as follows. We pick a
point $p$ from the boundary of $D' := D \setminus \cup_{1\le i \le j} \eta_i$
according to harmonic measure from $z_j$ and draw an SLE$_2$ curve in $D'$ from $p$ to
$z_{j}$. The joint law of the branches does not depend on the order in
which we sample the branches.
\end{prop}

\section{Winding of curves}\label{app:winding_int_top_supp}

In this section, we provide detailed proofs of the results
(Lemmas 2.1 to 2.4) in Section 2.1 of the main file, which concern
basic deterministic facts about intrinsic and topological windings of a simple path.

\subsection{Intrinsic and topological winding}
\label{sec:wind_int_supp}

\begin{lemma}[Lemma 2.1 in the main file]\label{lem:intrinsic->top_supp}
Let $\gamma : [0, 1] \mapsto \C $ be a smooth self avoiding curve with $\gamma'(s) \neq 0$ for all $s$. We have
\begin{equation}
W_{\i}(\gamma) = W(\gamma, \gamma(1) ) + W( \gamma, \gamma(0) ).\label{eq:int->top_gen_supp}
\end{equation}
\end{lemma}

\begin{proof}
Let $x$ $[2\pi]$ denote $x$ modulo $2\pi$ for any real
number $x$. Let $\Arg \in (-\pi,\pi]$ be the principal branch of argument with branch cut $(-\infty,0]$. Note that
 \begin{align*}
 \lim_{u\to t, u < t} \Arg ( \gamma(u) - \gamma(t))& \equiv \pi + \Arg (\gamma'(t)) [2\pi], \\  \lim_{u\to s, u > s} \Arg ( \gamma(u) - \gamma(s))& \equiv  \Arg (\gamma'(s)) \,\, [2\pi].
 \end{align*}
 Therefore for any $t$,
 \begin{align*}
 W(\gamma[0,t], \gamma(t) ) + W( \gamma[0,t], \gamma(0) ) & \equiv \left( \pi + \Arg (\gamma'(t)) - \Arg( \gamma(0) - \gamma(t) )\right) \\
 & \quad \quad+ \left( \Arg( \gamma(t) - \gamma(0) ) - \Arg (\gamma'(0)) \right) \,\,[2\pi] \\
   & \equiv  \Arg (\gamma'(t)) - \Arg (\gamma'(0)) \,\,[2\pi]\\
   & \equiv W_{\i}(\gamma[0,t])  \,\,[2\pi]
\end{align*}
where notice that $\pi+ \Arg( \gamma(t) - \gamma(0) )$ cancels with $\Arg( \gamma(0) - \gamma(t) )$.
Equivalently we can write
\begin{equation}
W(\gamma[0,t], \gamma(t) ) + W( \gamma[0,t], \gamma(0) ) =  W_{\i}(\gamma[0,t]) +2\pi k_t
\end{equation}
for some $k_t \in \Z$.
However, as $t$ goes to $0$, both $W(\gamma[0,t], \gamma(t) ) + W( \gamma[0,t], \gamma(0) ) $ and $W_{\i}(\gamma[0,t])$ go to $0$ which implies $k_0=0$. Since $\gamma$ is smooth and self avoiding, it is easy to check that both  the winding terms are continuous in $t$. But this implies $k_t$ is continuous in $t$ and hence we conclude $k_t = 0$ for all $t$. This completes the proof.
\end{proof}

\subsection{Distortion estimates}
In this section we record some distortion estimates required in order to give a proof of Lemma 2.4 of the main file. This will be stated and proved in \cref{cor:det_winding_supp}.

\begin{lemma}[Distortion estimates]\label{lem:distortion_supp}
Let $D$ be a domain containing $0$ and let $R = R(0,D)$.
Let $g$ be a conformal map defined on $D$ mapping $0$ to $0$. Then
for any $z \in B(0, R/8) $,
\[
 |g(z)| \le 4 |z g'(0)|
\quad ;\quad |g(z)-g'(0)z|<6\frac{|z|^2}{R}|g'(0)|.
\]
 In particular, the image of a straight line joining 0 to a point at a distance
$\ve< R/8$ under $g$ lies within a cone of angle $\arctan
(6\frac{\ve}{R})$.
\end{lemma}

\begin{proof}
The first statement follows from applying the Growth theorem (see Theorem 3.23 in \cite{Lawler}) to the function $g(R z /4) / (R g'(0)/4)$, which is defined on the unit disc by Koebe's 1/4 theorem.
The second statement follows from applying Proposition 3.26 in
\cite{Lawler} with $r=1/2$ to the same function. The final assertion is immediate since the distance of $g(z)$ from $g'(0)z$ is at most $|g'(0)z| 6 \ve/R$ for all $z$ with $|z | \le \ve$.
\end{proof}

\begin{lemma}[distortion of argument]
 \label{lem:distortion2_supp}
 Let $K$ be a closed subset of $\bar \D$ such that $H = \D \setminus K$ is simply connected (i.e. $K$ is a hull).
 Further assume that the diameter of $K$
is smaller than some $\delta < 1/2$ and $1 \notin B(K,\delta^{1/2})$. Let
$\tilde g$ denote the conformal map sending $H$ to $\D$ with $\tilde g(0) = 0$ and $\tilde g(1) = 1$. Then
$$|\arg_{\tilde g'  (H)}(\tilde
g'(0)) - \arg_{\tilde g'(H)} (\tilde g'(1))|<C\delta^{1/2},
$$
 where $C$ is a universal constant. Here
$\arg_{\tilde g'  (H)} (\cdot)$ is the argument in $\tilde g'  (H)$ (which does not contain 0), defined up to a global unimportant additive constant.
\end{lemma}

\begin{proof}
Let $T$ be the capacity of $K$ seen from $0$ and let $(K_t)_{t \leq T}$ be a growing sequence of hulls of capacity $t$ such that $K_T = K$.
Let $g_t$ denote the Loewner maps associated to the family $(K_t)_{t\le T}$, with the usual convention $g_t(0) = 0$
and $g_t'(0) \in \R^+$ and note that $\tilde g = g_T / g_T(1)$.
Let $W:[0,\infty) \mapsto \R$ be the driving function for the radial Loewner
differential equation \eqref{eq:rad_SLE_supp} for the maps $(g_t)_{t \ge
  0}$ and $H_t = \D \setminus K_t$.
\begin{figure}
\begin{center}
\includegraphics{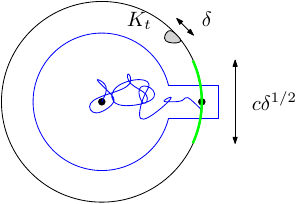}
\end{center}
\caption{Sketch of proof of \cref{lem:distortion2_supp}. The probability that the
Brownian motion exits through the green arc all the while staying within the
region bounded by blue arcs is $c\delta^{1/2}$. This is a lower bound for
$\harm_{H_T}(0, I^\pm)$. \label{F:argcont_supp}}
\end{figure}

Since the conformal radius $R(0, H_t)$ is at least the inradius which is at least $1-\delta$, we obtain that the capacity $t$ of $K_t$ seen from $0$ is always smaller
than $- \log (1- \delta) \le C \delta$ for $t \le T$, in particular $T \le C \delta$. Let $A_t=\partial \D \setminus g_t(\bar \D \setminus K_t)$, i.e., $A_t$ is the part of $\partial \D$ where $g_t$ maps the
boundary of $K_t$. By definition, $e^{i W_t} \in A_t$ for any $t$. Let
$\harm_D(z,S)$ denote the harmonic measure seen from $z$ of a set $S$ in the domain $D$. Since
$\dist(1, K) \geq \delta^{1/2}$, we can find $\theta_- < 0 <
\theta_+$ such that
$$
\harm_{H}(0, I^\pm) \ge c \delta^{1/2},\,\, I^\pm \cap K = \emptyset
$$
where $I^\pm$ are the two arcs connecting $e^{i \theta_\pm}$ to $1$ (the green
arcs in \cref{F:argcont_supp}). Note that since $K_t \subset K_T = K$, the same holds
for $\harm_{H_t}$.
  Applying conformal invariance of the harmonic
measure, we get $\harm_{\D} (g_t(I^{\pm}) ) \geq c\delta^{1/2}$ and
therefore $\diam(g_t(I^{\pm})) \geq c \delta^{1/2}$. Finally since
$I^\pm$ did not intersect $K_T$, $g_t(I_\pm) $ does not intersect
$A_t$ and therefore $\abs{e^{iW_t}-g_t(1)} \geq c \delta^{1/2}$.

Using this bound, the Loewner equation gives
\[
\abs{ \partial_tg_t(1)} = \abs{g_t(1) \frac{e^{i W_t}+ g_t(1)}{e^{iW_t}-g_t(1)}} \leq C \delta^{-1/2} .
\]
Integrating for a time $t \le T\le C\delta$ gives $\abs{g_t(1)-1} \leq
C\delta^{1/2}$. Recall $\Arg$ is the principal branch of argument in
$(-\pi,\pi]$. This gives $|\Arg(g_T(1))| \leq C \delta^{1/2}$ and since by
definition $g_t'(0) \in \R^+$,
\begin{equation}\label{E:Arg_supp}
\abs{\Arg \tilde g_t'(0)} =\abs{ \Arg g_t'(0) -
\Arg g_t(1)} \leq C \delta^{1/2},
\end{equation}
where $\tilde g_t = g_t / g_t(1)$ is the conformal map sending $H_t$ to $\D$ such that $\tilde g(0) = 0$ and $\tilde g( 1) = 1$.
Note that $\Arg \tilde g_t'(0)$ is continuous in $t$ since it is bounded by $C \delta^{1/2}$ (and so does not come into the region where $\Arg$ is discontinuous).

We deduce that the same bound as \cref{E:Arg_supp} holds for $\arg_{\tilde g_T' (H_T)}$.
Indeed note that one can find a curve $\lambda[0,1]$
connecting $1$ to $0$ which avoids $K_T$. By definition,
$$
\arg_{\tilde g_t'(H_t)}(g'_t(0)) - \arg_{\tilde g'_t(H_t)} (g'_t(1)) = \Im\left(\int_{\tilde g_t'(\lambda)}\frac{dz}{z}\right) = \int_{0}^1 \Im\left(\frac{\tilde g_t''(\lambda(s))\lambda'(s)}{\tilde g_t'(\lambda(s))} \right)ds.
$$
Thus the left hand side is continuous in $t$ since $\tilde g_t$ has a conformal extension in a neighbourhood of $\lambda$ which implies that all its derivatives are continuous in $t$ in this neighbourhood. Also for $t = 0$ the difference between the left hand side and $\Arg \tilde g'_t(0)$ is $0$ trivially. This
concludes the proof.
\end{proof}

 \subsection{Change in winding under conformal maps}
\label{sec:image-wind_supp}

We are now able to return to the proof of Lemma 2.4 in the main file which will be stated and proved in \cref{cor:det_winding_supp}. We will use a deformation argument and the main step is to prove the following intermediate lemma.

\begin{figure}
\begin{center}
\includegraphics[scale=.5]{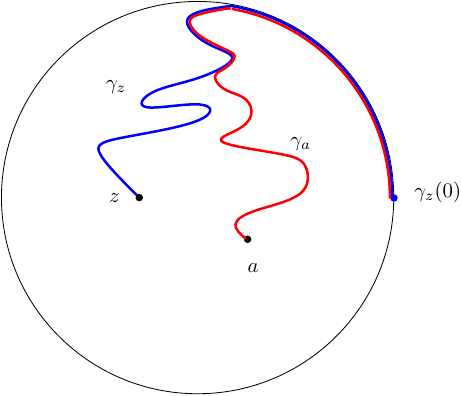}
\caption{A typical application of \cref{lem:image_intrinsic_supp}. Here $\gamma = \gamma_z$ (the curve going from $1$ to $z$ in a continuum
UST) and $\psi$ is the Loewner map removing $\gamma_a$. Note that $\psi(\gamma_z)$ is well-defined
as a continuous curve in $\bar \D$. \label{F:deformation_supp} so condition (iii) is satisfied.}
\end{center}
\end{figure}

\begin{lemma}\label{lem:image_intrinsic_supp}
Let $D,D'$ be domains with locally connected boundary. Let $\psi$ be a conformal map sending $D$ to $D'$. Let $\gamma : [0,1] \mapsto \bar D$ be a curve (not necessarily simple) in $\bar D$. Assume the following
\begin{enumerate}[{(}i{)}]
\item  The endpoints of $\gamma$ are smooth and simple.
\item $\arg_{\psi'(D)} \psi'$ extends continuously to a neighbourhood in $D$ of $\gamma(0)$ and $\gamma(1)$.
\item There exists a continuous curve $\tilde \gamma \subset \overline \D$ such that $\ph(\tilde \gamma) = \gamma$, where $\ph: \bar \D \mapsto \bar D$ is the extension of a conformal map in $\D$. (Such an extension exists by our assumption that $\partial D$ is locally connected.) See \cref{R:conformalboundary_supp} below for an equivalent reformulation of this condition.
\end{enumerate}
 Then,
\begin{multline*}
W\left( \psi(\gamma), \psi(\gamma(1)) \right) + W\left( \psi(\gamma), \psi(\gamma(0)) \right)\\ =  W(\gamma, \gamma(1) ) + W(\gamma, \gamma(0) )  + \arg_{\psi'(D)}(\psi'(\gamma(1) ) ) - \arg_{\psi'(D)}( \psi'(\gamma(0) )
\end{multline*}
where $\arg_{\psi'(D)}$ here is any determination of the argument on $\psi'(D)$.
\end{lemma}
\begin{proof}
We write $\arg$ for $\arg_{\psi'(D)}$ in this proof for notational convenience. First, we assume that $D$ is the unit disc $\D$ and $\tilde \gamma = \gamma$. We define a family of conformal maps,
\begin{align*}
\psi_t (z) &= \psi( t z)/t \text{ for }t\in (0, 1], & \psi_0(z) &= \psi'(0) z.
\end{align*}
The map $\psi_0$ is just a rotation and scaling so the topological winding does not change and $\arg(\psi'_0(\gamma(1) ) ) = \arg( \psi'_0(\gamma(0) ) = \arg(\psi'(0))$. Thus we have:
\begin{multline*}
W\left( \psi_0(\gamma), \psi_0(\gamma(1)) \right) + W\left( \psi_0(\gamma), \psi_0(\gamma(0)) \right) \\
=W(\gamma, \gamma(1) ) + W(\gamma, \gamma(0) ) + \arg(\psi'_0(\gamma(1) ) ) - \arg( \psi'_0(\gamma(0) ) .
\end{multline*}
On the other hand, since locally at $\gamma(0)$ and $\gamma(1)$, $\psi$ acts only by rotation and scaling we see that for all $t$,
\begin{multline}
W\left( \psi_t(\gamma), \psi_t(\gamma(1)) \right) + W\left( \psi_t(\gamma), \psi_t(\gamma(0)) \right) \\
\equiv W(\gamma, \gamma(1) ) + W(\gamma, \gamma(0) ) + \arg(\psi'_t(\gamma(1) ) ) - \arg( \psi'_t(\gamma(0) ) [2\pi] \label{eq:mod2pi_arg_change_Supp}.
\end{multline}
Indeed, let $\Arg$ denote the argument in $(-\pi,\pi]$ and note that
\begin{align*}
& \quad \quad W\left( \psi_t(\gamma), \psi_t(\gamma(1)) \right)\\
 & \equiv \lim_{s \to 1}\Arg(\psi_t(\gamma(s)) - \psi_t(\gamma(1))) - \Arg(\psi_t(\gamma(0)) - \psi_t(\gamma(1))) \,[2\pi]\\
& \equiv \Arg(\psi_t'(\gamma(1))) + \Arg(\gamma'(1)) +\pi- \Arg(\psi_t(\gamma(0)) - \psi_t(\gamma(1))) \,[2\pi]
\end{align*}
and
\begin{align*}
& W\left( \psi_t(\gamma), \psi_t(\gamma(0)) \right) \\
 \equiv& \Arg(\psi_t(\gamma(1)) - \psi_t(\gamma(0))) -  \lim_{s \to 0}\Arg(\psi_t(\gamma(s)) - \psi_t(\gamma(0))) \,[2\pi]\\
 \equiv & \Arg(\psi_t(\gamma(1)) - \psi_t(\gamma(0))) -  \Arg(\psi_t'(\gamma(0))) - \Arg(\gamma'(0))  \,[2\pi]
\end{align*}
and also
\begin{equation*}
W(\gamma, \gamma(1) ) + W(\gamma, \gamma(0) ) \equiv \Arg(\gamma'(1)) - \Arg(\gamma'(0))  \,[2\pi].
\end{equation*}
Both sides of \cref{eq:mod2pi_arg_change_Supp} match when $t=0$ and the right hand side is clearly continuous in $t$ since $\psi'$ extends continuously to $\gamma(0)$ and $\gamma(1)$, so we only have to argue that the left hand side is continuous. First, note that the curves $\psi_t(\gamma)$ are continuous in $t$ in the Hausdorff sense. Indeed the continuity is obvious for $t<1$ and at $t=1$ continuity follows from the fact that the map $\psi$ extends continuously to the boundary of $\D$. Now we argue that the topological winding of the curves are continuous in $t$, which will follow from the continuity of $\arg(\psi')$ up to the boundary.
Let us argue for the winding around $\gamma(0)$ at a fixed time $t_0$. Since $\gamma$ is smooth in a neighbourhood of $\gamma(0)$, we can find $\eps >0$ such that $\abs{\arg \gamma'(s) - \arg \gamma'(0)} \leq \eps$ for all $s\leq \eps$. We can further assume
 by continuity of $\arg \psi'$ that $\eps$ is such that $\abs{ \arg( \psi_t'( \gamma(s)) -\arg(\psi_{t_0}'(\gamma(0) )} \leq \eps$ for $s \leq \eps$ and $\abs{t-t_0} \leq \eps$. This shows that $\psi_t(\gamma[0,\eps])$ is a smooth curve  whose tangent at any point is always within $2\eps$ of $\arg(\psi_{t_0}'(\gamma(0) ) + \arg \gamma'(0)$. It is easy to see that such a curve cannot exit a cone of direction $\arg(\psi_{t_0}'(\gamma(0) ) + \arg \gamma'(0)$ and of angle $2\eps$ and therefore satisfies $W(\psi_t (\gamma[0, \eps]), \psi_t (\gamma(0)))\leq 2\eps$. On the other hand $\psi_t (\gamma[\eps, 1])$ stays uniformly away from $\psi_t( \gamma(0) )$ so $W(\psi_t (\gamma[\eps, 1]), \psi_t( \gamma(0) ))$ is continuous in $t$. Overall $W(\psi_t (\gamma), \psi_t( \gamma(0) ))$ is continuous in $t$. The argument for the winding around $\gamma(1)$ is identical and we are done in the case $D = \D$.


For the general case, take the conformal map $\psi \circ \ph:\D \to D'$. Using our previous argument for $\ph: \D \to D$ and $\psi \circ \ph$ gives two equations connecting the winding of $\tilde \gamma$ with the winding of $\gamma$ and $\psi(\gamma)$. Combining the two and noting that the equation do not depend on the particular realisation of $\arg$, we conclude.
\end{proof}

\begin{remark}\label{R:conformalboundary_supp}
The condition (iii) in \cref{lem:image_intrinsic_supp} is easier to understand when appealing to the notion of \emph{conformal boundary} (see \cite{BN_SLE_notes}). To explain what that is, we fix a conformal map $\ph: D \to \D$ and equip $D$ with the distance induced by $\ph$, i.e., we set for $z, z' \in D$, $d_\ph(z, z') = |\ph(z) - \ph(z')|$. The conformal closure $\text{cl}(D)$ of $D$ is defined as the completion of $D$ with respect to this distance. (The conformal boundary can then be defined to be $\text{cl}(D) \setminus D$; this notion is then equivalent to that of Poisson and Martin boundaries induced by Brownian motion on $D$, as well as prime ends in \cite{Pommerenke}.) Note that $\text{cl}(D)$ is then identified with the closed disc $\bar \D$ and so a point in $\text{cl}(D)$ projects to a unique point on $\bar D$ since by local connectedness of $\partial D$, $\ph^{-1}$ extends to the closed disc, see Theorem 2.6 in \cite{Pommerenke}.

With these definitions, assumption (iii) simply says that $\gamma$ is a continuous curve in $\text{cl}(D)$, i.e., $\gamma$ is the projection of a continuous curve in $\text{cl}(D)$.
\end{remark}

Putting \cref{lem:image_intrinsic_supp} together with Lemma 2.2 in the main file, we obtain the following.
\begin{lemma}\label{lem:det_winding2_supp}
Let $D,D'$ be domains with locally connected boundaries and let $\psi$ be a conformal map sending $D$ to $D'$.
Let $\gamma : [0,1] \to \bar D$ be a curve in $\bar D$ and assume that it is smooth and simple at $\gamma(1)$. Assume further that $x = \gamma(0) \in \partial D$ and that $\arg_{\psi'(D)}(\psi')$ extends continuously to $\gamma(0)$ and $\gamma(1)$. Then, letting $x' = \psi(x)$,
\begin{multline*}
W\left( \psi(\gamma), \psi(\gamma(1)) \right)  - W(\gamma, \gamma(1) )\\
 = \arg_{\psi'(D)} (\psi'(\gamma(1))) + \arg_{D; x}(\gamma(1)) - \arg_{D'; x'}(\psi(\gamma(1))  ).
\end{multline*}
where we \emph{choose} the global constants defining the arguments so that the chain rule holds at $x =\gamma(0)$:
\begin{equation}
\arg_{D';x'} ( (\psi \circ \gamma)'(0)) = \arg_{D,x} (\gamma'(0)) + \arg_{\psi'(D)} (\psi'(x)).
\label{eq:arg_condition_Supp}
\end{equation}
 Furthermore if $\arg_{\psi'(D)}(\psi')$ does not extend continuously to $\gamma(0)$, the formula still holds up to a global constant in $\R$ not depending on $\gamma$ and depending only on the choice of the constants in the definition of the arguments.
\end{lemma}

\begin{proof}
First assume that $\gamma(0)=x$ is a smooth simple point of $\gamma$, then the results follows directly from applying \cref{lem:image_intrinsic_supp} together with Lemma 2.2 in the main file. Now to get rid of this assumption, note that if we modify $\gamma$ in a small enough neighbourhood of $x$, we cannot change the value of the left hand side. We can therefore always change $\gamma$ into a curve satisfying our assumptions so we are done. For the case where $\arg(\psi')$ does not extend to $\gamma(0)$, note that as in the proof of \cref{lem:image_intrinsic_supp}, we can approximate $\psi$ by a map $\psi_t$ that has derivative in $\gamma(0)$. Applying the result for the smooth case modulo a global constant and then taking the approximation to $0$ (i.e., $t\to 1$)  yields the result.
\end{proof}

Finally in the case where we want to compute winding with respect to a point different from the endpoint -- which in practice is what we will do to truncate the winding of SLE -- the distortion lemma gives us a version of the \cref{lem:det_winding2_supp} with an error term.

\begin{corollary}[Lemma 2.4 in the main file]\label{cor:det_winding_supp}
Let $D,D'$ be bounded domains with locally connected boundary and let $\psi$ be conformal map sending $D$ to $D'$.
Let $\gamma : [0,1] \mapsto \bar D$ be a curve in $\bar D$. Assume further that $\arg_{\psi'(D)}(\psi')$ extends continuously to $\gamma(0)$ and $\gamma(1)$. Let $z$ be a point in $D\setminus \gamma[0, 1]$ and let $R = R(z, D)$ be its conformal radius and assume that $\abs{z-\gamma(1)}\leq R/8$. Then, letting $x = \gamma(0)$ and $x' = \psi (x)$,
\begin{multline}
W\left( \psi(\gamma), \psi(z) \right)  - W(\gamma, z ) \\
= \arg_{\psi'(D)} (\psi'(z)) + \arg_{D;x}(z) - \arg_{D';x'}(\psi(z)  ) +O(\abs{z-\gamma(1)}/R),\label{eq:det_winding_Supp}
\end{multline}
where the implicit constant in $O(\abs{z-\gamma(1)}/R)$ is universal. The constants in the arguments are defined as in \cref{lem:det_winding2_supp}. Furthermore if $\arg_{\psi'(D)}(\psi')$ does not extend to $\gamma(0)$, the formula still holds up to a global constant in $\R$ depending only on the choice of the constants for the arguments and not on $\gamma$.
\end{corollary}
\begin{proof}
Let $\tilde \gamma$ be obtained by appending a straight line connecting $\gamma(1)$ to $z$. By Koebe's 1/4 theorem $\tilde \gamma$ is still a curve in $D$ and it is obviously smooth at its last point so we can apply \cref{lem:det_winding2_supp} to it. On the other hand, by \cref{lem:distortion_supp}, we see that the image of the straight segment has winding $O(\abs{z-\gamma(t)}/R)$ since it stays in a cone of that angle. By additivity of the winding we are done.
\end{proof}

Note that the condition that the argument extends continuously
near $\gamma(0)$ (condition (ii) in \cref{lem:image_intrinsic_supp}, which also arises in \cref{cor:det_winding_supp}) is trivial if $\gamma(0)$ and $\gamma(1)$ are interior points but is non-trivial for boundary points. However Theorem 3.2 in \cite{Pommerenke} implies that if $D,D'$ are \emph{smooth} Jordan curves then this condition is satisfied. For completeness, we now give a slight generalisation of this fact, which gives us a simple geometric sufficient criterion for this condition to hold (essentially it suffices that the boundary is smooth locally near the point where we wish to extend the argument).

\begin{lemma}\label{lem:geometric_condition_Supp}
Let $\psi$ be a conformal map between two domains $D$ and $\tilde D$ and let
$x \in \partial D$ be fixed. Assume that $D$ and $\tilde D$ have locally connected boundary
and let $\lambda$ be a parametrisation of $\partial D$
coming from a map to the disc (i.e up to parametrisation, $\lambda$ is the curve $(g(e^{it}))_{0\le t \le 2\pi}$ with $g$ conformal from $\D$ to $D$).
If there exists an open interval $I$ such that $x \in \lambda(I)$ and both $\lambda(I)$
and $\psi(\lambda(I))$ are smooth curves,
then $\arg_{\psi'(D)} \psi'$ extends continuously to a neighbourhood of $x$. Here $\arg$ is any realisation of argument in the range of $\psi'$.
\end{lemma}

\begin{proof}

We consider first the case $D = \D$. Let us write $x' = \psi(x )$ and let $I$ be an interval given as in the statement. Let $\tilde \lambda$ be a $C^1$ parametrisation of $\psi(\lambda(I))$ with non vanishing derivative and $t_0$ be such that $x' = \tilde \lambda(t_0)$. Since $\psi(\lambda(I))$ is smooth, we can assume by taking a smaller $I$ if necessary that for all $t \in I$, $\abs{\arg_{\psi'(D)} \tilde \lambda'(t) - \arg_{\psi'(D)} \tilde \lambda'(t_0) } \leq 1$. As in the proof of \cref{lem:image_intrinsic_supp}, this implies that $\tilde \lambda$ stays in a cone of angle $2$ around each of its points and in particular $\tilde \lambda$ is injective, or in other word it is a Jordan arc. It is then easy to see that we can find a sub-domain $E \subset D'$ such that $E$ is bounded by a smooth Jordan curve and $\tilde \lambda(t) \in \partial E$ for $t$ in a neighbourhood of $t_0$.



Let $K = \psi^{-1}(D' \setminus E)$ and let $g$ be a map sending $\D \setminus K$ to $\D$ and $x$ to $x$. Observe that $\arg(g')$ extends to $x$ (e.g. by Schwarz reflection). Note also that the map $\ph = \psi \circ g^{-1}$ is a conformal map sending $\D$ to $E$ which is a smooth Jordan domain. By Theorem 3.2 in \cite{Pommerenke} $\arg(\ph')$ extends continuously to the boundary. Finally by construction $\psi = \ph \circ g$ in a neighbourhood of $x$ so $\arg \psi'$ also extends to the boundary in a neighbourhood of $x$. We conclude with arbitrary $D$ and $D'$ by considering the maps from $\D$ to $D$ and $\D$ to $D'$ and by composition.
\end{proof}

\subsection{One-point winding lemma}

Schramm's theorem about one-point winding (Theorem 3.1 in the main file) deals with SLE curves towards $0$ in the unit disc $\D$. We now provide an extension
of this result for SLE curves towards an arbitrary point in $\D$.

\begin{lemma}[Lemma 3.4 in the main file]\label{lem:one_point_ht_supp}
Let $z \in \D$ and let $\psi: \D \mapsto \D$ be the M\"obius transformation
mapping $z$ to $0$ and
$1$ to $1$. If $\gamma_z(t)  \in B(z,\ve)$ where $\eps \le R(z, \D)/8$, then
we have:
\begin{equation}
 W(\psi(\gamma_z[-1,t]),0) =
W(\gamma_z[-1,t],z) + \pi  - \arg_{\D ;1}(z) +\epsilon(t)\label{eq:1_supp}
\end{equation}
where the error term $|\epsilon(t)| \le C\ve/R(z,\D)$ for some universal
constant $C>0$
and $\arg_{\D, 1}$ is the argument function with values in $(\pi/2,3\pi/2)$ (so $\arg_{\D;1} (0) = \pi$). Also for all $s,t$
\begin{equation}
 \P(\abs{\gamma_z(t) - z} > e^{-t+s}R(z,\D)) \leq ce^{-c's}\label{eq:tip_supp}
\end{equation}
where $c,c'$ are independent of $z$.
\end{lemma}
\begin{proof}
  We are going to apply \cref{lem:det_winding2_supp} to the M\"obius transform $\psi$ of
  $\D$ mapping $z$ to $0$ and $1$ to $1$.
Computing $\psi(w)$ for $w \in \D$:
  \begin{equation}
    \label{eq:40_supp}
    \psi'(w) = \frac{1-|z|^2}{(1-w\overline z)^2}\frac{1-\overline z}{1-z}.
  \end{equation}
Now we argue that $\arg_{\psi' (\D)} (\cdot)$ can be taken to be $\Arg \in
(-\pi,\pi]$
with branch cut $(-\infty,0]$. Indeed let us describe $\psi'(\D)$. First, $1-\overline z \D$ is a
disc of radius $|z|$ contained in the right half plane $\{z:\Re (z)
>0\}$. Therefore $1/(1-\D\overline z)$ is a subset of the right half
plane since $z \mapsto 1/z$ preserves the right half
plane. Thus $(1-\D\overline z)^{-2}$ does not intersect $(-\infty,0]$.
Multiplying by the positive real number $(1- |z|^2)$ and rotating by $ \theta = \Arg( (1-\overline{z})/(1-z))$, we see that
$\psi'(\D)$ does not contain the half line $w$ joining $0$ and $e^{i (\theta-\pi
)}$. Therefore $\arg_{\psi' (\D)}$ coincides with the argument $\tilde Arg$ with
branch cut $w$ that can take the value $0$ (since they differ by some multiple of $2\pi$ and agree at $\psi'(1)$ by definition). With our choice of $\theta$, $\tilde Arg (\cdot)$ takes values in
$(\theta - \pi,\theta+\pi]$. Therefore $\tilde Arg ((1- \bar z)/(1-z) ) = \theta =  \Arg ( (1- \bar z)/(1-z))$ (since this is the only value congruent to $\theta$ mod $2\pi$ in the interval $(\theta - \pi, \theta + \pi]$) and hence,
\begin{multline*}
\arg_{\psi' (\D)}((1-\overline{z})/(1-z)) =\tilde
Arg(1-\overline{z})/(1-z)) = \theta \\
= \Arg(1-\overline{z})/(1-z))   = 2\pi - 2\Arg(z-1)    =2\pi -2\arg_{\D;1}(z)
\end{multline*}
where we used the rules of addition for arguments when going across a branch cut. Plugging this in \cref{lem:det_winding2_supp} and making the cancellation, we have the desired
expression \eqref{eq:1_supp}.

Now without loss of generality, assume $t-s \ge 10$. Applying Koebe's 1/4 theorem (see Theorem 3.17 in \cite{Lawler}), we see that
$$ B(0,e^{-t+s} /4) \subset \psi(B(z,e^{-t+s}R(z,\D) )). $$ Hence applying
Theorem 3.1 of \cite{SLE} and conformal invariance we have \eqref{eq:tip_supp}.
\end{proof}

\section{Flow lines of the GFF}\label{sec:coupling_GFF_supp}
\label{S:appendix_supp}
In this section, we provide a proof of Theorem 2.8 in the main file which is restated below in \cref{thm:coupling_intro_supp}. The reader is advised to recall the definitions of intrinsic winding boundary conditions on a simply connected domain $D$ with a marked point $x$ (which we denote by  $u_{D,x}$) from the main file.

\begin{thm}[Imaginary geometry coupling; Theorem 2.8 in the main file]

  \label{thm:coupling_intro_supp}
Let $D$ be a simply connected domain with a marked point $x$ on the boundary. Let $ h = \frac1{\sqrt{2}} u_{(D,x)} +  h^0_D$ where $ h^0_D$ is a GFF with Dirichlet boundary conditions in $D$.
There
exists a coupling between the continuum wired UST on $D$ and $h$ such that the
following is
true. Let $\{\gamma_i\}_{1 \le i \le k}$ be the branches of the continuum wired
UST
from points $\{z_i\}_{1 \le i \le k}$ in $D$ and let $D' = D \setminus \cup_{1 \le i \le k}
\gamma_i$ . Then the conditional law
of $h$ given $\{\gamma_i\}_{1 \le i \le k}$ is the same as $ \frac1{\sqrt{2}} u_{(D', x)} + h^0_{D'}$ where $h^0_{D'}$ is a GFF with Dirichlet boundary condition in $D'$. Furthermore, $h$ is completely determined by the UST and vice-versa.
\end{thm}

Given a planar graph $G$ with boundary $\partial G$ (here the boundary is just a given subset
of vertices of $G$), one can consider a uniform spanning tree in it with either
wired (when all the boundary vertices are collapsed into a single vertex) or
free boundary condition.
Every spanning tree in a planar graph corresponds to a spanning tree
in the dual graph.
 It is easy to see that the dual of a uniform spanning tree with free
boundary conditions is a uniform spanning tree in the dual graph with
a wired boundary condition. More generally, it will be convenient to consider trees
with \textbf{mixed boundary conditions}, that is, partially free and
partially wired. For example, we can divide the boundary $D^\d$
into an union of arcs $\alpha^\d$ and $\beta^\d$ and consider
the spanning tree with wired boundary on $\alpha^\d$ and free
boundary on $\beta^\d$. Then the dual tree is a spanning tree of
the dual $(D^\d)^\dagger$ with wired boundary conditions near $\beta^\d$
and free boundary conditions near $\alpha^\d$. See e.g. \cite{LSW} for details.

With every spanning tree in a planar graph, one can associate a curve
which forms the interface
in between the spanning tree in the graph and the dual spanning tree
in the dual graph. This interface visits every face, and is unique up to reparametrisation. This is called the \textbf{Peano curve} of the tree.
In case of the mixed boundary spanning tree described above, the
Peano curve joins the endpoints of the two arcs. The convergence of
this Peano curve to chordal SLE$_8$ was also established in \cite{LSW} using the following
topology: for paths $\beta,\gamma$ in $\bar D$, let
$d_{\text{Peano}} (\beta,\gamma) = \inf\sup_t |\beta(t)  -
\gamma(t)|$ where the infimum is over all possible parametrisations of
the curves $\beta, \gamma$. This is the \textbf{Peano curve
  topology}.

  \begin{thm} [\cite{LSW}]\label{thm:LSW_supp}

Let $\partial D$ be a $C^1$-smooth simple loop. Let
$(D^\delta)^{\dagger}$ denote the dual graph of $D^\delta$ which is a
subset of $(\Z +1/2)^2$.
\begin{itemize}
\item The uniform spanning tree with wired
boundary condition as well as the uniform spanning tree with free boundary condition converge as $\delta \to 0$ in the Schramm
topology. We call these limits the continuum wired uniform spanning tree
 and the continuum free uniform spanning tree
respectively.
\item Let $\alpha$ and $\beta$ be smooth arcs in $\partial D$ that are
  complementary: $\alpha \cup \beta = \partial D$ and let $a$ and $b$
  be points in $\partial D$ where $\alpha$ and $\beta$ meet. Let $\alpha^\d$
  be an approximation of $\alpha$ in $\partial D^\d$ and
  $\beta^\d$ be an approximation of $\beta$ in $\partial D^\d$
  in $(D^\d)^{\dagger}$ (the precise sense of this approximation can
  be found in Section 4 of Lawler, Schramm and Werner
  \cite{LSW}). Consider the UST in $D^\d$ with wired boundary
  condition on $\alpha^\d$ and the UST in $(D^\d)^{\dagger}$
  with wired boundary condition on $\beta^\d$. Then the Peano
  curve in between these two trees converge in law (with the underlying topology being the Peano curve
  topology) to a chordal SLE$_8$ curve from $a$ to $b$ in $D$.
\end{itemize}
\end{thm}

Now we describe a mapping which maps the Peano curve to an element in the Schramm space and essentially shows that the Peano curve topology is stronger than the Schramm topology. For this, we need to first go back to the discrete and describe the left and right boundary of the Peano curve.

In the graph $D^\d$, we join together every dual
vertex with the primal vertices adjacent to its face. These edges along with
the primal and the dual trees form a triangulation. Every edge in the primal or the dual tree belongs to
two triangles in this triangulation. Thus each triangle has an opposite triangle corresponding to the
(primal or dual) edge where it belongs. The Peano
curve visits each such triangle exactly once and we can fix the orientation of the Peano curve so that the primal tree edges always lie to its left. After
drawing a certain number of edges in the Peano curve, the triangles visited by
the Peano curve whose opposite triangle have not yet been visited form the
boundary of the curve. The boundary triangles corresponding to the
dual tree edges form the \textbf{left boundary} of the curve and the
triangles corresponding to the primal tree edges form the \textbf{right
boundary} of the curve.

 It is known that a chordal SLE$_8$
 is space filling almost surely (see Proposition 6.11 or more precisely Theorem 7.9 in \cite{Lawler}) which means that a.s. for all $z \in D$, the SLE$_8$ visits $z$.
 Since the curve is space filling, at
any stopping time $\tau$, the complement of the $\eta[0,\tau]$ consists of a single
component. The boundary of $D \setminus \eta[0,t]$ can
be naturally divided into two parts: the \textbf{left boundary} and the
\textbf{right boundary}. By definition, the left
boundary always contains a part of the wired boundary in the domain and the right
boundary contains a part of the free boundary of the domain.
Hence for $z,w
\in \bar D$, we can grow the curve until the times $\tau_z, \tau_w$ when it hits $z$ and $w$ respectively. One can think of
the curve as oriented towards its target. Then we define a (mixed) continuum UST $\cT$ (i.e., an element of the Schramm space) associated to $\eta$ by taking the path connecting $z$ and $w$ to be simply given by the left boundary of $\eta( \min(\tau_z, \tau_w), \max(\tau_z, \tau_w))$. Furthermore, since almost surely, SLE$_8$ visits Lebesgue almost every point only once, it is easy to deduce that the above map is a.s. continuous.
This implies that convergence in the Peano sense is stronger than in the Schramm sense for SLE$_8$.

One issue we have when applying \cref{thm:LSW_supp} is that the
convergence of the Peano curve is written for spanning trees with mixed boundary condition
only. Indeed, for fully wired boundary
(or fully free boundary in the dual), the convergence is no longer to
a chordal SLE$_8$ curve, but presumably to a certain space-filling loop. However, we could not locate a convergence result for
this loop in the literature, so we bypass this problem using the
following lemma.

\begin{lemma}
  \label{lem:mixed->free_supp}
Consider a sequence of arcs $\beta_n$ and $\alpha_n$ in $\partial D$ such that
$\beta_n \cup \alpha_n = \partial D$ with the diameter of $\alpha_n $ being
at most $1/n$. Let $\alpha^\d_{n}$ and $\beta^\d_{n}$ be a partition
of $\partial D^\d$ into connected segments which approximate
$\alpha_n$ and $\beta_n$ respectively as $\delta \to 0$ in the sense
of \cite{LSW}. Consider a mixed continuum tree $\cT_n$ obtained as the scaling limit of the
spanning tree
$\cT_n^\d$ in
$D^{\d}_n$ with wired boundary condition in $\beta_{n}^\d$ and
free boundary condition in $\alpha_{n}^\d$. (Such a scaling limit exist because the scaling limit of the Peano curve exists from \cite{LSW} and the Peano curve topology is stronger than Schramm topology).
For all $\ve>0$ there exists a coupling between the continuum mixed tree
$\cT_n$ and the continuum \emph{wired} UST $\cT_{w}$ in $D$ (constructed in Section \ref{SS:contUST_supp}) such that $\cT_n \cap (D \setminus
B(\alpha_n,\ve))  = \cT_w \cap (D \setminus
B(\alpha_n,\ve))$ for all $n >n_0(\ve)$ with probability at least
$1-\ve$.
In particular, $\cT_n$
converges
in probability to the continuum wired UST in $D$ as $n \to \infty$, in the
Schramm topology.
\end{lemma}

\begin{proof}[Proof of \cref{lem:mixed->free_supp}]
Let $\cT^\d$ denote the discrete tree with fully wired boundary condition.
  This lemma basically follows from Schramm's finiteness Theorem (Theorem
10.2 in \cite{SLE}, see also Lemma 4.18 in the main file).
For a fixed $\eps>0$, find a finite set $V$ of vertices (corresponding to $\cQ_{j}$ with $j = j_0(\eps)$ in the notations of Lemma 4.18 of the main file, for the domain $D^\d \setminus B(\alpha_n^\d, 2\ve)$). Choose $n_0(\eps)$ such that if $n \ge n_0(\eps)$, the arc $\alpha_n^\d$ is sufficiently small that for any $\delta < \delta_0(\eps)$, no random walk starting from $V$ hits $\partial D$ on $\alpha_n^\d$, with probability at least $1- \eps$. (This is possible because a $\delta \to 0$, the random walks starting from $V$ converge to a finite number of Brownian motions starting from a finite number of points in $D \setminus B(\alpha_n, 2 \eps)$).

On this event, which has probability at least $1- 2 \eps$, the branches from $\cT^\d$ and $\cT_n^\d$ emanating from $V$ can be taken to be identical thanks to Wilson's algorithm (which is valid both for $\cT^\d$ and $\cT_n^\d$, with the difference that in $\cT_n^\d$ the walks are reflected on $\alpha_n^\d$ instead of being killed). Moreover, by Lemma 4.18 in the main file, the choice of $V$ is such that simple random walk started from any other vertex $v$ will not exceed diameter $\eps$ before hitting one of the branches emanating from $V$, with probability at least $1-\eps$, for any $\delta \le \delta(\eps)$.

On these two events, all the branches from $V$ are identical in $\cT_n^\d$ and $\cT^\d$, and all branches from $D^\d \setminus B(\alpha_n^\d, \ve)$ can be taken to be the same in $\cT_n^\d$ and $\cT^\d$. Hence we have found a coupling so that $\cT_n^\d$ and $\cT^\d$ agree in $D^\d \setminus B(\alpha_n^\d, \ve)$ for any $\delta < \delta(\eps)$.
%
%
The result now follows by letting $\delta
\to 0$.
\end{proof}

\paragraph{Flow lines}

The theory of imaginary geometry developed by Miller and Sheffield in \cite{IG1, IG4} following earlier work of Dub\'edat \cite{Dub1} and Schramm--Sheffield \cite{SchrammSheffield} gives us a way to couple
SLE curve realised as flow lines of a Gaussian free field. Fix the constants:
$$\kappa =2; \quad \kappa' =
8; \quad \lambda' = \frac{\pi}{\sqrt{8}};\quad
\lambda = \frac{\pi}{\sqrt{2}}; \quad     \chi  = \chi(\kappa) =
\frac{2}{\sqrt{\kappa}} - \frac{\sqrt \kappa}{2}. $$
 The notations are consistent with the
notations in \cite{IG1, IG4}. We write $\chi = \chi(2)$, $\chi' = \chi(8) = - \chi(2)$.

Take a domain $D  \subset \C$ with two marked points on the boundary $a$ and
$b$ (possibly with a boundary which is rough). A canonical example is
$(\H,\infty,0)$ where $\H$ is the upper half plane. Fix $\alpha>0$ which we refer to as the \emph{angle} in what follows. Let $h_{\kappa'}$ be a GFF in $\H$ with
boundary condition given by $-\lambda'+\alpha \chi'$ on the positive
real line and $\lambda'+\alpha \chi'$
on the negative real line.  Let $h_\kappa$ be  a GFF in $\H$ with boundary condition $\lambda+ \alpha \chi$ on
the positive real line and $-\lambda+ \alpha \chi$ for the negative
real line for $\kappa$. Let $\varphi_D: D \to \H$ be a conformal equivalence
between $\H$ and $D$ which maps $b$ to $0$ and $a$ to $\infty$. Define a
\textbf{mixed intrinsic winding boundary condition} GFF with angle $\alpha$ on
$(D,a,b)$ with parameter $x\in \{\kappa,\kappa'\}$ to be
defined by
\begin{equation}
 h_x \circ \varphi - \chi(x)\arg(\varphi') \label{eq:conf_formula_supp}
\end{equation}
When we do not specify the angle, we mean $\alpha=0$ by default.
So that the reader remains in context, this boundary condition will correspond
to the mixed boundary uniform spanning tree (with a specific $\alpha$) where the boundary type changes at
$a$ and $b$. Although $\varphi$ is not unique, this definition defines a unique
law via scale invariance.

It can be checked that the mixed intrinsic winding boundary
condition $(\D,1,e^{i\theta_b})$ can be defined as follows:

\begin{align}
f_{1,b,\kappa'}(e^{i\theta}) =
\begin{cases}
 -\lambda'+\chi'(\theta+\frac{\pi}{2}) \text{ if }0 \le \theta \le
\theta_b\\
\lambda' + \chi'  (\theta+\frac{\pi}{2})\text{ if }\theta_b <\theta
<2\pi
\end{cases}
\end{align}
\def\k{\kappa}
Replacing $\kappa'$ by $\kappa$ and $\lambda'$ by $-\lambda$ we get a function
$f_{1,b, \kappa}$. We extend the definition of $f_{1,b,\k}$ (resp. $f_{1,b,
\k'}$) to $\theta_b = 2\pi$ by defining
$f_{1,1,\k}(e^{i\theta}) = \lambda+\chi (\theta+\pi/2) $ (resp.
$f_{1,1, \k'} (e^{i
\theta} ) =  -\lambda' + \chi' (\theta+\pi/2)$).
Hence $f_{1,1,\k}$ (resp. $f_{1,1,\kappa'}$) is the same as $\chi$ times (resp.
$\chi'$ times) the intrinsic winding boundary
condition plus a global shift. We
extend this to all marked domains $(D,a,b)$ via the coordinate
change formula \cref{eq:conf_formula_supp}; this defines functions $f_{a,b, \k}$ and
$f_{a,b , \k'}$ on this doubly-marked domain, and with a slight abuse of
notation we will still call $f_{a,b, x}$ the harmonic extension for $x \in \{
\kappa, \kappa'\}$.

%
%
%
%

 We are going to quote a special case of some theorems from \cite{IG4}.

\begin{thm}\label{thm:flow_lines_supp}
Let $(D,a,b)$ be a domain with two distinct marked points on its boundary. Let
$\{z_n\}_{n \in \N}$ be a countable dense set of interior points in $D$. Let
$h$ be a zero boundary $GFF$ in $D$ and let $f_x = f_{a,b,x}$ be as above.

Let
$\eta'[0,\infty)$ be an SLE$_{\kappa'}$ curve in $(D,a,b)$ from $b$ to $a$
parametrised by capacity. There exists a unique coupling between $h$,  $\eta'$ and a
set of simple curves $\{\eta^{(i)}[0,\tau_i]\}_{i \in \N}$ in $D$ starting at $z_i$ for each $i \ge 1$ (defined up to
monotone reparametrisation) where $\tau_i$ is the stopping time when the curves
hit $\partial D$, such that the following is true.

\begin{itemize}

\item[i] Let $\tau'_i$ be the stopping time when $\eta'$ hits $z_i$
($\tau'_i<\infty$ almost surely since $\eta'$ is space filling almost surely).
Then the left boundary of $\eta'[0,\tau'_i]$ is the same as $\eta_i$ almost
surely.

 \item[ii] For any stopping time $\tau'$ of $\eta'$, the conditional law of
$-h+f_{\kappa'}$ in $D \setminus \eta'[0,\tau']$ (notice that this is a
connected domain since SLE$_{\kappa'}$ is space filling almost surely) is given
by an independent GFF in $(D \setminus \eta'[0,\tau'],\eta'(\tau'),b)$ with
mixed winding boundary condition with parameter $\kappa'$. Such a curve
$\eta'$ coupled with $h$ is called a \textbf{counterflow line} in the terminology of of \cite{IG4}.

\item[iii] For any $k \ge 1$, for all $1 \le i,j\le k$, the curves $\eta^{(i)}$ and
$\eta^{(j)}$ have the coalescing property: if they intersect, they continue together until they hit $\partial D$.
In particular, $D \setminus \cup_{i=1}^k \eta^{(i)}[0,\tau_i]$ is a
connected domain. Furthermore, the conditional law of
$h+f_{\kappa}$
given
$\cup_{i=1}^k
\eta^{(i)}[0,\tau_i] $ is given by an an independent GFF in $(D \setminus
\cup_{i=1}^k
\eta^{(i)}[0,\tau_i],a,b)$
with mixed intrinsic winding boundary condition with
angle $\pi/2$.
These curves are the \textbf{flow lines of $h$ with angle $\pi/2$} in the terminology of \cite{IG4}. The full set of curves $ \eta^{(i)}$ is called a \textbf{tree of
flow lines} starting from the countable set of points $\{z_i\}_{i \in \N}$ and
angle $\pi/2$.

\item[iv]  The curve $\eta'$ determines $h$ and vice versa almost surely. The
curves
$(\eta^{(i)})_{i \in \N}$ determines $h$ and vice versa almost surely.

\end{itemize}

Furthermore in the case where $a=b$ we still get a coupling between $h$ and a
set of simple curves $\eta^{(i)}$ such that item iii above still holds.

\end{thm}

\begin{proof}
 The coupling between $\eta',h $
 is described
in Theorem 1.1 of \cite{IG1} or Theorem 6.4 in
\cite{Dub1}, while the coupling between $h$ and $\eta^{(i)}$ follows from Theorem 1.8 in \cite{IG4}. Note that the latter still holds in the case $a = b$.
The properties described in the coupling are special cases of
Theorem 1.8 and Theorem 1.13 of \cite{IG4}. That the curve $\eta'$ determines
$h$
and vice versa follows from Theorem 1.2 of \cite{IG1,IG4} and Theorem 1.16 of
\cite{IG4}. That the curves $\{\eta^{(i)}\}_{i \in \N}$ determine $h$ follows
from Theorem 1.10 of \cite{IG4}.
\end{proof}

At this point, we know that in the mixed boundary case and $a\neq b$, the coupling above holds for $h$, the Peano path $\eta'$ and the flow lines $\eta^{(i)}$. We know that for $\kappa = 2$, $\eta'$ is a chordal SLE$_8$ and hence the scaling limit of the discrete Peano path of a mixed UST. Since convergence in the Peano sense is stronger than in the Schramm sense, we know that the left and right boundaries of the path $\eta'$ are branches in the mixed continuum UST and its dual. By the above theorem  these are also the flow lines associated to $h$ with angle $\pm \pi/2$. We want to deduce by letting $b \to a$ that the tree of flow lines in the case $a=b$ described in the above theorem are the branches of the wired continuum UST.

Essentially, the argument goes as follows. By \cref{lem:mixed->free_supp} we know that as $b\to a$, the mixed continuum UST converges to the wired UST. The GFF determined by the branches of the mixed tree will be shown to converge. This requires an argument because the function which associates a field to a tree of flow lines is not known to be continuous, only measurable. However, the conditional law of the field with mixed boundary conditions, given a finite number of its flow lines, converges to that of a field with wired boundary conditions given these curves.
Since these conditional distributions characterise entirely the joint distribution between all the branches and the field, the result follows.


\begin{lemma}
 Let $(D,x)$ be a domain with a marked point.
Let
$\{z_i\}_{i \in \N}$ be a countable dense set of points in $D$.
Then there
exists a coupling between a GFF in $(D,x)$ and the continuum wired UST in $D$
such that flow line tree from
$\{z_i\}_{i \in
\N}$ with angle $\pi/2$ is equal to the branches of the continuum
wired UST from $\{z_i\}_{i
\in \N}$ almost surely.
\end{lemma}

\begin{proof}
Fix $k \in \N$ and any set of $k$ points from $\{z_i\}_{i \in \N}$. Without
loss of generality let this set be $\{z_1, \ldots, z_k\}$.
It suffices to find a coupling between $(h, \cT)$ such that in this coupling $h$ has $\chi$ times intrinsic winding boundary conditions and $\cT$ is a wired UST and such that for each $k \ge 1$ the branches $\tilde \eta_1, \ldots, \tilde \eta_k$ emanating from $z_1, \ldots, z_k$ are equal to the flow lines $\eta_1, \ldots , \eta_k$ with angle $\pi/2$.

To find such a coupling, take a sequence $y_n \in \partial D$ such that $y_n \to x$ and call $\alpha_n$ the arc between $y_n$ and $x$ (which converges to $\{x\}$) and $\beta_n$ the complementary arc (which converges to $\partial D$). Let $\cT_n$ be a mixed boundary continuum uniform spanning tree with free
boundary in the arc $\alpha_n $
and wired boundary in the rest. Let $\cT_n$ be coupled with a continuum wired
UST
denoted by $\cT$ in $(D,x)$ as in \cref{lem:mixed->free_supp}. The branches
$\eta_{n,i}$ of $\cT_n$ from $\{z_1,\ldots,z_k\}$  (until they hit the wired
boundary) in $(D,x,y_n)$ are unique almost surely and are equal to $\tilde
\eta_1, \ldots, \tilde \eta_k$ with high probability as $n\to \infty$. Let
$h_{n}  $ be the field determined by the branches of $\cT_n$ viewed as flow
lines (such an identification is possible when $n$ is finite and the boundary
conditions are mixed by \cref{thm:flow_lines_supp}). This defines a coupling of
$(h_{n}, \cT)$ such that the branches  $\tilde \eta_1, \ldots, \tilde \eta_k$
are equal with high probability to the flow lines of $h_{n}$ with angle
$\pi/2$. The marginal $h_{n}$ in this coupling is a mixed intrinsic
winding boundary GFF and so converges to a GFF with $\chi$ times intrinsic
winding minus $\lambda$ boundary conditions as $n\to \infty$. Hence $(h_{n},
\cT)$ is tight and we can consider a subsequential limit $(h, \cT)$. Notice that
in such a coupling necessarily we have that $\tilde \eta_1, \ldots, \tilde
\eta_k$ are flow lines of $h$, hence we have the desired coupling (this follows
easily from the fact that the law of
$(\eta_{n,1}, \ldots, \eta_{n,k})$ converges holds in total variation, see \cref{lem:mixed->free_supp}).
This completes the proof.
\end{proof}

We summarise the findings of this section in the following theorem:
\begin{thm}[Coupling]\label{thm:coupling_expanded_supp}
Let $D \subset \C$ be a domain which is not $\C$. Let $\cT$
 be a continuum spanning tree in $D$ with wired condition and root $x \in \partial D$. There exists a
coupling between the $\cT$, a Gaussian free field $h$ on $(D,x)$ with intrinsic winding
boundary condition such that $\cT$ is the flow line tree of $h$
with angle $\pi/2$.

Consequently, given $z_1,\ldots,z_k \in D$ and
the set of branches $\cB$ of $\cT$ from these points to
$x$, the conditional law of $h$ in $D \setminus \cB$ is given
by a GFF in $D \setminus \cB$ with intrinsic winding boundary
condition in $(D \setminus \cB,x)$.
Also,  $\cT$ determines $h$ and vice-versa.
\end{thm}
The second paragraph of \cref{thm:coupling_expanded_supp} is exactly the statement of \cref{thm:coupling_intro_supp}, which is what we want.
\section{Discrete estimates}
In this section, we prove the technical random walk estimates that we need in Section 4 of the main file.

\subsection{Winding of random walk}

\begin{lemma}[Lemma 4.7 in the main file]\label{lem:cond_winding_supp}
Fix $0<r <R$.
There exists
$\alpha =\alpha(R/r) >0$ and $c$ such that for all $x \in \C$, $\delta \in (0, c r
\delta_0)$, $v \in A(x, r + \frac{R-r}{3}, R - \frac{R-r}{3})^\d$,
for all $u$ such
that $\P_v(X_\tau = u) > 0$ where $\tau$ is the exit time of $A(x,r, R)$, and for all
$n \ge 1$, we have:
\[
\P_v\left[ \sup_{
 \cY \subset X[0, \tau]}
\abs{W( \cY, x) }\geq n | X_\tau = u
\right] \leq C (1- \alpha)^n .
\]
where the supremum is over all continuous paths $\cY$ obtained by erasing portions from $X[0, \tau]$.
%
\end{lemma}

\begin{proof}
The proof is divided in two steps.
First we construct a set $U'$ close to $u$ and macroscopic such that
if the random walk is started in $U'$ it has a good chance to exit the annulus through $u$. In the second step we use the conditional crossing estimate (Lemma 4.4
in the main file) to control the winding: every time the walk crosses the positive real line, the chance that it will wind once more before exiting at $u$ is small because there is a good chance to go hit the macroscopic set $U'$.


\mn \textbf{Step 1.} Let us assume first that $u$ is a point on the outer boundary of the annulus. Up to a rotation and a translation of the graph, we can assume that $x = 0$ and $u$ is on the negative real axis. We also simplify notations by writing $A = A(x, r, R)$. For a continuous path $\gamma$ we write $\gamma^\d$ for a discrete path staying at distance $2\delta/\delta_0$ of $\gamma$ (such a path can be constructed thanks to our Russo--Seymour--Welsh crossing assumption).

We start as in the proof of Lemma 4.4 in the main file by controlling the function $h(v) = \P_v[ X_\tau = u]$.
More precisely we claim that there exists $c>0$ such that if $\rho = \frac{R-r}{10}$, for all $\delta$ small enough, for all $a \in U := \big(A \cap \partial B(u, \rho)\big)^\d$ and $b \in U': = \big( \{z | \arg(z-u) \in [-\frac{\pi}{4}, \frac{\pi}{4}] \} \cap \partial B(u, \rho) \big)^\d$,
\begin{equation} \label{controlh_supp}
h(b) \geq c h(a).
\end{equation}
Indeed fix $b \in U'$ and $a \in U$.
Since $h$ is harmonic there exists a path $\gamma$ from $a$ to $u$ along which $h$ is non-decreasing and we can assume that $\gamma$ does not intersect $U$ outside of $a$ (otherwise change $a$ to be the last intersection point). Let $\tau_\gamma$ be the first hitting time of $\gamma \cup \partial A$ by the simple random walk. By harmonicity we have
\[
h(b) = \E_b[h(X_{\tau_\gamma})] .
\]
On the other hand, by the crossing estimate, the random walk has a positive probability $c$ independent of $\rho$ to hit $\gamma$ irrespective of the relative positions of $a,b, u$. (For example this is at least the probability to surround $\partial B(u, \rho) \cap A$ in the clockwise and anticlockwise directions, staying in the annulus $A(u, 99\rho/100, 101\rho /100)$.) Hence using this event to lower bound the expectation, we find $h(b) = \E_b[h(X_{\tau_\gamma})]  \ge h(a) c$. This proves the claim \eqref{controlh_supp}.

\mn \textbf{Step 2.} It will be easier to bound crossings of the real line. Let $\ell_+$ be a path connecting two boundary pieces of $A$ and approximating the interval $[r,R]$. Here approximating means the path stays within $O(\delta/\delta_0)$ from $\ell_+$ and the existence of such a path is guaranteed by the crossing estimate. Similarly, define $\ell_-$ approximating $[-R,-r]$.

Let $\tau_{+,1}$ be the first hitting time of $\ell_+$ and by induction let $\tau_{-,i}$ the first hitting time of $\ell_-$ after $\tau_{+,i}$ and $\tau_{+, i}$ the hitting time of $\ell_+$ after $\tau_{-, i-1}$. Let $I_+$ the number of $\tau_+$ before $\tau$, i.e $I_+ = \abs{\{ i | \tau_{+,i} \leq \tau\}}$.

Note that there exists $\alpha \in (0,1)$ depending only on $R/r$ such that for any $w \in \ell_+$,
$$\P_w( \tau_{-,1} < \tau) \le 1-\alpha.$$
Hence applying the strong Markov property $n$ times,
\begin{equation}\label{eq:hitlines_supp}
\P_v( \tau_{+,n} \le \tau; X_\tau = u) \le (1-\alpha)^n \sup_{w \in \ell_+} \P_w( X_\tau = u).
\end{equation}
However, we claim that for any $w \in \ell_+$, $\P_w( X_\tau = u) = h(w) \le  h(v)/\alpha $ where $\alpha $ depends only on $R/r$, whenever $v \in A (0, r+ (R-r)/3, R - (R-r)/3)$. Indeed, by the maximum principle, $h(w) \le \sup_{U} h \le (1/c) \inf_{U'} h$ by \eqref{controlh_supp}. On the other hand, $h(v) \ge \alpha \inf_U h$ since there is a uniformly positive chance $\alpha$ of hitting $U$ before $\tau$ whenever $v$ is in the allowed region. Hence $h(v) \ge \alpha h(w)$ for any $w \in \ell_+$ and $v$ in the allowed region. Plugging into \eqref{eq:hitlines_supp}, we deduce
$$
\P_v( \tau_{+,n} \le \tau| X_\tau = u) \le (1/\alpha) (1- \alpha)^n.
$$
However notice that deterministically,
$$
\sup_{\cY} |W (\cY, x) | \le 2\pi (I_+ +1)
$$
and hence the result follows.

Finally it is clear that the above proof extends to the case where $u$ is a point in the inner boundary instead of being in the outer boundary.
\end{proof}

The next lemma is Lemma 4.8 in the main file which, we recall, is an analogue of \cref{lem:cond_winding_supp} for the largest scale.

\begin{lemma}[Lemma 4.8 in the main file]\label{lem:winding_macro_supp}
Fix $r< R$
There exists
$\alpha =\alpha(R/r) >0$ and $c,C$ such that for all $\delta \le c r \delta_0$, for all $x \in \C$, for all domains $D$ such that $\dist( x, \partial D^\d ) = R$ (in particular $D \supset B(x,R)$), for all
 $v \in A(x, r + \frac{R-r}{3}, R - \frac{R-r}{3})^\d$, writing $\tau$ for the exit time of $D\setminus B(x, r)$,
\begin{equation}\label{eq:cond_wind_dom_supp}
\forall n \ge 1, \quad \quad \P_v\big(\sup_{
 \cY \subset X[0, \tau]}
\abs{W( \cY, x) }\geq n \big| X_\tau \in \partial D^\d
\big) \leq C (1-\alpha)^n ,
\end{equation}
and for all $u \in B(x, r)$ such that $\P( X_\tau = u ) > 0$,
\begin{equation}\label{eq:cond_wind_inside_supp}
\forall n \ge 1, \quad \quad \P_v\big(\sup_{
 \cY \subset X[0, \tau]}
\abs{W( \cY, x) }\geq n \big| X_\tau= u
\big) \leq C (1-\alpha)^n .
\end{equation}
In both cases, the supremum is over all continuous paths $\cY$ obtained by erasing portions from $X[0, \tau]$.
\end{lemma}

Before we start the proof, we remind the reader that the above result \eqref{eq:cond_wind_dom_supp} could be wrong if we conditioned on the precise exit point $X_\tau \in \partial D^\d$, so it is important to work with the above formulation.

\begin{proof}
We start with the proof of \eqref{eq:cond_wind_dom_supp}. Note that
doing a full turn outside $B(x, R)$ implies necessarily that the walk has left $D^\d$ since $\dist(x, \partial D^\d) = R$. This readily implies
 $\P( X_\tau \in \partial D^\d ) \ge 1-\alpha(R/r)$;  hence in order to show \eqref{eq:cond_wind_dom_supp} it suffices to prove the corresponding unconditional statement:
 \begin{equation}\label{uncondWind_supp}
\forall n \ge 1, \quad \quad \P_v\big(\sup_{
 \cY \subset X[0, \tau]}
\abs{W( \cY, x) }\geq n
\big) \leq C (1-\alpha)^n .
\end{equation}
This is easy to see: indeed, as in the proof of \cref{lem:cond_winding_supp}, let $\ell_+ = [r, \infty)^\d$ and and $\ell_- = (-\infty, -r]^\d$ be the \emph{infinite} half lines starting from the inner radius to infinity. Every time the walk hits $\ell_+$ there is a probability at least $\alpha = \alpha (R/r)>0$ that the walk will leave $D$ without touching $\ell_-$ and then $\ell_+$ again (and hence without making another turn): let $\tau_{+, i}$ and $\tau_{-, i}$ be defined as before. It is clear from the uniform crossing estimate that there exists $\alpha = \alpha(R/r) > 0$ such that
\begin{equation}\label{fullturn_supp}
\forall v \in \ell_+, \quad \P_v[ X[0, \tau_{+, 1} ] \text{ does a full turn outside } B(x, R) ] \geq  \alpha.
\end{equation}
(Note that this holds even if $v$ is quite far away from $B(x,R)$ by taking rectangles of sufficiently large scale in the crossing assumption). Iterating and applying the Markov property immediately implies \eqref{uncondWind_supp}.

We now turn to \eqref{eq:cond_wind_inside_supp}. This is essentially the same as \cref{lem:cond_winding_supp}, with a few modifications. We take $\ell_-, \ell_+$ to be the infinite half-lines as above, and note that
$$
\P( \tau_{+,n} \le \tau) \le (1- \alpha)^n
$$
by \eqref{fullturn_supp}. We conclude as in the proof of \cref{lem:cond_winding_supp}, noting that the maximum principle applies in any domain, and that if $v \in A(x, r+ (R-r)/3, R- (R-r)/3)$ the probability that the walk will hit the set $U'$ is still uniformly bounded below by $\alpha (R/r)$.
\end{proof}

\begin{lemma}[Lemma 4.9 in the main file]\label{lem:winding_boundary_supp}
There exists $\alpha \in (0,1)$ and $c,C>0$ depending only on the constants in the uniform crossing conditions
such that the following holds. For all $\delta <c\delta_0R$,
\[
\forall n \ge 1, \quad \quad \P_v\big( \big| W(\gamma_v^\d[-1, 0], v) - \E  W(\gamma_v^\d[-1, 0], v) \big| \geq n \big) \leq C (1-\alpha)^n .
\]
\end{lemma}
\begin{proof}
We drop the superscript $\cdot ^\d$ for convenience and use a replica technique: let $\tilde \gamma_v$ be an independent realisation of  $\gamma_v$. Clearly it suffices to show that
\[
\forall n \ge 1, \quad \quad \P_v\big( | W(\gamma_v[-1, 0], v) - W( \tilde \gamma_v([-1,0] , v) | \geq n \big) \leq C (1-\alpha)^n .
\]
Let $\lambda$ be a continuous parametrisation of $\partial D$ (as usual, this exists because of our assumption that $\partial D$ is locally connected). Then note that if we choose $S,S'$ such that $\lambda(S) = \gamma_v(0)$ and $\lambda(S') = \tilde \gamma_v(0)$ then by additivity of winding
$$
|W(\gamma_v[-1, 0], v) - W( \tilde \gamma_v([-1,0] , v)| = |W( \lambda([S, S']), v )|
$$

 Fix a domain $D^\d$ and a starting point $v$, let $R = \dist( v, \partial D^\d)$ and let $r = R/2$, we let the constant $c,C,\alpha$ be chosen so that \cref{lem:winding_macro_supp} applies. Take two independent random walks starting from $v$ run until they leave $D$, and consider the times at which they each exit $B(v,3R/4)$ for the first time after the last exit from $B(v,R/2)$. Let
 $X,X'$ be the continuation of the two walks from that point onwards and let $w,w'$ be their starting points respectively on $\partial B(v, 3R/4)^\d$.
 Observe that both $X,X'$  are independent random walks, starting from $w,w' \in A(v, r + \frac{R-r}{3}, R - \frac{R-r}{3})^\d$ and conditioned to exit $D\setminus B(v,r)$ through $ \partial D^\d$. We now make use of the following topological fact:
 \[
| W( \lambda[S, S'] ) | \leq 4 \pi + \sup_{
 \cY \subset X} | W( \cY, v ) | + \sup_{
 \cY' \subset X'} | W( \cY', v ) |.
\]
which immediately proves the lemma using \cref{lem:winding_macro_supp}. To see the above inequality, we simply point out that the winding of any loop around any point is bounded by $2\pi$ (note that this is true even if the loop is nonsimple, so long as it is noncrossing, and that $\lambda$ is such a curve). Also note that if the walks $X,X'$ do not intersect each other, then we can form a simple loop by joining $w, w'$ by a circular arc,
the loop erasures $\cY$ and $\cY'$ of $X$ and $X'$ after the last time they touch this circular arc, and $\lambda([S,S'])$. If the walks do intersect, then we can do the same but looking at loop erasures of the walks starting from their last intersection point. This completes the proof.
 \end{proof}

\subsection{Proof of exponential tail of loop-erased winding (Proposition 4.12)}
\label{SS:proofwindingtail_supp}

\begin{figure}[p]\label{fig:proof_ht_supp}
\begin{center}
\def\svgwidth{0.7\textwidth}
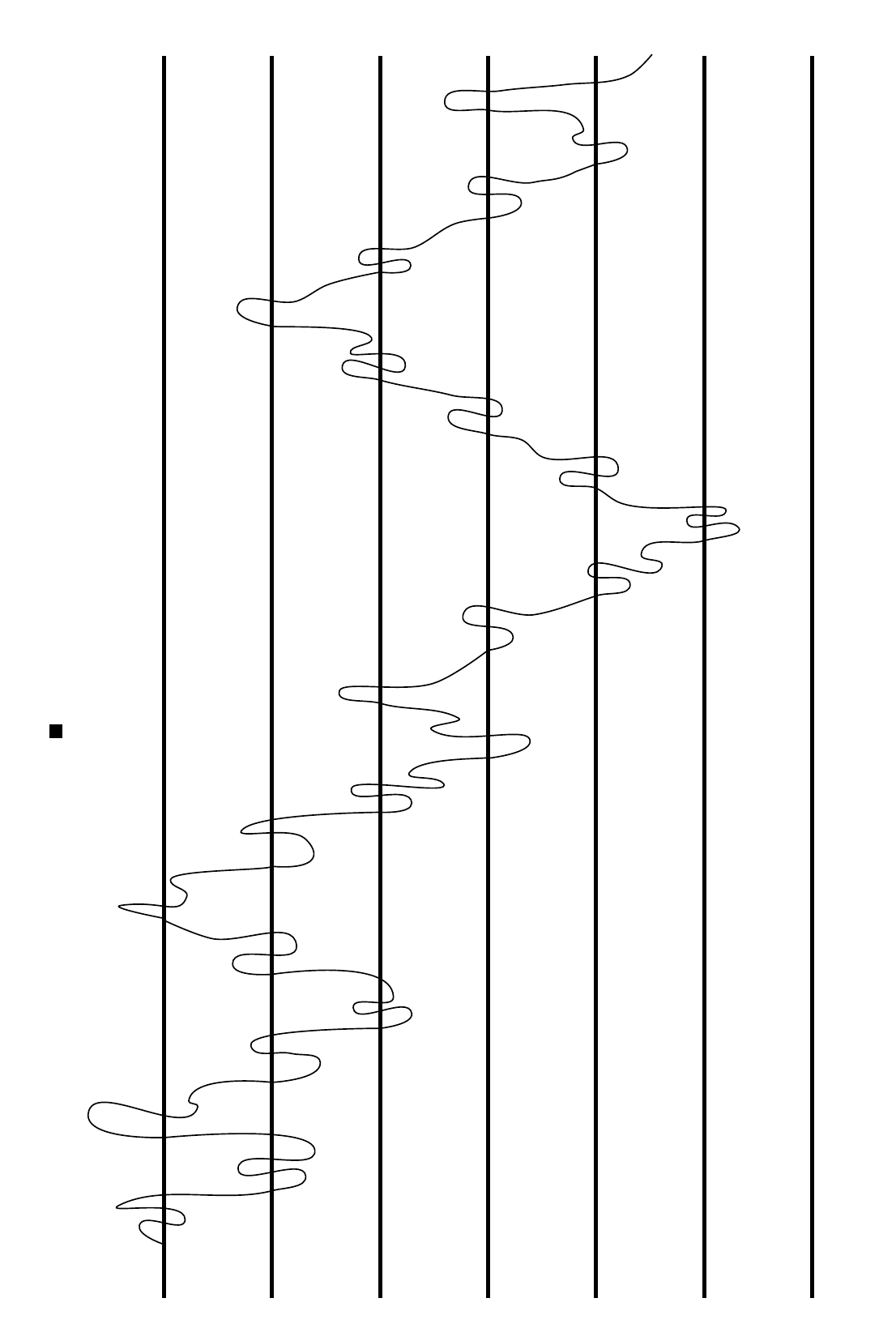
\caption{Schematic illustration of the proofs in Section \ref{SS:proofwindingtail_supp}. Time runs upward. The times $\tau_k$ are indicated by the boxes, with a blue circular arrows above time $\tau_k$ indicating when $X[\tau_k, \tau_{k+1}]$ does a full turn in an annulus. The horizontal blue half lines indicate that everything below these lines is erased and ignored in the decomposition of Lemma \ref{claim:erasure_supp}. In this example, there is no full turn at scale $2$ after $\kappa_1$ so we have $\kappa_{2} = \kappa_1 +1$ as indicated. There is no crossing of $C_4$ after $\kappa_3$ so $\kappa_4 = \infty$ and $I = 3$, and the event pictured is exactly how we see that $I$ has exponential tail. The red boxes denote the set $\cS^1$ in the proof of Lemma \ref{claim:size_G_supp}.}
\end{center}
\end{figure}

We now provide proofs of the sub-lemmas stated in the main file for the proof of Proposition 4.12 or, more precisely, Lemma 4.13.

\begin{lemma}[Lemma 4.14 in the main file]\label{claim:erasure_supp}
Recall $\cY, B_0$ from Lemma 4.13 (and above) in the main file. Then
\[
\cY \cap B_0^c \subset  \bigcup_{0\leq i \leq I} \bigcup_{k \in \cG_i} X[\tau_{k}, \tau_{k+1}].
\]
Furthermore, one can write $\cY {\cap B_0^c} = \cup_{i \le I} \cup_{k \in \cup \cG_i} \cY_{k}$ where
$\cY_k$ are disjoint intervals of the loop erased
random walk of the form $\cY_k = (Y_{j_k},Y_{j_k+1}, \ldots, Y_{j_k+i_k}) $ and $\cY_k \subset
X[\tau_k,\tau_{k+1}]$.
\end{lemma}
\begin{proof}
We consider the chronological erasure of loops so at each time $T$ we have a loop-erased path $Y^{(T)}$ obtained by erasing loops from the random walk $(X_t)_{t \leq T}$.

Notice that by construction, after time $\tau_{k_{\max}}$ the random walk does not return to $B(0, r_1)$ so $ \cY \subset Y^{(\tau_{k_{\max}})}$ (it might be smaller because
more loops can occur, erasing further the final set $\cY$, but no more points can be added to $\cY$ since we do not return to the annulus). This justifies looking only at $k \leq k_{\max}$ above.

Observe that if $X$ does a full turn between times $\tau_k$ and $\tau_{k+1}$, then $Y^{(\tau_{k+1})} \subset B(v, r_{i(k) +1}) $. Indeed if $Y^{(\tau_{k})} \subset B(v, r_{i(k) +1}) $ the statement is trivial. Otherwise $Y^{(\tau_{k})}$ includes a path from $C_{i(k)-1}$ to $C_{i(k)+1}$ which has to be crossed by the full turn at some time $T$ and this erases everything outside of $B(v, r_{i(k) +1})$.

In particular, we see that $Y^{\tau_{\kappa_{-1}+1}} \subset B(v, r_0)$. This
implies that ${\cY  \cap B_0^c} \subset X[\tau_{\kappa_{-1}+1}, \tau_{k_{\max}}]$ since all
the rest of $X$ either was erased or contributes to parts of $Y$ outside of
$\cY \cap B_0^c$. Using the same argument, we can erase the parts of the walk before
$\tau_{\kappa_0}$ that lie outside of $C_1$. But we need to keep all the visits
to $C_0$ (that is all the elementary pieces of random walk in $\cG_0$); so we
get
\[
{\cY \cap B_0^c} \subset \bigcup_{k \in \cG_0}  X[\tau_{k}, \tau_{k+1}] \cup
X[\tau_{\kappa_{0}+1}, \tau_{k_{\max}}].
\]
We now proceed by induction to complete the proof of the first part of the
statement.

\smallskip

For the last sentence, note that the edges of $Y$ are created by
the random walk in order from $v$ to the boundary.
Therefore the set of edges created in an interval of time by the random walk
has to form an interval in the loop-erased walk $Y$. We emphasise that here we
only look at the times when edges are added to $Y$, the full loop-erased walk,
and not the loop-erased path $Y^{(t)}$ at any intermediate point of time.
\end{proof}

The next delicate lemma in the proof of Lemma 4.13 in the main file is the following:

\begin{lemma}[Lemma 4.15 in the main file]\label{claim:size_G_supp}
 There exists $C,c,c'>0$, such that, for all $\delta \leq ce^{-t} \dist(v, \partial D^\d) \delta_0$, for all $n>0$,
 \[
 \P(\sum_{0 \le j \le I} \abs{\cG_j} \geq n) \leq C \exp(-c'n).
\]
\end{lemma}
\begin{proof}
  Fix a $\delta<ce^{-t}\delta_0$ with $c$ small enough that all the estimates below are valid for such a $\delta$.
First we observe that there exist $c_1>0$ such that
\begin{equation}
\P(I \ge n) \le e^{-c_1n} . \label{eq:I_supp}
\end{equation}
Indeed, condition on the sequence of crossing positions $\cS = ( X_{\tau_k})_{k \geq 0}$. Recall that by construction $i(k_{\max} -1) = 1$, therefore for every $i \geq 2$, just after the last crossing $\tau_k$ of $C_i$ before $k_{\max}$ the random walk goes to $C_{i-1}$. In particular if there is a full turn within the annulus $A(v, r_i, r_{i+1})$ during $[\tau_k, \tau_{k+1}]$, then  by definition $\kappa_{i+1} =+ \infty$. By Lemma 4.4 and
Corollary 4.5 in the main file, the conditional probability of making such a full turn has a uniformly positive probability to occur for each $i$ so we are done.

\smallskip

To proceed, we have to work on the law of $\cS$ so it is more tricky. Note that by Corollary 4.5 in the main file again, we immediately see that the number of crossings of $C_i$ after $\tau_{\kappa_i}$ has geometric tail. Indeed, conditionally on $\cS$, the event that there are $n$ crossings of $C_i$ after $\tau_{\kappa_i}$ is the event that the last $n$ crossings of $C_i$ -- which are measurable with respect to $\cS$ -- were not followed by a full turn in the annulus $A(v, r_i, r_{i+1})$. But a full turn in this annulus has uniformly positive probability by
Corollary 4.5 in the main file at each visit, conditionally on $\cS$. 
So this conditional probability has geometric tail, and so does the unconditional probability.

However $\abs{\cG_i}$ is the number of visits to $C_i$ after $\kappa_{i-1}$, so we also need to exclude the possibility that the random walk alternates many times between $C_i$ and $C_{i+1}$ without crossing $C_{i-1}$. We are actually going to prove that the number of crossings of $C_{i}$ between two successive crossings of $C_{i-1}$ has an exponential tail which is uniform over $i$. The difficulty comes from the fact that we are looking close to random times which are not stopping times, so we don't know the law of the walk. To get around this issue we will need to be more careful and discover the set $\cS$ in steps.

For any $j\ge 0$, let us define $\tau^j$, $i^j$ and $\cS^j$ as before but using $C_j$ as the reference smallest circle rather than $C_{-1}$ as above. More precisely, we define $\tau^{j}_0 = 0$, $\tau^j_1$ to be the first crossing of $C_j$. Then by induction, whenever $i^j(k) = j$ we define $\tau^j_{k+1}$ to be the first crossing of $C_{j+1}$ after $\tau^j_k$ and otherwise if $i^j(k) > j$ then we set $\tau^j_{k+1}$ to be the first crossing of $C_{i^j(k) \pm 1}$ after $\tau^j_k$.
Finally we set $\cS^j = (X(\tau^j_k))_{k}$. If needed, let us modify this a bit and define a sequence $\tilde{\cS}^j$ by inserting in the sequence $\cS^j$ the point $X(\tau_{k_{\max}-1})$ where the last crossing of $C_1$ happens. Let $\{\tilde{\tau}^j_k\}_{k\ge 0}$ be the times corresponding to the sequence $\tilde{\cS}^j$.

As before, conditionally on $\tilde{\cS^j}$ the pieces of random walk $X[\tilde{\tau}^j_k, \tilde{\tau}^j_{k+1}]$ are independent. For any $k$ such that $i^j(k)= j$, the piece of random walk $X[\tilde{\tau}^j_k, \tilde{\tau}^j_{k+1}]$ is distributed as a random walk starting at some point on $C_j$ and conditioned on its exit point of $B(v, r_{j+1})$. By Corollary 4.6 in the main file, this random walk has a strictly positive probability to intersect $C_{j-1}$ and this probability is independent of everything else.
Therefore in $\tilde{\cS}^j$, looking at the successive visits to $C_j$, the gaps between visits where the walk also crossed $C_{j-1}$ have an uniform exponential tail, even if we condition on $\tilde \cS_j$.
Let $\Xi^j_1,\Xi^j_2,\ldots$ be the number of such visits to $C_j$ in between visits to $C_{j-1}$. Since conditioned on $\tilde{\cS}^j$, the walks are independent, we conclude from the above discussion that $\Xi^j_1,\Xi^j_2,\ldots$ are independent and have an exponential tail which is uniform and independent of everything else.


Now let $\cI_{j-1} = \{k \in \cV_{j-1}: k \ge \kappa_{j-1}\}$ (which has uniform geometric tail by \eqref{eq:I_supp}). Thus
$$
|\cG_j| \le \sum_{1\le i \le |\cI_{j-1}|}   \Xi_i
$$
We conclude that $|\cG_j|$ has an exponential tail which is uniform over everything else. In other words, there exist constants $c_2,C_2>0$ such that for all $j \ge 1$
\begin{equation}
\E(e^{c_2|\cG_j|}| \tilde {\cS}^j) \le C_2 \label{G_supp}
\end{equation}
Now we need to sum $|\cG_j|$ over $1 \le j \le I$. Note that $\cG_j$ are not independent so we cannot directly obtain the desired exponential tails (although stretched exponential tails follow automatically). However, there is enough independence so that we can reveal information step by step and bound the probabilities. Note that $\cS \supseteq \tilde{\cS^0}\supseteq\tilde {\cS^1} \supseteq \ldots$ and $\cG_j$ is measurable with respect to $\tilde{\cS}^{j-1}$. Therefore for any $m \ge 1$, discovering the $\tilde \cS^j$ successively with decreasing $j$,
\begin{align*}
\E(e^{c_2\sum_{1 \le j \le m} |\cG_j |}) & = \E(e^{c_2\sum_{2 \le j \le m} |\cG_j |}  \E (e^{c_2 |\cG_1|} | \tilde {\cS^1}   )) \\
 & \le C_2\E(e^{c_2\sum_{2 \le j \le m} |\cG_j|} ) \le C_2^m
\end{align*}
using \eqref{G_supp} for the first inequality and the final inequality is done by iterating. Therefore
\begin{align*}
\P(\sum_{0 \le j \le I} |\cG_j | > n ) & \le \P(\sum_{0 \le j \le \ve n} |\cG_j| > n ) + \P(I \ge \ve n) \\
& \le (C_2)^{\ve n}e^{-n} + e^{-c_1\ve n}
\end{align*}
The result follows by choosing $\ve$ small enough depending only on $C_2$.
\end{proof}

\section{Comparisons of capacity}\label{SS:capacity}

In this section we prove \cref{lemma:coupling_Poisson_supp} which is Lemma 4.25 from the main file. Recall the setup: we have two domains $D$ and $\tilde D$, $z_0 \in D^\d \cap \tilde D^\d$ and we couple the brances $\gamma$ and $\tilde \gamma$ starting from $z_0$ for the wired USTs in $D^\d$ and $\tilde D^\d$ respectively.

Firstly, we drop the $\delta$ in this section for notational convenience.
Since we are only interested in a single point, it is also convenient to drop the linear shift in the parametrisation of $\gamma,\tilde \gamma$ by $\log R( z_0, D), \log R(z_0,\tilde D)$ respectively, hence we do it in this Section. As explained in the main file (and also in the following paragraph), the idea is to compare the rate of \emph{change} of capacity in the two domains, so this global shift in parametrisation is immaterial in the argument.

 For $z \in \gamma \cap \tilde \gamma$, call respectively $T$ and $\tilde T$ the corresponding capacity in $\gamma$ and $\tilde \gamma$ seen from their respective starting point up to $z$. The idea will be that $T$ and $\tilde T$ are so close as a function of $z$ that, if we sample the capacity at which we cut randomly (using a Poisson point process in practice), the law of corresponding random points along $\gamma$ and $\tilde \gamma$ are very close in total variation and can be made to agree exactly. The technical difficulty with that strategy is that a uniform bound on $T-\tilde T$ is not sufficient, we need to show that ``the derivative of $T-\tilde T$ with respect to $z$'' is small.

We first need some technical estimates on the comparison of the two capacities.
In this section $D$ will be a fixed domain. We denote by $\gamma$ and $\tilde \gamma$ two simple paths from a point $z_0 \in D$ to respectively $\partial D$ and infinity and we will always assume that there exists $I$ such that $\gamma \cap B(z_0, e^{-I}) = \tilde  \gamma \cap B(z_0, e^{-I})$. For points $z, z' \in \gamma$ we will write $\gamma(z, z')$ for the part of the path between $z$ and $z'$, $\gamma(z, \partial D)$ for the path between $z$ and the boundary and similarly for $\tilde \gamma$. We denote by $T$ and $\tilde T$ the capacity in $D$ or $\C$, i.e $T(z) = -\log R( z_0, D \setminus \gamma(z, \partial D) )$ and $\tilde T (z)= -\log R( z_0, \C \setminus \tilde \gamma(z, \infty) )$, where recall that $R(z,D)$ denotes the conformal radius of $z$ in $D$.

The following lemma is elementary and well known.

\begin{lemma}\label{lemma:representation_capacity_supp}
Let $D$ be a simply connected domain and $z_0$ in $D$. Let $B_t$ be a Brownian motion started in $z_0$ and $\tau$ the exit time of $D$. We have
\[
  \log R(z_0, D) = \E(\log \abs{B_\tau - z_0}).
\]
\end{lemma}
\begin{proof}
Let $g$ be a map sending $D$ to $\D$ and $z_0$ to $0$. Let $\phi(z)=\frac{g(z)}{z-z_0}$ for $z \neq z_0$ and $\phi(z_0) = g'(z_0)$. Then $\phi$ is holomorphic in $D$ therefore $\log \abs{ \phi} = \log \abs{g} - \log \abs{z-z_0}$ is harmonic. Furthermore by definition $R(z_0, D) = 1/|\phi(z_0)|$ and for $z \in \partial D$, $\abs{g(z)} = 1$ hence $
\log |\phi |(z) = - \log |z-z_0|$ on $\partial D$.

Applying the optional stopping theorem, we deduce that
$$
- \log R(z_0,D) = \E( - \log |B_{t \wedge \tau} - z_0|)
$$
Note that $\tau$ always has polynomial tail (even if $D$ is unbounded) and hence $\sup_{t \le \tau} \log |B_{t\wedge \tau } - z_0|$ has exponential tail and hence is integrable. So the use of the dominated convergence theorem is justified and hence the result follows.
\end{proof}

\begin{lemma}\label{lemma:absolute_difference_supp}
There exist positive constants $C, c$ such that for any pair of paths
$\gamma$ and $\tilde \gamma$ and any $z \in \gamma \cap B(z_0, e^{-I})$, we have
\[
 \abs{T(z) - \tilde T(z)} \leq C(T(z)+\tilde T(z)) e^{-c(T(z)-I)} .
\]
In particular, if $T-I$ is large enough, then
\[
\abs{T(z) - \tilde T(z)} \leq C' T(z) e^{-c(T(z)-I)}.
\]
\end{lemma}
\begin{proof}
Let $B_t$ be a Brownian motion starting from $z_0$ and let $\tau$ and $\tilde \tau$ denote respectively the hitting time of $\partial D \cup \gamma(z, \partial D)$ and $\tilde \gamma(z, \infty)$. Let $\tau_I$ denote the exit time of $B(z_0, e^{-I})$.

By \cref{lemma:representation_capacity_supp}, $T(z) - \tilde T(z) = \E_{z_0}[-\log \abs{B_\tau - z_0}  + \log \abs{B_{\tilde \tau} - z_0}]$ but the right hand side is $0$ whenever $\tau \leq \tau_I$. Furthermore, $d(z_0, \gamma(z, \partial D) ) \leq e^{-T(z)}$ by Schwarz's lemma. Hence let $z'$ be such that $z' \in \gamma(z, \partial D)$ and such that $|z'- z_0| \le e^{- T(z)}$.
Then by the Beurling estimate (see \cite[Theorem 3.76]{Lawler}), $\P(\tau_I \geq
\tau) \leq C \sqrt{|z'-z_0|/e^{-I}} \le C e^{-(1/2) (T-I)}$. Also by Beurling estimate, $\log|(B_\tau - z_0)|$ has an exponential tail and therefore its second moment is comparable to the square of its first moment. Hence we obtain, applying Cauchy--Schwarz,
\begin{multline*}
|\E\log|(B_\tau - z_0)/(B_{\tilde \tau} -z_0)| \mathbbm{1}_{\tau \ge \tau_I}| \\\le \sqrt{\E\log^2 |(B_\tau - z_0)/(B_{\tilde \tau} -z_0)|  \P(\tau \ge \tau_I)} \le C(T(z)+\tilde T (z))e^{-c'(T(z)-I)}.
\end{multline*}
This completes the proof as the second inequality is easily derivable from the first.
\end{proof}

Now we will compare the growth rate of the capacities of $\gamma$ and $\tilde \gamma$.
We will also need to introduce $\bar T(z) = - \log R (z_0; B(z_0, e^{-I} ) \setminus \gamma(z, \partial D))$. In other words this is the capacity of $\gamma(z, \partial D)$ within $B(z_0, e^{-I})$.
The next lemma depends on the geometry of the curves so it is convenient to go to the unit disc. Let $\ph: D \to \D$ be the conformal map to the unit disc sending $z_0$ to 0 and such that $\ph'(z_0) > 0$.

We parametrise the curve $\gamma$ by capacity: for $T \ge T_0 = - \log R(z_0, D)$, let $z_T$ be the point on $\gamma$ such that $T(z_T) = T$ and define $g_T: D \setminus  \gamma(z_T, \partial D) \to \D$. Let $K_T = \ph(\gamma(z_T, \partial D))$ for $T\ge T_0$. Then $(K_T)_{T \ge T_0}$ is a growing family of hulls in the unit disc. Let $(g_T)_{T\ge T_0}$ be the (radial) Loewner flow in $\D$ and $(W_T)_{T \ge T_0}$ be the driving function (with $W_T \in \R$) corresponding to $(K_T)_{T \ge T_0}$. Also set $B' = \ph (B(z_0, e^{-I}))$, $H_T = \D \setminus K_T$ and finally set $L_T = g_T( H_T \setminus B')$.

\begin{figure}
\begin{center}
\includegraphics[width = 0.8\textwidth]{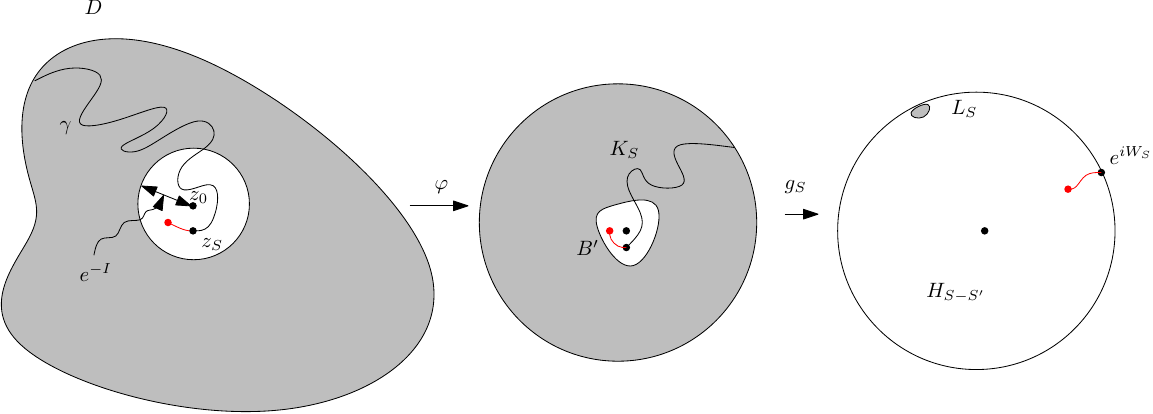}
\end{center}
\caption{Sketch of the maps in the proof of \cref{lemma:derivative_capacity_supp} }
\label{fig:cap_supp}
\end{figure}

\begin{lemma}\label{lemma:derivative_capacity_supp}
There exist positive constants $A, a,\ve_0$ such that the following holds.
Let $S$ and $S'$ be fixed with $0 < S' - S < \ve_0$. Let
$$
d = \inf_{S \leq T
\leq S' } d(L_T, e^{iW_T}).
$$ Let $w = z_S$ and $w'= z_{S'}$. If $S-I$ is large
enough and if $d$ is sufficiently large so that both $d \geq A e^{-\frac{a}{4}(S-I)}$
and $d^2  \ge A(S'-S)$ hold, then
\[
 1- Ae^{-a(S-I)}/d^2 < \frac{\bar T(w') - \bar T(w)}{S' - S} < 1 + Ae^{-a(S -
I)}/d^2.
\]
\end{lemma}

\noindent \emph{Proof.}
Let $\rho$
be the image of $\gamma(w', w)$ under $g_S \circ \varphi$, i.e, the red curve
in the rightmost image of
\cref{fig:cap_supp}.     First we note that by definition
\begin{align}
 \bar T(w) - S & = - \log R(0, \D\setminus
L_{S}), \label{eq:der1_supp} \\
\bar T (w') - S & = - \log R(0, \D \setminus (L_{S} \cup
\rho)) \label{eq:der2_supp},\\
S'-S & = - \log R(0, \D \setminus
\rho ).\label{eq:der3_supp}
\end{align}

Let $\tau_\rho$, $\tau_L$ and $\tau_\partial$ denote the hitting time of
respectively $\rho$, $L_S$ and $\partial \D$ by a Brownian motion $B$ starting from zero. Combining \eqref{eq:der2_supp},
\eqref{eq:der1_supp}, \eqref{eq:der3_supp} and \cref{lemma:representation_capacity_supp},
we have
\begin{equation}
\bar T (w') - \bar T (w) - (S' - S) = \E^0\left[ -\log \abs{B_{\tau_\partial
\wedge \tau_L \wedge \tau_\rho}} +\log \abs{B_{\tau_\partial \wedge \tau_\rho}}
+\log \abs{B_{\tau_\partial \wedge \tau_L}} \right] . \label{eq:key_eq_supp}
\end{equation}
Hence it suffices to show that
\begin{equation}\label{goalcap_supp}
\abs{ \E^0\left[ -\log \abs{B_{\tau_\partial
\wedge \tau_L \wedge \tau_\rho}} +\log \abs{B_{\tau_\partial \wedge \tau_\rho}}
+\log \abs{B_{\tau_\partial \wedge \tau_L}} \right]  }\le \frac{Ce^{-c(S-I)}}{d^2} (S'-S).
\end{equation}

Notice that if
$\tau_\partial <\tau_L \wedge \tau_\rho$ then the random variable in the right hand side of
\eqref{eq:key_eq_supp} is zero. We will consider separately the two cases where $\tau_L$ is smallest and also $\tau_\rho$ is smallest below in Steps 2 and 3 respectively, but in step $1$, we establish some geometric estimates on $L_S$ and $\rho$.

\textbf{Step 1.}  First we prove
that the distance between $L_S$ and $\rho$ is at least $d/10$ for small enough $S'-S$.
From the choice of
$d$,
we can draw an arc $I_\ell$ to the left of the leftmost point in $L_S \cap
\partial D$ and $I_r$ to the right of the rightmost point in $L_S \cap \partial
D$ both of length $d/4$ (and hence they do not intersect $e^{iW_S}$).
Let $b \in L_S \cap \partial \D$ and let $a',c$ be the two extremities of the
arcs $I_\ell, I_r$ which are farthest from $L_S$.
Using the radial Loewner equation applied to $h_t = g_{S + t} \circ g_S^{-1}$, at $z=b$ for now, we see
that
\begin{equation}\label{bound_der_b_supp}
|\partial_t h_t(b)| = |h_t(b) \frac{e^{iW_{S+t}}+h_t(b)}{e^{iW_{S+t}}-h_t(b)}| \le
\frac{2}{d}
\end{equation}
for all $t \le S'-S$ (this is because over this interval of time, we must have $d(h_t(b) , e^{iW_{t+S}}) \ge d$ and hence the denominator in the right hand side is greater than $d$).
Applying the same Loewner equation at $z=a'$ and $z=c$, so long as $d( h_t(a'), e^{iW_{t+S}}) \ge d/2$, and so long as that $t\le S'-S$, we get
$|\partial_t h_t(a')| \le 4/d.$ Hence $|\partial_t h_t(a') - \partial_t h_t(b)| \le 6/d$ on that interval.
Combining with the fact that $h_t(b)$ is at least $d$ away from $e^{iW_{t+S}}$ and \cref{bound_der_b_supp}, we deduce that the time it would take for $h_t(a')$ to be less than $d/2$ away from $e^{i W_{t+S}}$ is at least $d^2/24 $. Hence if
$S'-S \le d^2/C$ for some sufficiently large constant $C>24$, then the condition $d( h_t(a'), e^{iW_{t+S}}) \ge d/2$ is in fact always fulfilled throughout $t \in [0, S'-S]$. Observe that as $S'-S \to 0$, $d \to d(L_S,e^{iW_S} )>0$ and hence we can always make such a choice. Consequently, we deduce that a bound similar to \cref{bound_der_b_supp} holds at $z=a'$ and $z=c$ with a different constant: for $t \le S'-S$,
\begin{equation}\label{bound_der_ac_supp}
|\partial_t h_t(z)| = |h_t(z) \frac{e^{iW_{S+t}}+h_t(z)}{e^{iW_{S+t}}-h_t(z)}| \le
\frac{4}{d}, \ \ z=a',c.
\end{equation}
Applying the above bound for the extremities of both the arcs $I_\ell,I_r	$ and integrating this bound over $[0, S'-S]$ we see that $h_{S'-S}(I_\ell)$ and $h_{S'-S}(I_r)$ are arcs whose length is $(d/4) (1+ O( 1/ C))$. Hence this can be made arbitrarily close to $d/4$. Note that the lengths of $h_{S'-S}(I_\ell)$ and $h_{S'-S}(I_r)$ are nothing else but the harmonic measures seen from 0 in $\D\setminus \rho$ of $I_\ell, I_r$ respectively.

\begin{figure}
\begin{center}
\includegraphics{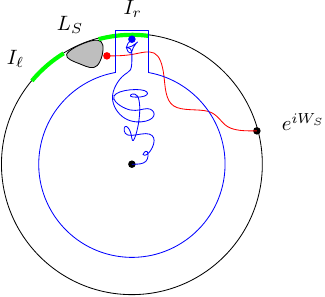}
\end{center}
\caption{Proof of step 1.}
\label{F:Lfar_supp}
\end{figure}

From this we can conclude that $\rho$ does not come within distance $d/10$ of $L_S$. Indeed, if it did, the harmonic measures of $I_r$ or $I_\ell$ would be small. More precisely, there is a universal constant $c>0$ such that conditionally on $B_{\tau_\partial} \in I_r$ say, $B$ crosses any set coming within distance $d/10$ of  $L_S$ with probability at least $c$: for instance this necessarily happens if $B$ stays in a certain deterministic set consisting of a union of a rectangle of dimensions $d\times (d/20) $ and a disc of radius $1- d/2$ (see blue set in the accompanying Figure \ref{F:Lfar_supp}). This would imply that the harmonic measure of $I_\ell$ or $I_r$ in $\D \setminus \rho$ is at most $d/4 -c$ which is a contradiction to the fact that they can be made arbitrarily close to $d/4$ by choosing $S'-S$ small enough.
Note also, for future reference, that $h_{S'-S}(I_\ell)$ and $h_{S'-S}(I_r)$ are arcs of length at least $d/8$ and which don't intersect $h(\rho)$ (since $\rho$ does not hit $\partial \D$).

\medskip \textbf{Step 2.}
Now we estimate the right hand side of \eqref{eq:key_eq_supp}, which was
$$\E^0\left[ -\log \abs{B_{\tau_\partial
\wedge \tau_L \wedge \tau_\rho}} +\log \abs{B_{\tau_\partial \wedge \tau_\rho}}
+\log \abs{B_{\tau_\partial \wedge \tau_L}} \right],
$$
in the case where the Brownian motion hits $\rho$ before $L$
or $\partial \D$. In this case the first and second terms in the random variable above cancel each other and we need to estimate $\E^0(1_{\{\tau_\rho < \tau_\partial \wedge \tau_L\}}\E^{B_{\tau_\rho}}(\log
\abs{B_{\tau_\partial \wedge \tau_L}})) $. Let $\cC$ denote the circle of radius $d/10$ centered at $B_{\tau_\rho}$.
Notice that by Step 1, since $\rho$ does not come within distance $d/10$ of $L_S$,
\begin{equation}
\abs{ \E^{B_{\tau_\rho}}(\log
\abs{B_{\tau_\partial \wedge \tau_L}})) } \le
\abs{\log(1-\diam(L_S)) }\P^{B_{\tau_\rho}}(\tau_\cC < \tau_\partial)   \label{eq:long_proof2_supp}
\end{equation}
where $\tau_\cC$ is the stopping time when the Brownian motion hits $\cC$.
Now, observe that
\begin{equation}\label{diamLS_supp}
\diam( L_{S}) \le C\  \harm_{\D\setminus L_{S}} (0;L_{S}) \le C e^{- C' (S-I)}
\end{equation}
by Beurling's estimate, where $\harm_D (z, \cdot)$ denotes harmonic measure in $D$ seen from $z$.
Hence
$$
 \E^{B_{\tau_\rho}}(\log
\abs{B_{\tau_\partial \wedge \tau_L}})) \le
Ce^{-c'(S-I)}\P^{B_{\tau_\rho}}(\tau_\cC < \tau_\partial)
$$
using $|\log(1-x)| =O(x)$. Now it remains to bound $\P^{B_{\tau_\rho}}(\tau_\cC < \tau_\partial)$ from above. Let $z = B_{\tau_\rho}$. Then it is elementary to check that we have $\P^z(\tau_\cC < \tau_\partial) \le c (1- |z|) /d$ for some constant $c>0$ (this is the probability of reaching distance $d/10$ within a half-plane from $i(1-|z|)$: this can be seen by applying the map $z \mapsto z^2$ and Beurling's estimate).
Plugging this
back into \eqref{eq:long_proof2_supp}, we
conclude:
$$
\abs{ \E^{B_{\tau_\rho}}(\log
\abs{B_{\tau_\partial \wedge \tau_L}})) } \le Ce^{-c'(S-I)} \frac{1- |B_{\tau_\rho}|}{d} \le  \frac{Ce^{-c'(S-I)} }{d} \abs{\log |B_{\tau_\rho}|}.
$$
Taking expectations with respect to $\E^0$ on the event $\tau_{\rho}< \tau_\partial \wedge \tau_L$,
we deduce that
\begin{multline*}
\abs{\E^0 [ 1_{\tau_{\rho}< \tau_\partial \wedge \tau_L} \E^{B_{\tau_\rho}}(\log
\abs{B_{\tau_\partial \wedge \tau_L}}))]  }\le  \frac{Ce^{-c'(S-I)} }{d}   \E^0  |\log |B_{\tau_\rho}| |\\
= \frac{Ce^{-c'(S-I)} }{d} (S'-S),
\end{multline*}
which is slightly better than the bound $ ({Ce^{-c'(S-I)} }/{d^2} )(S'-S)$ we are aiming for.

\medskip \textbf{Step 3.}
Finally we consider the case where $\tau_L< \tau_\partial \wedge \tau_\rho$, which is the most delicate. If the Brownian motion hits $L$ first then the first and third terms cancel each other and we need to estimate
 $\E^0(\E^{B_{\tau_L}}(\log
\abs{B_{\tau_\partial \wedge \tau_\rho}}) 1_{\tau_L < \tau_\partial \wedge \tau_\rho})$ using the strong Markov property.
We now claim that, almost surely,
\begin{equation}
 \Big |\frac{\E^{B_{\tau_L}}(\log \abs{B_{\tau_\partial
\wedge \tau_\rho}})}{ \E^{0}(\log \abs{B_{\tau_\partial
\wedge \tau_\rho}})   }\Big | \le
\frac{C}{d^2}e^{-c'(S-I)}\label{eq:long_proof!_supp}
\end{equation}
To do this we are going use
conformal invariance of harmonic measure and map out $\rho$ via $h_{S'-S}$.
Then we can use an explicit bound on the Poisson kernel on the disc which
allows us to compare the harmonic measure of a set seen from $0$ and from a
point in $L_S$. By conformal invariance of harmonic measure, letting $h = h_{S'-S}$
\begin{align*}
\E^{B_{\tau_L}} (- \log | B_{\tau_\partial \wedge \tau_\rho}|)  & = - \int_\rho \log |z| \ \harm_{\D \setminus \rho}(B_{\tau_L}, dz)  \\
& =- \int_{h (\rho)}  \log |h^{-1} (z)|   \harm_\D ( h(B_{\tau_L}) , dz)  \\
& = -  \int_{h (\rho)}  \log |h^{-1} (z)|   \frac{d \harm_\D  ( h(B_{\tau_L}) , dz) }{d \harm_\D ( 0 , dz)  } \harm_\D ( 0 , dz)
\end{align*}
Observe that if we write $h(B_{\tau_L}) = r e^{i \theta} \in L_{S'}$, then the Radon-Nikodym derivative above is simply the Poisson kernel and at the point $z = e^{it}$ it is thus equal to $P_r(\theta - t)$ where
\[
 P_r(\theta) = \frac{1-r^2}{1+r^2 - 2r\cos \theta}
\]
Note that when $e^{it} \in h(\rho)$ then $|\theta - t | \ge d/100$ by Step 1. Indeed, we know that $h(I_\ell)$ and $h(I_r) $ are arcs of length at least $d/8$ which do not intersect $h(\rho)$ (see the end remark of Step 1) and we know that the diameter of $L_{S'}$ is at most $C e^{- C' (S-I)} \le d /10^3$ by assumption for a choice of constants $A>C,a<C'$ and large enough $S-I$. Hence the difference in arguments for a point in $L_{S'}$ and a point in $h(\rho)$ is at least $d/100$.
Thus we have $\cos(\theta- t ) \le 1- d^2/10^5$. Hence
$$
P_{r} (\theta-t) \le \frac{1- r^2}{(1-r)^2 + 2r d^2/10^5} \le \frac{2C e^{- C' (S-I)}}{d^2}
$$
by our assumptions on $d$, the above choice of constants $a,A$ and the fact that
$$
1- r\le \diam( L_{S'}) \le  C e^{- C' (S'-I)} \le C e^{- C'(S-I)}
$$
by \cref{diamLS_supp}. Since $ \E^{0}(\log \abs{B_{\tau_\partial
\wedge \tau_\rho}} ) = S'-S$, putting everything together we obtain \cref{lemma:derivative_capacity_supp}. \qed

Taking the increment $S'-S$ to $0$ we get the following corollary:
\begin{lemma}\label{lemma:RNderivative_supp}
Let $a,A$ be as in \cref{lemma:derivative_capacity_supp}.
 Let $$d = \inf_{S \leq T
\leq S + 1/10 } d(L_T, e^{iW_T}).$$
Let $w$ and $w'$ be the two points on $\gamma$ corresponding respectively to $S$ and $S + 1/10$.
Let $\mu$ and $\bar \mu$ be the measures on $\gamma$ obtained by capacity in $D$ and $B (z_0, e^{-I})$ respectively: that is, if $z = z_T$ and $z' = z_{T'}$ then $\mu (\gamma (z, z')) = |T'-T|$, and likewise for $\bar \mu$.   If $S-I$ is large enough  and if $d \geq A e^{-\frac{a}{4}(S-I)}$, then, uniformly over $\gamma(w, w')$, we have
\[
  1- A e^{-a(S-I)}/d^2  \leq \frac{d \bar \mu}{d \mu} \leq 1 +  Ae^{-a(S-I)}/d^2.
\]
\end{lemma}

Now  we check that the assumption on $d$ in the previous lemma holds when the curve is SLE$_2$ with high probability.

\begin{lemma}\label{L:sledist_supp}
Suppose $\gamma$ is a radial SLE$_2$ process independent of $I$. Define $d$ for $\gamma$ as in \cref{lemma:RNderivative_supp}. For any $a,A>0$ we have
\begin{equation}\label{sledist_supp}
d\ge  A e^{-a(S-I) }\text{ with probability }\ge 1 - c\exp ( - c' ( S-I)).
\end{equation}
for some $c,c'>0$ depending only on $A,a$.
Now suppose $\gamma$ is a loop-erased random walk in $D^\d$  from $z_0$ to $\partial D$  and $I$ as defined in Theorem 4.21 in the main file (applied for a single point). We emphasise that $I$ might not be independent of $\gamma$.
Let $d$ be as above.\ Then there exists $\delta = \delta(S)$ such that conditioned on $I$, if $S-I$ is large enough and $\delta \le \delta(S)$, for any $a,A>0$, we have
\begin{equation}\label{lerwdist_supp}
d\ge  A e^{-a(S-I) }\text{ with probability }\ge 1 - \frac{c}{2}\exp ( - c' ( S-I)).
\end{equation}
\end{lemma}

\begin{proof}
Let $T$ be the smallest $t$ such that $\gamma(z(t)) \in B(z_0, e^{-I})$, and let $e^{i \theta_t}$ be one of the two images $ g_t   \circ \ph (z)$ for $T \le t \le S$ where $z = z(T)$. Then note that $e^{i \theta_S} \in \partial \D \cap L_S$. Moreover we claim that with high probability, if $S-I$ is large, then $e^{i \theta_S}$ is far away from $e^{i W_S}$.

We can follow the evolution of $e^{i \theta_t} $ under the Loewner flow $g_{t}$ for $T \le t \le S$ as follows. If  $Y_t = \theta_t - W_t$ then $Y$ solves a stochastic differential equation:
$$
dY_t = \cot(Y_t/2) - dW_t \ ; \  Y_T = 0.
$$
(see (6.13) in \cite{Lawler}). By comparing with a Bessel process (of dimension 5 in the case $\kappa = 2$), we can see that the probability for $Y_t$ to hit $[0, \delta]$ in any particular interval of length $O(1)$ is polynomial in $\delta$.
Thus for any given $S$ such that $S- I$ is large enough, then for any $c,C$,
$$
\P\left( \min_{S \leq t \leq S +  1/10} Y_t \leq 10 e^{-c(S-I)} \right) \le C'\exp ( - c' ( S-I))
$$
is exponentially small in $S-I$. Now, since $\diam(L_S) \le C e^{- c(S-I)}$ by Beurling's estimate (see \cref{diamLS_supp}) we see that
$$
d\ge  C e^{-C(S-I) }\text{ with probability }\ge 1 - C'\exp ( - c' ( S-I)).
$$
This proves \cref{sledist_supp}.

To deduce \cref{lerwdist_supp}, we recall that loop-erased random walk converges to SLE$_2$ under our assumptions by \cite{YY}. Then use Remark 4.24 in the main file to note that conditionally on $I=i$, the law of $\gamma$ after the first time it hits $B(z_0, e^{-I})$ is absolutely continuous with respect to the unconditional law $(\gamma_t, t\ge T_i)$ (where $T_i$ is the first time the path enters $B(z_0, e^{-i})$). Also the derivative of the conditional law with respect to the unconditional law is bounded by $C>0$.
Since the distance $d$ is a continuous function (in the mesh size $\delta$) of $\gamma$ the estimate \cref{sledist_supp} applies with $\kappa =2$ for all $\delta$ sufficiently small and $S-I$ large enough.
\end{proof}

Now we can finally prove our coupling estimate. Recall the full coupling $(\gamma, \tilde \gamma)$ (Theorem 4.21 with a single point), where $\gamma$ is a loop-erased random walk in $D^\d$ and $\tilde \gamma$ is a loop-erased random walk in $\tilde D^\d$ starting from a vertex $v$ where $\tilde D$ is arbitrary. It is recommended to think of $\tilde D$ as the full plane.

\begin{lemma}\label{lemma:coupling_Poisson_supp}
There exists a universal constant $c$ such that the following holds. For any $t>0$ there exists $\delta = \delta(t)$ such that for any $\delta \in (0, \delta(t))$,
we can find a pair of random variables $(X,\tilde X)$ such that individually, $X$ and $\tilde X$ are each independent of $(\gamma, \tilde\gamma )$
and
\[
\P[\gamma(t + X) = \tilde \gamma(t +  \tilde X)] \geq 1- Ce^{-ct }.
\]
Furthermore, both $X$ and $\tilde X$ are random variables which are bounded (by $1/20$).
\end{lemma}
In this proof, we will need to choose constants depending on each other so we number them.
\begin{proof}
Let $\mu$ be the measure on $\gamma$ defined as in \cref{lemma:RNderivative_supp} and let $\tilde \mu$ be the equivalent measure for the curve $\tilde \gamma$. Note that \cref{lemma:RNderivative_supp} allows us to compare $\mu$ to $\bar\mu$ and also $\tilde \mu$ to $\bar \mu$ since $\gamma$ and $\tilde \gamma$ are assumed to coincide within $B(z_0, e^{-I})$. Consequently, we will have a way of comparing $\mu $ and $\tilde \mu$ and show that the Radon-Nikodym derivative of one with respect to the other is very close to one, from which the result will follow.

Here are the details. Let $z=z(t)$ be the point of capacity $t$ seen from $v$ in $D$ for $\gamma$. Let $\tilde T(z)$ be the capacity of $\tilde \gamma$ at $z$ seen from $v$ in $\tilde D$. Let $\cA$ be the event that $I \le t/2$, $|\tilde T(z(t)) - t|  \le e^{-c_1t}$ and  both $\gamma, \tilde \gamma$ do not exit $B(z_0,e^{-t/2})$ after capacity $t$. Recall that $I$ has exponential tail (Theorem 4.21 in the main file) and for small enough $\delta = \delta(t)$, we see from \cref{lemma:absolute_difference_supp} and Schramm's estimate (Theorem 3.1 in the main file) that both $\gamma, \tilde \gamma$ do not exit $B(z_0,e^{-I})$ after capacity $t$ with exponentially high probability for large enough $t$. Thus overall, the probability of $\cA$ is at least $1-e^{-c_2t}$ for some $c_2>0$ for small enough $\delta(t)$.

Let $d,a,A$ be as in \cref{lemma:RNderivative_supp} for $\gamma,D$ and let $\tilde d,\tilde a, \tilde A$ be the equivalent quantities for $\tilde \gamma, \tilde D$. By \cref{lerwdist_supp}, we observe that the event $\cB:=\{d > Ae^{-a/4t}, \tilde d> \tilde Ae^{-\tilde a/4t}\}$ has probability at least $1-e^{-c_3t}$ for some $c_3>0$.

Applying \cref{lemma:RNderivative_supp}, we find that for $t$ sufficiently large on $\cA \cap \cB$,

\begin{equation}\label{RN_supp}
 1- e^{-c_4 t} \leq \frac{d \mu}{d \tilde \mu} \leq 1+ e^{-c_4 t}
\end{equation}
on a subset of the path including $\Gamma := \tilde \gamma( t + e^{-c_1 t}, t  - e^{-c_1 t}  + \frac{1}{20})$ for some small enough $\delta(t)$. 
We let $\nu = (1- e^{-c_4 t} ) 1_{\Gamma}\tilde \mu$.
Now we choose $0< c_5 < c_4$ and we define three independent Poisson processes on $\gamma$, say $P_1, P_2, P_3$, with densities respectively $e^{c_5 t} \nu$, $e^{c_5 t}(\tilde \mu - \nu)$ and $e^{c_5 t}(\mu - \nu)$.
Note also that the processes $P_1$ and $P_2$ depend only on $\tilde \gamma$.

Let $\tilde z$ be the first point in $P_1 \cup P_2$ after capacity $t$ in $\tilde D$ and let $z$ be the first point in $P_1 \cup P_3$ after capacity $t$ in $D$. Let $X = T(z) - t$ and $\tilde X = T(\tilde z) - t$; if there are no points in $P_1 \cup P_3$ or $P_1 \cup P_2$ we abort the coupling (this has probability smaller than $e^{-\exp(c_5t/30)}$). Note that $X, \tilde X \le 1/20$ by construction and they are each independent of $(\gamma, \tilde \gamma)$ (since the conditional marginal law of both $X$ and $\tilde X$ is exponential; however, note that the pair $(X, \tilde X)$ is \emph{not} independent from $(\gamma, \tilde \gamma)$).
Now observe that with very high probability both $z, \tilde z \in P_1$ and therefore are equal. Indeed, $\P( \tilde z \not \in P_1) \leq e^{(c_5-c_4)t} $ and similarly  $\P( z \notin P_1) \le e^{(c_5-c_4) t}$ using \cref{RN_supp} which concludes the proof.
\end{proof}
\section{General domains}\label{app:generaldomain_supp}

In this section we state the following theorem which is the main result of Section 3.5 of the main file.
 Recall that our definition of $u_{(D,x)}$ from the main file.

\begin{thm}\label{L:change_coord_supp}
Let $D$ be as above.
Let $f$ be any bounded Borel test function defined on $\bar D$. Let $ h = \hg^0 + \chi
u_{(D,x)}$
 be the GFF coupled to the UST according to the imaginary geometry coupling of
\cref{thm:coupling_intro_supp} and $u_{(D,x)}$ is as in the main file. Then
 $( h_t^D ,f) $ converges to $(\hg^D,f)$ in $L^2(\P)$ and in
probability as $t\to \infty$, where $\hg^D = \chi^{-1}h+\pi/2$.
\end{thm}
\begin{proof}
Let $h^D_t$ be the winding field in $D$ as defined as in (3.1) or (3.2) from the main file as appropriate. Let $\cT^D$ denote the continuum wired UST in $D$. Fix a conformal map $\ph:\D\to D$
sending the marked point $1$ to the marked point $x \in \partial D$ (note that since $\partial D$ is locally connected, any conformal map from $\D$ to $D$ extends continuously to $\partial D$, see \cite{Pommerenke}). Let $h^\D$
denote the intrinsic winding field associated to the continuum spanning tree
$\cT = \ph^{-1} (\cT^D)$.

Let $\tilde \gamma_w(t)$ denote the branch of the UST towards $w$ in $\cT^D$. Let $\cA^D(t,w)$ be the event that $|\tilde \gamma_w(t) -w| < e^{-t/2}R(w,D)$. Applying Koebe's $1/4$ theorem, we conclude using Theorem 2.11 from the main file that $\P(\cA^D(t,w)) \ge 1- e^{-ct}$.

First note that in the case of a smooth domain, using Lemma 2.4 from the main file (we can apply this since we are dealing with conformal images of continuous curves in $\D$ and we can also extend $\arg \ph'$ to $1$ using \cref{lem:geometric_condition_Supp}), we conclude that (writing $w = \ph(z))$,
\begin{equation}
h_t^D\circ \ph(z) 1_{\cA^D(t,w)} = \Big(h_t^\D(z)  + \arg_{\ph'(\D)} (\ph'(z)) + \epsilon(z) \Big)1_{\cA^D(t,w)}\label{eq:good_moment_supp}
\end{equation}
where $|\epsilon(z)| < e^{-ct}$. In the case of general boundary, using Lemma 2.4 from the main file, we get the same equation \eqref{eq:good_moment_supp} up to a global constant.
Using this and (3.11) from the main file, we immediately conclude that $$\E((h_t^D\circ \ph(z) - \arg_{\ph'(\D)} (\ph'(z)))^2 1_{\cA^D(t,w)}) \le C(1+t).$$
We now claim that the same estimate also holds on the complement of $\cA^D(t,w)$, which will follow from Lemma 2.10 of the main file.

Indeed, note that even when $\cA^D(t,w)$ does not hold, the curve $\tilde \gamma_w[-1,t]$ must come within distance $e^{-t/2}R(w,D)$ of $w$. So by Koebe's $1/4$ theorem there is a (random) time $t'$ satisfying $t/2 - \log 4<t'<t/2$ so that $|\tilde \gamma_w(t') -w| <e^{-t/2} R(w, D)$. We then simply apply Lemma 2.4 of the main file at $t'$ instead of $t$. More precisely, we write
\begin{equation}
h_t^D \circ \ph(z) = h_{t'}^D \circ \ph(z) + W(\tilde \gamma_w[t',t];w).
\end{equation}
Using \cref{eq:good_moment_supp} and applying Lemma 2.10 of the main file to $h_{t'}^\D$ in $\D$ we see that the second moment of $h_{t'}^D \circ \ph(z) - \arg_{\ph'(\D)} (\ph'(z))$ is bounded by $C(1+t)^2$. Moreover the second moment of $W(\tilde \gamma_w[t',t];w)$ is bounded by $C(1+ t-t')^2$ using Lemma 2.10 of the main file again, but this time in $D$.

 Thus overall, we get
\begin{equation}
\E\Big ((h_t^D\circ \ph(z) -\arg_{\ph'(\D)} (\ph'(z)) )^2\Big) \le C(1+t) \label{eq:good_moment1_supp}
\end{equation}
Also in the general case, \eqref{eq:good_moment1_supp} is valid up to a global constant (meaning that for every choice of constant in the left hand side, the inequality holds for an appropriate choice of $C$ in the right hand side).
Now note that
\begin{equation*}
\E(h_t^D \circ \ph - \hg^D \circ \ph , f \circ \ph)^2 \le 2\E(h_t^D \circ \ph - h_t^\D - \arg_{\ph'(\D)} \ph' ,f\circ \ph )^2 + 2\E(h_t^\D - \hg^\D ,f \circ \ph )^2
\end{equation*}
The second term on the right hand side above converges to $0$ via Theorem 3.13 of the main file. The first term can be written as
\begin{align*}
&\E(h_t^D \circ \ph - h_t^\D - \arg_{\ph'(\D)} \ph' ,f \circ \ph)^2  = \int_{\D^2} dzdy f \circ \ph(z)f \circ \ph (y) \\
&\times \E (h_t^D \circ \ph(z) - h_t^\D(z) - \arg_{\ph'(\D)} \ph'(z))
  (h_t^D \circ \ph(y) - h_t^\D(y) - \arg_{\ph'(\D)} \ph'(y))\\
 & \le \int_{\D^2} dzdy f \circ \ph(z)f \circ \ph (y)\big[\E (\epsilon(z) \epsilon(y) 1_{\cA^D(t, \ph(z)) \cap \cA^D(t, \ph(y))} ) + c(1+t)e^{-ct}\big]
\end{align*}
where $|\epsilon(z)| \le e^{-ct}$ by \eqref{eq:good_moment_supp}. The result follows.
\end{proof}

\begin{remark}
We point out that the integral of $\arg_{\ph'(\D)}(\ph'(\cdot))$ might be infinite even in the smooth case. Also for the joint convergence of the moments, we do not need the domain to be bounded.
\end{remark}



\bibliographystyle{abbrv}
\bibliography{winding.bib}

\end{document}

%% file: proof_ht_revision.pdf_tex
\begingroup%
  \makeatletter%
  \providecommand\color[2][]{%
    \errmessage{(Inkscape) Color is used for the text in Inkscape, but the package 'color.sty' is not loaded}%
    \renewcommand\color[2][]{}%
  }%
  \providecommand\transparent[1]{%
    \errmessage{(Inkscape) Transparency is used (non-zero) for the text in Inkscape, but the package 'transparent.sty' is not loaded}%
    \renewcommand\transparent[1]{}%
  }%
  \providecommand\rotatebox[2]{#2}%
  \ifx\svgwidth\undefined%
    \setlength{\unitlength}{530.4bp}%
    \ifx\svgscale\undefined%
      \relax%
    \else%
      \setlength{\unitlength}{\unitlength * \real{\svgscale}}%
    \fi%
  \else%
    \setlength{\unitlength}{\svgwidth}%
  \fi%
  \global\let\svgwidth\undefined%
  \global\let\svgscale\undefined%
  \makeatother%
  \begin{picture}(1,1.48265469)%
    \put(0,0){\includegraphics[width=\unitlength,page=1]{proof_ht_revision.pdf}}%
    \put(0.05280874,0.6861306){\color[rgb]{0,0,0}\makebox(0,0)[lb]{\smash{$C_{-2}$}}}%
    \put(0.14319393,1.43458663){\color[rgb]{0,0,0}\makebox(0,0)[lb]{\smash{$C_{-1}$}}}%
    \put(0.27524299,1.43936035){\color[rgb]{0,0,0}\makebox(0,0)[lb]{\smash{$C_{0}$}}}%
    \put(0.41635617,1.43593886){\color[rgb]{0,0,0}\makebox(0,0)[lb]{\smash{$C_{1}$}}}%
    \put(0.51119421,1.43413128){\color[rgb]{0,0,0}\makebox(0,0)[lb]{\smash{$C_{2}$}}}%
    \put(0.63523856,1.43418675){\color[rgb]{0,0,0}\makebox(0,0)[lb]{\smash{$C_{3}$}}}%
    \put(0.76226274,1.43635267){\color[rgb]{0,0,0}\makebox(0,0)[lb]{\smash{$C_{4}$}}}%
    \put(0.89625349,1.43442104){\color[rgb]{0,0,0}\makebox(0,0)[lb]{\smash{$C_{5}$}}}%
    \put(0,0){\includegraphics[width=\unitlength,page=2]{proof_ht_revision.pdf}}%
    \put(0.56562027,1.23089756){\color[rgb]{0,0,0}\makebox(0,0)[lb]{\smash{$k_{\text{max}}$}}}%
    \put(0,0){\includegraphics[width=\unitlength,page=3]{proof_ht_revision.pdf}}%
    \put(0.1185877,0.18536053){\color[rgb]{0,0,0}\makebox(0,0)[lb]{\smash{$\kappa_{-1}$}}}%
    \put(0.24284993,0.37615129){\color[rgb]{0,0,0}\makebox(0,0)[lb]{\smash{$\kappa_{0}$}}}%
    \put(0.35229389,0.68835204){\color[rgb]{0,0,0}\makebox(0,0)[lb]{\smash{$\kappa_1$}}}%
    \put(0.56390734,0.74876603){\color[rgb]{0,0,0}\makebox(0,0)[lb]{\smash{$\kappa_{2}$}}}%
    \put(0.69059188,0.93228904){\color[rgb]{0,0,0}\makebox(0,0)[lb]{\smash{$\kappa_{3}$}}}%
    \put(0.86199095,0.95248876){\color[rgb]{0,0,0}\makebox(0,0)[lb]{\smash{}}}%
    \put(0,0){\includegraphics[width=\unitlength,page=4]{proof_ht_revision.pdf}}%
  \end{picture}%
\endgroup%

%% file: winding_final_arxiv_combined_251118.bbl
\def\cprime{$'$}
\begin{thebibliography}{10}

\bibitem{ber2013diffusion}
N.~Berestycki.
\newblock Diffusion in planar {L}iouville quantum gravity.
\newblock {\em Annales de l'Institut Henri Poincar{\'e}, Probabilit{\'e}s et
  Statistiques}, 51(3):947--964, 2015.

\bibitem{NBnotes}
N.~Berestycki.
\newblock Introduction to the {G}aussian free field and {L}iouville quantum
  gravity.
\newblock {\em Lecture notes, available on author's webpage}, 2015.

\bibitem{BLRannex}
N.~Berestycki, B.~Laslier, and G.~Ray.
\newblock A note on dimers and {T}-graphs.
\newblock {\em https://arxiv.org/abs/1610.07994}, 2016.

\bibitem{BN_SLE_notes}
N.~Berestycki and J.~Norris.
\newblock Lectures on {Schramm--Loewner} {E}volution.
\newblock {\em Cambridge University}, 2014.

\bibitem{Billingsley}
P.~Billingsley.
\newblock {\em Convergence of probability measures}.
\newblock John Wiley \& Sons, 2013.

\bibitem{BufetovGorin}
A.~Bufetov and V.~Gorin.
\newblock Fluctuations of particle systems determined by {S}chur generating
  functions.
\newblock {\em arXiv preprint arXiv:1604.01110}, 2016.

\bibitem{ChelkakSmirnov}
D.~Chelkak and S.~Smirnov.
\newblock Discrete complex analysis on isoradial graphs.
\newblock {\em Advances in Mathematics}, 228(3):1590--1630, 2011.

\bibitem{CohnKenyonPropp}
H.~Cohn, R.~Kenyon, and J.~Propp.
\newblock A variational principle for domino tilings.
\newblock {\em Journal of the American Mathematical Society}, 14(2):297--346,
  2001.

\bibitem{Tiliere07}
B.~{de Tili{\`e}re}.
\newblock {Scaling limit of isoradial dimer models and the case of triangular
  quadri-tilings}.
\newblock {\em Annales de L'Institut Henri Poincare Section (B) Probability and
  Statistics}, 43:729--750, Nov. 2007.

\bibitem{Dub1}
J.~Dub{\'e}dat.
\newblock S{LE} and the free field: partition functions and couplings.
\newblock {\em J. Amer. Math. Soc.}, 22(4):995--1054, 2009.

\bibitem{FieldLawler}
L.~S. Field and G.~F. Lawler.
\newblock Reversed radial {SLE} and the {B}rownian loop measure.
\newblock {\em Journal of Statistical Physics}, 150(6):1030--1062, 2013.

\bibitem{GarbanBourbaki}
C.~Garban.
\newblock Quantum gravity and the {KPZ} formula (after
  {D}uplantier--{S}heffield).
\newblock {\em Bourbaki seminar}, (352):315--354, 2013.

\bibitem{LBM13}
C.~Garban, R.~Rhodes, and V.~Vargas.
\newblock Liouville {B}rownian motion.
\newblock {\em Ann. Probab.}, 44(4):3076--3110, 2016.

\bibitem{giuliani2017haldane}
A.~Giuliani, V.~Mastropietro, and F.~L. Toninelli.
\newblock Haldane relation for interacting dimers.
\newblock {\em Journal of Statistical Mechanics: Theory and Experiment},
  2017(3):034002, 2017.

\bibitem{InteractingDimers}
A.~Giuliani, V.~Mastropietro, and F.~L. Toninelli.
\newblock Height fluctuations in interacting dimers.
\newblock {\em Annales de l'Institut Henri Poincar{\'e}, Probabilit{\'e}s et
  Statistiques}, 53(1):98--168, 2017.

\bibitem{Arctic}
W.~Jockusch, J.~Propp, and P.~Shor.
\newblock Random domino tilings and the arctic circle theorem.
\newblock {\em arXiv preprint math/9801068}, 1998.

\bibitem{Kasteleyn}
P.~W. Kasteleyn.
\newblock The statistics of dimers on a lattice {I}. the number of dimer
  arrangements on a quadratic lattice.
\newblock {\em Physica}, 27:1209--1225, 1961.

\bibitem{KenyonSurvey}
R.~Kenyon.
\newblock Lectures on dimers. {IAS}/{P}ark {C}ity mathematical series, vol. 16:
  Statistical mechanics, {AMS}, 2009.
\newblock {\em arXiv preprint arXiv:0910.3129}.

\bibitem{Kenyon_ci}
R.~Kenyon.
\newblock Conformal invariance of domino tiling.
\newblock {\em Ann. Probab.}, 28(2):759--795, 2000.

\bibitem{KenyonGFF}
R.~Kenyon.
\newblock Dominos and the {G}aussian free field.
\newblock {\em Ann. Probab.}, 29(3):1128--1137, 2001.

\bibitem{Kenyon_iso}
R.~Kenyon.
\newblock The {L}aplacian and {D}irac operators on critical planar graphs.
\newblock {\em Inventiones mathematicae}, 150(2):409--439, 2002.

\bibitem{KenyonHex}
R.~Kenyon.
\newblock Height fluctuations in the honeycomb dimer model.
\newblock {\em Communications in Mathematical Physics}, 281(3):675--709, 2008.

\bibitem{SixVertexSLE}
R.~Kenyon, J.~Miller, S.~Sheffield, and D.~B. Wilson.
\newblock Six-vertex model and schramm-loewner evolution.
\newblock {\em Physical Review E}, 95(5):052146, 2017.

\bibitem{KenyonOkounkov}
R.~Kenyon and A.~Okounkov.
\newblock Planar dimers and {H}arnack curves.
\newblock {\em Duke Mathematical Journal}, 131(3):499--524, 2006.

\bibitem{KenyonOkounkovSheffield}
R.~Kenyon, A.~Okounkov, and S.~Sheffield.
\newblock Dimers and amoebae.
\newblock {\em Ann. Math}, pages 1019--1056, 2006.

\bibitem{KPWtemperley}
R.~W. Kenyon, J.~G. Propp, and D.~B. Wilson.
\newblock Trees and matchings.
\newblock {\em Electron. J. Combin.}, 7, 2000.

\bibitem{dimer_tree}
R.~W. Kenyon and S.~Sheffield.
\newblock Dimers, tilings and trees.
\newblock {\em J. Combin. Theory Ser. B}, 92(2):295--317, 2004.

\bibitem{LaslierCLT}
B.~Laslier.
\newblock Central limit theorem for {T}-graphs.
\newblock {\em arXiv preprint arXiv:1312.3177}, 2013.

\bibitem{Lawler}
G.~F. Lawler.
\newblock {\em Conformally invariant processes in the plane}.
\newblock Number 114. American Mathematical Soc., 2008.

\bibitem{Lawler_RW}
G.~F. Lawler.
\newblock {\em Intersections of random walks}.
\newblock Springer Science \& Business Media, 2013.

\bibitem{LSW}
G.~F. Lawler, O.~Schramm, and W.~Werner.
\newblock Conformal invariance of planar loop-erased random walks and uniform
  spanning trees.
\newblock {\em Ann. Probab.}, 32(1):939--995, 2004.

\bibitem{Li}
Z.~Li.
\newblock Conformal invariance of dimer heights on isoradial double graphs.
\newblock {\em Annales de l'€™Institut Henri Poincar{\'e} D}, 4(3):273--307,
  2017.

\bibitem{LP:book}
R.~Lyons and Y.~Peres.
\newblock {\em Probability on trees and networks}, volume~42.
\newblock Cambridge University Press, 2016.

\bibitem{IG1}
J.~Miller and S.~Sheffield.
\newblock Imaginary geometry {I}: interacting {SLE}s.
\newblock {\em Probab. Theory Related Fields}, 164(3-4):553--705, 2016.

\bibitem{IG4}
J.~Miller and S.~Sheffield.
\newblock Imaginary geometry {IV}: interior rays, whole-plane reversibility,
  and space-filling trees.
\newblock {\em Probab. Theory Related Fields}, 169(3-4):729--869, 2017.

\bibitem{pemantle}
R.~Pemantle.
\newblock Choosing a spanning tree for the integer lattice uniformly.
\newblock {\em Ann. Probab.}, 19(4):1559--1574, 10 1991.

\bibitem{Petrov}
L.~Petrov.
\newblock Asymptotics of uniformly random lozenge tilings of polygons.
  {G}aussian free field.
\newblock {\em Ann. Probab.}, 43(1):1--43, 2015.

\bibitem{Pommerenke}
C.~Pommerenke.
\newblock {\em Boundary behaviour of conformal maps}, volume 299 of {\em
  Grundlehren der mathematischen Wissenschaften}.
\newblock Springer-Verlag, 2013.

\bibitem{Romik}
D.~Romik.
\newblock Arctic circles, domino tilings and square {Y}oung tableaux.
\newblock {\em Ann. Probab.}, 40(2):611--647, 2012.

\bibitem{SLE}
O.~Schramm.
\newblock Scaling limits of loop-erased random walks and uniform spanning
  trees.
\newblock {\em Israel J. Math.}, 118:221--288, 2000.

\bibitem{SchrammSheffield}
O.~Schramm and S.~Sheffield.
\newblock A contour line of the continuum {G}aussian free field.
\newblock {\em Probab. Theory Related Fields}, 157(1-2):47--80, 2013.

\bibitem{GFFShe}
S.~Sheffield.
\newblock Gaussian free fields for mathematicians.
\newblock {\em Probab. Theory Related Fields}, 139(3-4):521--541, 2007.

\bibitem{zipper}
S.~Sheffield.
\newblock Conformal weldings of random surfaces: {SLE} and the quantum gravity
  zipper.
\newblock {\em Ann. Probab.}, 44(5):3474--3545, 2016.

\bibitem{TemperleyFisher}
H.~N.~V. Temperley and M.~E. Fisher.
\newblock Dimer problem in statistical mechanics--an exact result.
\newblock {\em Philosophical Magazine}, 6(68):1061--1063, 1961.

\bibitem{Thurston}
W.~P. Thurston.
\newblock Conway's tiling groups.
\newblock {\em Amer. Math. Monthly}, 97(8):757--773, 1990.

\bibitem{YY}
A.~Yadin and A.~Yehudayoff.
\newblock Loop-erased random walk and {P}oisson kernel on planar graphs.
\newblock {\em Ann. Probab.}, pages 1243--1285, 2011.

\end{thebibliography}
